\documentclass[11pt]{article}
\usepackage{enumerate}
\usepackage{amsthm,amsmath,amssymb}
\usepackage[T1]{fontenc} 
\usepackage[utf8]{inputenc}
\allowdisplaybreaks

\usepackage{tikz}
\usetikzlibrary{arrows,arrows.meta,calc}

\usepackage{color}
\definecolor{vert}{rgb}{0.0, 0.5, 0.0}
\definecolor{amber}{rgb}{1.0, 0.75, 0.0}

\usepackage{vmargin}
\setmarginsrb{3cm}{2cm}{3cm}{2cm}{0cm}{2cm}{0cm}{1cm}

\theoremstyle{plain}
\newtheorem{theorem}{Theorem}
\newtheorem{lemma}[theorem]{Lemma}
\newtheorem{corollary}[theorem]{Corollary}
\newtheorem{proposition}[theorem]{Proposition}

\theoremstyle{definition}
\newtheorem{definition}[theorem]{Definition}
\newtheorem{example}[theorem]{Example}
\newtheorem{runningexample}[theorem]{Running Example}

\newtheorem{remark}[theorem]{Remark} 

\newcommand{\N}{\mathbb{N}}
\newcommand{\K}{\mathbb{K}}
\newcommand{\Z}{\mathbb{Z}}
\newcommand{\Suff}{{\rm Suff}}
\newcommand{\rad}{\rm rad}
\newcommand{\lex}{\rm lex}
\DeclareMathOperator{\Card}{{\rm Card}}
\newcommand{\KA}{\K\langle \langle A \rangle \rangle}
\newcommand{\KAd}{\K\langle \langle \bA \rangle \rangle}
\newcommand{\rep}{\mathrm{rep}}
\newcommand{\val}{\mathrm{val}}
\newcommand{\Sh}{\mathrm{Sh}}
\newcommand{\ba}{\boldsymbol{a}}
\newcommand{\bb}{\boldsymbol{b}}
\newcommand{\bi}{\boldsymbol{i}}
\newcommand{\bk}{\boldsymbol{k}}
\newcommand{\bm}{\boldsymbol{m}}
\newcommand{\bn}{\boldsymbol{n}}
\newcommand{\bw}{\boldsymbol{w}}
\newcommand{\bp}{\boldsymbol{p}}
\newcommand{\bs}{\boldsymbol{s}}
\newcommand{\bu}{\boldsymbol{u}}
\newcommand{\bv}{\boldsymbol{v}}
\newcommand{\bx}{\boldsymbol{x}}
\newcommand{\by}{\boldsymbol{y}}
\newcommand{\bz}{\boldsymbol{0}}
\newcommand{\bA}{\boldsymbol{A}}

\newcommand{\bL}{\boldsymbol{L}}
\newcommand{\bS}{\boldsymbol{\mathcal{S}}}
\newcommand{\bdiese}{\boldsymbol{\#}}
\newcommand{\couple}[2]{\left[\begin{smallmatrix} #1 \\ #2 \end{smallmatrix}\right]}
\newcommand{\triple}[3]{\left[\begin{smallmatrix} #1 \\ #2 \\ #3 \end{smallmatrix}\right]}
\newcommand{\duple}[2]{\left[\begin{smallmatrix} #1 \vspace{-.2cm} \\\vdots\\ #2 \end{smallmatrix}\right]}
\newcommand{\dcart}[2]{#1\times \cdots\times #2}

\author{\'Emilie Charlier, Célia Cisternino and Manon Stipulanti\\
Department of Mathematics\\
University of Li\`ege\\
All\'ee de la D\'ecouverte 12\\
4000 Li\`ege, Belgium\\
{\tt \{echarlier,ccisternino,m.stipulanti\}@uliege.be}}

\title{Regular sequences and synchronized sequences in abstract numeration systems}
\date{\today} 

\begin{document}

\thispagestyle{empty}
\maketitle

\begin{abstract} 
The notion of $b$-regular sequences was generalized to abstract numeration systems by Maes and Rigo in 2002. Their definition is based on a notion of $\mathcal{S}$-kernel that extends that of $b$-kernel. However, this definition does not allow us to generalize all of the many characterizations of $b$-regular sequences. In this paper, we present an alternative definition of $\mathcal{S}$-kernel, and hence an alternative definition of $\mathcal{S}$-regular sequences, which enables us to use recognizable formal series in order to generalize most (if not all) known characterizations of $b$-regular sequences to abstract numeration systems. We then give two characterizations of $\mathcal{S}$-automatic sequences as particular $\mathcal{S}$-regular sequences. Next, we present a general method to obtain various families of $\mathcal{S}$-regular sequences by enumerating $\mathcal{S}$-recognizable properties of $\mathcal{S}$-automatic sequences. As an example of the many possible applications of this method, we show that, provided that addition is $\mathcal{S}$-recognizable, the factor complexity of an $\mathcal{S}$-automatic sequence defines an $\mathcal{S}$-regular sequence. In the last part of the paper, we study $\mathcal{S}$-synchronized sequences. Along the way, we prove that the formal series obtained as the composition of a synchronized relation and a recognizable series is recognizable. As a consequence, the composition of an $\mathcal{S}$-synchronized sequence and a $\mathcal{S}$-regular sequence is shown to be $\mathcal{S}$-regular. All our results are presented in an arbitrary dimension $d$ and for an arbitrary semiring $\K$.
\end{abstract} 

\bigskip
\hrule
\bigskip



\bigskip
\hrule
\bigskip

\section{Introduction}

Automatic sequences form a family of infinite words over finite alphabets that is widely studied in combinatorics on words since the pioneer work of Cobham~\cite{Cobham1969,Cobham1972}. A sequence $f$ is said to be automatic if it is generated by a deterministic finite automaton with output (DFAO) as follows: the $n$th term $f(n)$ is the output of the DFAO when the input is the representation of $n$ in a suitable numeration system. For integer base $b$ numeration systems, we talk about $b$-automatic sequences. The most famous example is the Thue-Morse sequence, which is $2$-automatic. See the handbook~\cite{AlloucheShallit2003} for an comprehensive presentation of the subject.

A well-known characterization of $b$-automatic sequences is given through the notion of $b$-kernel: a sequence $f$ is $b$-automatic if and only if its $b$-kernel, which is the set of sequences $\{n\mapsto f(b^e n + r) \mid e\in\N,\ r\in[\![0,b^e-1]\!]\}$, is finite. With the aim of generalizing $b$-automatic sequences to sequences taking infinitely many values, Allouche and Shallit~\cite{AlloucheShallit1992} introduced the notion of $(b,\K)$-regular sequences over a N\oe therian ring $\K$. They define a sequence to be $(b,\K)$-regular if the $\K$-module spanned by its $b$-kernel is finitely generated. Then, it can be shown that a sequence $f\colon\N\to \K$ is $b$-automatic if and only if it is $(b,\K)$-regular and takes only finitely many values.

Regularity can be defined in numeration systems generalizing integer bases. See for instance~\cite{AlloucheScheicherTichy2000,
Charlier2018,CharlierCisterninoStipulanti,CharlierRampersadShallit2012, LeroyRigoStipulanti2017,RigoMaes2002} for work on the matter. The most general setting in this area is given by the so-called abstract numeration systems \cite{LecomteRigo2001}. An abstract numeration system $\mathcal{S}$ is a triple $(L,A,<)$ where $L$ is an infinite regular language over a totally ordered alphabet $(A,<)$. The numeration language $L$ is ordered thanks to the radix (or genealogical) order induced by $<$. Any non-negative integer $n$ is then represented by the $n$-th word of the language. This family of numeration systems includes the usual integer bases and more generally, any positional numeration system having a regular numeration language. Maes and Rigo~\cite{RigoMaes2002} extended $(b,\K)$-regular sequences to abstract numeration systems $\mathcal{S}$ by generalizing the notion of $b$-kernel to that of $\mathcal{S}$-kernel.
However, this definition is not satisfying for it does not allow us to generalize most characterizations of $(b,\K)$-regular sequences.

In this paper, we present an alternative definition of regular sequences in abstract numeration systems. Following the lead of Berstel and Reutenauer~\cite{BerstelReutenauer2011}, we define a sequence $f$ over an arbitrary semiring $\K$ to be $(\mathcal{S},\K)$-regular if the formal series $\sum_{w  \in A^*} f(\val_\mathcal{S}(w))\,w$ is $\K$-recognizable. Next, we introduce an alternative definition of $\mathcal{S}$-kernel. This new framework enables us to generalize most known characterizations of $(b,\K)$-regular sequences to abstract numeration systems with a prefix-closed numeration language. In doing so, we unify the formal series and kernel's points of view. 

In a previous work~\cite{CharlierCisterninoStipulanti}, we already used the latter definition of regular sequences in the context of Pisot numeration systems, i.e., numeration systems based on an increasing sequence of integers that satisfies a linear recurrence over $\Z$ whose characteristic polynomial is the minimal polynomial of a Pisot number. Such systems generalize integer bases and can be seen as particular abstract numeration systems since it is well known that their numeration languages are regular~\cite{FrougnySolomyak1996}. We showed in~\cite{CharlierCisterninoStipulanti} that the notion of $(U,\K)$-regular sequences is robust, in the sense that it is independent of the algorithm used in order to represent the integers, provided that it still gives rise to a regular numeration language. The obtained characterization generalizes results of both Allouche, Scheicher and Tichy \cite[Section 5]{AlloucheScheicherTichy2000} and Berstel and Reutenauer \cite[Prop. 5.1.1]{BerstelReutenauer2011}. But precisely, this one characterization of $(U,\K)$-regular sequences \emph{cannot} be generalized to abstract numeration systems. Indeed, there is no such notion of robustness in the context of abstract numeration systems since the algorithm used for representing integers is given by the definition of the considered abstract numeration system itself. 

In a second part of the paper, we study the particular cases of $\mathcal{S}$-automatic sequences and $\mathcal{S}$-synchronized sequences in detail. We also generalize the technique initiated in~\cite{CharlierRampersadShallit2012} in order to build various examples of $(\mathcal{S},\K)$-regular sequences by enumerating suitable properties of $\mathcal{S}$-automatic sequences. We take care of presenting all our results in an arbitrary dimension $d$. Along the way, we illustrate most introduced notions and obtained results through a running example.

The paper is organized as follows.
Section~\ref{Sec : Preliminaries} is devoted to the necessary background on abstract numeration systems and formal series. In Section~\ref{Sec : Regular}, we fix a $d$-tuple $\bS=(\mathcal{S}_1,\ldots,\mathcal{S}_d)$ of abstract numeration systems and we define the notion of $(\bS,\K)$-regular sequences in terms of formal series. Then we study their closure properties and growth rate. In Section~\ref{Sec : WH}, we introduce a general working hypothesis. Namely, we ask that the numeration language $\bL$ is prefix-closed. This hypothesis will only be needed in Sections~\ref{Sec : Kernel} and~\ref{Sec : Automatic} in order to obtain an analogy between the formal series point of view and the sequences one. In Section~\ref{Sec : Kernel}, we define the $\bS$-kernel of a sequence $f\colon\N^d\to\K$. Then we prove three characterizations of $(\bS,\K)$-regular sequences, which are given by Theorems~\ref{The : Reg-sous-module},~\ref{The : PracticalCriterion} and~\ref{The : reg-noyau-fin-eng}. These three characterizations may be seen as analogues of the characterizations of $\K$-recognizable series given in~\cite{BerstelReutenauer2011}. Section~\ref{Sec : Automatic} is concerned with $\bS$-automatic sequences. First, Theorem~\ref{The : automatic-finite-kernel} states that a sequence $f\colon\N^d\to\K$ is $\bS$-automatic if and only if its $\bS$-kernel is finite. Next we establish in Theorem~\ref{The : Regular-Automatic} that the family of $\bS$-automatic sequences is included in that of $(\bS,\K)$-regular sequences. In Section~\ref{Sec : Enumeration} we present a general method in order to obtain $(\bS,\N)$-regular sequences by enumerating $\bS$-recognizable properties of $\bS$-automatic sequences. In particular, we provide an $\bS$-recognizable way to enumerate elements of $\N^d$. Section~\ref{Sec : Synchronized} is devoted to $(\bS,\bS')$-synchronized sequences $f\colon\N^d\to \N^{d'}$. We prove that if $d'=1$, the family of $(\bS,\mathcal{S}')$-synchronized sequences lies in between those of $\bS$-automatic sequences and $(\bS,\N)$-regular sequences. Then we provide various closure properties of $(\bS,\bS')$-synchronized sequences. In particular, thanks to the general notion of synchronized relations, we show that the composition of synchronized sequences is synchronized. This property does not hold for regular sequences in general. By relaxing the hypotheses on one of the sequences, we prove in Section~\ref{Sec : Reg-Synch} that the composition of a synchronized sequence and a regular one remains regular. The main tool to prove the latter result is a composition-like operation between a two-tape automaton and a weighted automaton, generalizing the operation defined in~\cite{CharlierCisterninoStipulanti}.

\section{Preliminaries}
\label{Sec : Preliminaries}

An interval of integers $\{i,\ldots,j\}$ is denoted $[\![i,j]\!]$. We make use of common notions in formal language theory, such as alphabet, letter, word, length of a word, language and usual definitions from automata theory~\cite{Lothaire2002}. In particular, we denote the empty word by $\varepsilon$ and for a finite word $w$, $|w|$ is its length and for each $j\in[\![1,|w|]\!]$, $w[j]$ is its $j$-th letter. We also make use of some classical algebraic structures such as semirings and modules.

\subsection{Abstract numeration systems}
\label{Sec : NumSys}

An \emph{abstract numeration system} is a triple $\mathcal{S}=(L,A,<)$ where $L$ is an infinite regular language over a totally ordered alphabet $(A,<)$. The words in $L$ are ordered with respect to the radix (or genealogic) order $<_{\rad}$ induced by the order $<$ on $A$: for $u,v\in A^*$, $u<_{\rad}v$ either if $|u|<|v|$, or if $|u|=|v|$ and $u$ is lexicographically less than $v$. The \emph{$\mathcal{S}$-representation function} $\rep_\mathcal{S} \colon \N \to L$ maps any non-negative integer $n$ onto the $n$th word in $L$ (note that we start indexing from $0$). The \emph{$\mathcal{S}$-value function} $\val_\mathcal{S} \colon L \to \N$ is the reciprocal function of $\rep_\mathcal{S}$.


\begin{runningexample}
We start a running example by considering the abstract numeration system $\mathcal{S}=(a^*b^*, a<b)$. Since the radix order on $a^*b^*$ gives 
\[
	\varepsilon <_{\rad} a <_{\rad} b <_{\rad} aa <_{\rad} ab <_{\rad} bb <_{\rad} aaa <_{\rad} aab <_{\rad} \cdots
\]
the $\mathcal{S}$-representation of $7$ is $\rep_\mathcal{S}(7)=aab$. For instance, $\val_\mathcal{S}(aaa)=6$ but the $\mathcal{S}$-value function is not defined on the word $ba$. It is easily seen that for all $m,n\in\N$, we have $\val_\mathcal{S}(a^m b^n)=\frac{(m+n)(m+n+1)}{2}+n$.
\end{runningexample}

In general, for any alphabet $A$ and any symbol $\$\notin A$, we let $A_\$=A\cup \{\$\}$ and $\tau_{A,\$}\colon A_\$^*\to A^*$ be the morphism that erases the letter $\$$ and leaves the other letters unchanged. If there is no ambiguity on the alphabet $A$, we simply use the notation $\tau_\$$. For a $d$-tuple
\[
	\duple{w_1}{w_d}
\]
in $\dcart{A_1^*}{A_d^*}$ where $A_1,\ldots,A_d$ are any alphabets, we set 
\[
	\duple{w_1}{w_d}^{\$}
	=\duple{\$^{\ell-|w_1|}w_1}{\$^{\ell-|w_d|}w_d}
\]
where $\$$ is a symbol not belonging to $A_1\cup\cdots\cup A_d$ and $\ell$ is the maximum of the lengths of the words $w_1,\ldots,w_d$. Thus, we have padded the shortest words with leading symbols $\#$ in order to obtain $d$ words of the same length. Then the obtained $d$-tuple can be seen as a word over the $d$-dimensional alphabet $\dcart{(A_1)_\$}{(A_d)_\$}$, that is, an element of $\big(\dcart{(A_1)_\$}{(A_d)_\$}\big)^*$. 

From now on, we fix a dimension $d$, a $d$-tuple $\bS=(\mathcal{S}_1,\ldots,\mathcal{S}_d)$ of abstract numeration systems $\mathcal{S}_1=(L_1,A_1,<_1),\ldots, \mathcal{S}_d=(L_d,A_d,<_d)$ and a padding symbol $\#$, which does not belong to $A_1\cup\cdots\cup A_d$. We set 
\[
	\bdiese=\duple{\#}{\#}
\]
and 
\[
	\bA=\big(\dcart{(A_1)_\#}{(A_d)_\#}\big)\setminus\{\bdiese\}.
\]
We also define 
\[
	\bL=(\dcart{L_1}{L_d})^\#.
\] 
Note that since $L_1,\ldots,L_d$ are regular languages, so is $\bL$. We now extend the definition of the maps $\rep_\mathcal{S}$ and $\val_{\mathcal{S}}$ as follows:
\[
	\rep_{\bS}\colon \N^d \to \bL,\ 
	\duple{n_1}{n_d}
	\mapsto 
	\duple{\rep_{\mathcal{S}_1}(n_1)}{\rep_{\mathcal{S}_d}(n_d)}^\#
\]
and
\[
	\val_{\bS} \colon \bL\to \N^d,\
	\duple{w_1}{w_d} \mapsto 
	\duple{\val_{\mathcal{S}_1}(\tau_{\#}(w_1))}{\val_{\mathcal{S}_d}(\tau_{\#}(w_d))}.
\]
The maps $\rep_{\bS}$ and $\val_{\bS}$ are reciprocal bijections between $\N^d$ and $\bL$. Therefore, $\bL$ is called the \emph{numeration language} of the $d$-dimensional abstract numeration system $\bS$. 

\begin{runningexample}
Consider the $2$-dimensional abstract numeration system $\bS=(\mathcal{S},\mathcal{S})$. We have $\bdiese=\couple{\#}{\#}$ and
\[
	\bA=\left\{\couple{\#}{a},\couple{\#}{b},\couple{a}{\#},\couple{a}{a},\couple{a}{b},\couple{b}{\#},\couple{b}{a},\couple{b}{b}
\right\}.
\] 
The numeration language $\bL$ is the set of words over $\bA$ whose components  both belong to $\#^*a^*b^*$. For instance, $\rep_{\bS}\couple{4}{9}=\couple{\#ab}{bbb}=\couple{\#}{b}\couple{a}{b}\couple{b}{b}$ and $\val_{\bS}\couple{aab}{\#\#a}=\couple{\val_{\mathcal{S}}(aab)}{\val_{\mathcal{S}}(a)}=\couple{7}{1}$. 
\end{runningexample}

A subset $X$ of $\N^d$ is \emph{$\bS$-recognizable} if the language $\rep_{\bS}(X)$ is regular. A DFA accepting $\rep_{\bS}(X)$ is said to \emph{recognize} $X$ (with respect to the abstract numeration system $\bS$).

\begin{runningexample}
The subset $X=\{\frac{n(n+1)}{2}\couple{1}{1}+\couple{0}{n}\colon n\in\N\}$ of $\N^2$ is $\bS$-recognizable since $\rep_{\bS}(X)=\couple{a}{b}^*$.
\end{runningexample}


\subsection{Formal series}
\label{Sec : Formal series}

Let $A$ be a finite alphabet and $\K$ be a semiring. A \emph{(formal) series} is a function $S\colon A^* \to \K$. The image under $S$ of a word $w$ over $A$ is denoted $(S, w)$ and is called the \emph{coefficient} of $w$ in $S$. A series $S\colon A^* \to \K$ is also written as $S = \sum_{w\in A^*} (S,w)\, w$. We let $\KA$ denote the set of formal series over $A$ with coefficients in $\K$. For $S,T\in\KA$ and $k\in\K$, we define
\[
	S+T=\sum_{w\in A^*} \big((S,w)+(T,w)\big)w
	\quad\text{and}\quad
	kS=\sum_{w\in A^*} k(S,w)\, w.
\]
These operations provide $\KA$ with a structure of $\K$-module. For $S \in\KA$ and $u\in A^*$, we define 
\[
	Su^{-1}=\sum_{w\in A^*} (S,wu)\, w.
\] 
A $\K$-submodule of $\KA$ is called \emph{stable} if for all $u\in A^*$, it is closed under the operation $\KA\to \KA,\ S \mapsto Su^{-1}$.

A series $S \colon A^* \to \K$ is \emph{$\K$-recognizable} if there exist $r \in\N_{\ge 1}$, a morphism of monoids $\mu \colon A^* \to \K^{r\times r}$ (with respect to concatenation of words and multiplication of matrices) and two matrices $\lambda \in \K^{1 \times r}$ and $\gamma \in \K^{r \times 1}$ such that for all $w\in A^*$, $(S, w) = \lambda \mu(w) \gamma$. The triple $(\lambda, \mu, \gamma)$ is called a \emph{linear representation} of $S$ of \emph{dimension} $r$.

The following results can be found in the textbook \cite{BerstelReutenauer2011}. Note that these authors use the left operation $S\mapsto u^{-1}S$. Both choices lead to the same notion of $\K$-recognizable series since the family of $\K$-recognizable series is closed under reversal.

\begin{theorem}
\label{The : Series-Rec-Submodule}
A series $S \colon A^* \to \K$ is $\K$-recognizable if and only if there exists a stable finitely generated $\K$-submodule of $\KA$ containing $S$. 
\end{theorem}

\begin{runningexample}
Consider the series $S\colon\bA^*\to\N,\ \couple{u}{v}\mapsto\max|\Suff(u)\cap \Suff(v)|$. Some coefficients of $S$ are given in Figure~\ref{Fig : CoeffS}.
\begin{figure}[htb]
\[
	\begin{array}{c|ccccc}
	\bw & \couple{\#ab}{aab} & \couple{aaaab}{\#aaab}   & \couple{aab}{bab}  & \couple{aa}{ab} & \couple{a\#a}{aba} \\[3pt]
	\hline
	(S,\bw) & 2  & 4    & 2  & 0 & 1\\
	\end{array}
\]
\caption{Some coefficients of the series $S\colon\bA^*\to\N$.}
\label{Fig : CoeffS}
\end{figure}

Let $T\colon\bA^*\to\N,\ \bw\mapsto 1$. 
For all $\ba\in\bA$, $T\ba^{-1}=T$ and
\[
	S\ba^{-1}
	=\begin{cases}
	S+T & \text{if } \ba\in \left\{\couple{a}{a},\couple{b}{b}\right\} \\
	0 & \text{otherwise.}
	\end{cases}
\]
Therefore, $\langle S,T\rangle_\N$ is a stable finitely generated $\N$-submodule of $\N\langle \langle A \rangle \rangle$ containing $S$. By Theorem~\ref{The : Series-Rec-Submodule}, the series $S$ is $\N$-recognizable. More precisely, following the proof of Theorem~\ref{The : Series-Rec-Submodule} given in \cite{BerstelReutenauer2011},  the $\N$-submodule $\langle S,T\rangle_\N$ gives rise to the following linear representation $(\lambda,\mu,\gamma)$ of $S$:
\begin{align}
\nonumber
&\lambda=\left(\begin{smallmatrix}
			0 & 1		
			\end{smallmatrix}\right),\
	\gamma=\left(\begin{smallmatrix}
			1 \\
			0
			\end{smallmatrix}\right),\
	\mu\couple{a}{a}=\mu\couple{b}{b}
	=\begin{pmatrix}
			1 & 0\\
			1 & 1		
	\end{pmatrix},\\
\label{Eq:LinearRep}
	&\text{and for }\ba\in \bA\setminus\left\{\couple{a}{a},\couple{b}{b}\right\},\
	\mu(\ba)=\begin{pmatrix}
			0 & 0\\
			0 & 1		
			\end{pmatrix}.
\end{align}
\end{runningexample}

Since for all $u\in A^*$, the operation $\KA\to \KA,\ S \mapsto Su^{-1}$ is linear and for all $u,v\in A^*$, $(Su^{-1})v^{-1}=S(vu)^{-1}$, the smallest stable $\K$-submodule of $\KA$ containing $S$ is $\langle \{Su^{-1} \colon u\in A^*\} \rangle_\K$. For an arbitrary semiring $\K$, the fact that a series $S$ is $\K$-recognizable does not imply that $\langle \{Su^{-1} \colon u\in A^*\} \rangle_\K$ is finitely generated. The following result provides us with some special cases where this property holds. 

\begin{theorem}
\label{The : Series-Rec-Smallest-Submodule}
Suppose that $\K$ is finite or is a commutative ring. A series $S \colon A^* \to \K$ is $\K$-recognizable if and only if the smallest stable $\K$-submodule of $\KA$ containing $S$ is finitely generated. 
\end{theorem}

\begin{theorem}
\label{The : FibresRegulieres}
Suppose that $\K$ is finite or is a commutative ring. If $S\colon A^*\to \K$ is a $\K$-recognizable series with finite image, then for all $k\in\K$, $S^{-1}(k)$ is a regular language.
\end{theorem}

The \emph{Hadamard product} of two series $S,T \in \KA$ is the series 
\[
	S \odot T
	=\sum_{w\in A^*} \big((S,w)(T,w)\big)\, w.
\] 
Note that for all $u\in A^*$, $(S \odot T)u^{-1}=
Su^{-1} \odot Tu^{-1}$. If $L$ is a language over $A$, then its \emph{characteristic series} is the formal series $\underline{L} = \sum_{w \in L} w$. In particular, $S \odot \underline{L} = \sum_{w\in L} (S,w) \, w$. Also note that for all $u\in A^*$, $\underline{L} u^{-1}=\underline{Lu^{-1}}$. 

\begin{proposition}
\label{Pro : HadamardReg}
If $S\colon A^*\to \K$ is a $\K$-recognizable series and $L\subseteq A^*$ is a regular language, then $S\odot\underline{L}$ is $\K$-recognizable.
\end{proposition}

\section{$(\bS,\K)$-Regular sequences and first properties}
\label{Sec : Regular}

From now on, we let $\K$ designate an arbitrary semiring.

\begin{definition}
A sequence $f\colon\N^d\to\K$ is called \emph{$(\bS,\K)$-regular} if the formal series 
\[
	S_f:=\sum_{\bw\in\bL} f(\val_{\bS}(\bw))\,\bw
\]
is $\K$-recognizable. 
\end{definition}

\begin{runningexample}
Consider the $2$-dimensional sequence
\[
	f \colon \N^2 \to \N,\ 
	\couple{m}{n} \mapsto 
	\max|\Suff(\rep_{\mathcal{S}}(m))\cap \Suff(\rep_{\mathcal{S}}(n))|.
\]
We have $S_f=S\odot \underline{\bL}$. Since $S$ is an $\N$-recognizable series and $\bL$ is a regular language, by Proposition~\ref{Pro : HadamardReg}, the series $S_f$ is $\N$-recognizable, and hence the sequence $f$ is $(\bS,\N)$-regular. An example of unidimensional $(\mathcal{S},\N)$-regular sequence is given by the identity function $\N\to\N,\ n\mapsto n$~\cite{Rigo2001}.
\end{runningexample}

\subsection{Left-right duality}

In order to represent vectors of integers, we chose the convention of padding the $\bS$-representations to the left. Therefore, unless we work in dimension $1$, the notion of $(\bS,\K)$-regular sequences defined above is not left-right symmetric. In fact, the choice of a left padding will influence all choices that will be made in order to characterize $(\bS,\K)$-regular sequences. In what follows, we will take care to emphasize all definitions that depend on this "left choice" and comment on the consequences.

\subsection{First properties}

We first present some closure properties of $(\bS,\K)$-regular sequences. 

The Hadamard product of two sequences $f\colon\N^d\to\K$ and $g\colon\N^d\to\K$ is the sequence $f \odot g \colon\N^d\to\K,\ \bn \mapsto f(\bn )g(\bn )$.

\begin{proposition}
Let $f,g\colon\N^d\to\K$ be two $(\bS,\K)$-regular sequences and let $k\in \K$. The sequences $f+g$, $kf$ and $f\odot g$ are $(\bS,\K)$-regular. 
\end{proposition}

\begin{proof}
We have $S_{f+g}=S_f+S_g$, $S_{kf}=kS_f$ and $S_{f \odot g} = S_f \odot S_g$, so the statement follows from closure properties of $\K$-recognizable series \cite[Chapter 1]{BerstelReutenauer2011}.
\end{proof}

The family of $\K$-recognizable sequences is closed under finite modifications.

\begin{proposition}
\label{Pro : FiniteModif}
Let $f\colon\N^d\to\K$ be a $\K$-recognizable sequence and let $g\colon\N^d\to\K$ be such that $g(\bn)=f(\bn)$ for all $\bn\in\N^d$ except a finite number of them. Then the sequence $g$ is $\K$-recognizable.
\end{proposition}

\begin{proof}
Let $F=\rep_{\bS}(\{\bn\in\N^d\colon g(\bn)\ne f(\bn)\})$. By assumption, $F$ is a finite language. Therefore, $S_g=S_f\odot \underline{\bA^*\setminus F}+\sum_{\bw\in F}g(\val_{\bS}(\bw))\bw$ is $\K$-recognizable by Proposition~\ref{Pro : HadamardReg} and since polynomials (i.e., series with finite support) are always $\K$-recognizable.
\end{proof}

Next we study the growth rate of $(\bS,\K)$-regular sequences.

\begin{proposition}
\label{Pro : GrowthRate}
Assume that $\K$ is equipped with an absolute value $|\cdot|_\K\colon \K\to\mathbb{R}_{\ge 0}$ and that there are a function $g\colon \N^d \to \N$ and $N\in\N$ such that for all $\bn\in\N^d$ satisfying  $\min\{n_1,\ldots,n_d\}\ge N$, we have $|\rep_{\bS}(\bn)|\le g(\bn)$. Then for any $(\bS,\K)$-regular sequence $f\colon\N^d\to\K$, there exists $c\in\mathbb{R}_{>0}$ such that $|f(\bn)|_\K\in O(c^{g(\bn)})$. 
\end{proposition}

\begin{proof}
Let $f\colon\N^d\to\K$ be a $(\bS,\K)$-regular sequence and let $(\lambda,\mu,\gamma)$ be a linear representation of $S_f$. Consider any submultiplicative matrix norm  $||\cdot||_\K$ induced by the absolute value $|\cdot|_\K$ on $\K$ and let $c=\max\{ ||\lambda||_\K, \max_{\ba\in \bA} ||\mu(\ba)||_\K, ||\gamma||_\K \}$. If $\bn\in\N^d$ is such that all its components are large enough, then $|\rep_{\bS}(\bn)|\le g(\bn)$ and we get $|f(\bn)|_\K 
= |\lambda \mu(\rep_{\bS}(\bn)) \gamma|_\K
\le c^{g(\bn)+2}$.
\end{proof}

\begin{runningexample}
Let $\mathsf{v}\colon\N\to \N,\ \ell\mapsto \Card\{w\in a^*b^*\colon |w|\le \ell\}$. Then  for all $n,\ell\in\N$, $|\rep_\mathcal{S}(n)|=\ell$ if and only if $n\in[\![\mathsf{v}(\ell-1),\mathsf{v}(\ell)-1]\!]$. It is easily seen that for all $\ell\in\N$, $\mathsf{v}(\ell)=\frac{(\ell+1)(\ell+2)}{2}$. Thus, for all $n\in\N$, $|\rep_{\mathcal{S}}(n)|\le \sqrt{2n}$. Applying Proposition~\ref{Pro : GrowthRate}, we get that $f\couple{m}{n} \in O(c^{\sqrt{\max\{m,n\}}})$ for some $c\in\mathbb{R}_{>0}$.
\end{runningexample}

Let us now show that the family of regular sequences is closed under projection. For $i\in[\![1,d]\!]$, we let $\delta_i(\bS)$ denote the $(d{-}1)$-dimensional abstract numeration system $(\mathcal{S}_1,\ldots,\mathcal{S}_{i-1},\mathcal{S}_{i+1},\ldots,\mathcal{S}_d)$, we let $\delta_i(\bdiese)$ denote the $(d{-}1)$-dimensional letter all whose components are equal to $\#$ and we let $\delta_i(\bA)=\big((A_1)_\#\times \cdots\times (A_{i-1})_\#\times (A_{i+1})_\#\times\cdots\times (A_d)_\#\big)\setminus\{\delta_i(\bdiese)\}$ denote the corresponding $(d{-}1)$-dimensional alphabet. Similarly, for $\ba\in\bA$, we let $\delta_i(\ba)$ denote the letter obtained by deleting the $i$-th component of $\ba$.

\begin{proposition}
\label{Pro : Projection-Regular}
Let $f\colon\N^d\to \K$ be an $(\bS,\K)$-regular sequence, let $i\in[\![1,d]\!]$ and let $k\in\N$. Then the $(d{-}1)$-dimensional sequence 
\[
	\delta_{i,k}(f)\colon\N^{d-1}\to \K,\
	\left[\begin{smallmatrix}
	n_1\vspace{-.15cm}\\ \vdots\\ n_{i-1}\\ n_{i+1}\vspace{-.15cm}\\ \vdots\\ n_d
	\end{smallmatrix}\right]
	\mapsto f\left[\begin{smallmatrix}
	n_1\vspace{-.15cm}\\ \vdots\\ n_{i-1}\\ k\\ n_{i+1}\vspace{-.15cm}\\ \vdots\\ n_d
	\end{smallmatrix}\right]
\]
is $(\delta_i(\bS),\K)$-regular.
\end{proposition}

\begin{proof}
The set $X_{i,k}:=\{\bn\in\N^d\colon n_i=k\}$ is $\bS$-recognizable. By Proposition~\ref{Pro : HadamardReg}, the series $S_f\odot \underline{\rep_{\bS}(X_{i,k})}$ is $\K$-recognizable. Let $(\lambda,\mu,\gamma)$ be a linear representation, say of dimension $r$, of the latter series. Define a morphism $\mu'\colon \big(\delta_i(\bA)\big)^*\to \K^{r\times r}$ by setting, for each $\bb\in\delta_i(\bA)$, 
\[
	\mu'(\bb)=\sum_{\substack{\ba\in\bA\\ \delta_i(\ba)=\bb}}\mu(\ba).
\]
Then for all $\by\in(\delta_i(\bA))^*$ such that $|\by|\ge |\rep_{\mathcal{S}_i}(k)|$, we have $\lambda\mu'(\by)\gamma=(S_{\delta_{i,k}(f)},\by)$. Therefore, the series $S_{\delta_{i,k}(f)}$ is a finite modification of a $\K$-recognizable series, hence it is $\K$-recognizable.
\end{proof}

\section{Working hypothesis}
\label{Sec : WH}

In what follows, we sometimes impose as an extra condition on the abstract numeration system $\bS$ that the numeration language $\bL$ is \emph{prefix-closed}, that is, 
\begin{equation}
	\tag{WH}
	\forall \bu,\bv\in \bA^*,\ 
	\bu\bv \in \bL \implies \bu\in \bL.
\end{equation}
This amounts to asking that all languages $L_1,\ldots,L_d$ are prefix-closed. 

We note that this working hypothesis is satisfied by most of the usual numeration systems. For example, it is true for integer base numeration systems, the Zeckendorf numeration system, and more generally for all Bertrand numeration systems~\cite{Bertrand-Mathis1989,BruyereHansel1997}, as well as for substitutive numeration systems~\cite{DumontThomas1989}. This hypothesis is also crucial in order to generalize various properties of integer bases to abstract numeration systems. In particular, it is used in order to be able to represent real numbers~\cite{CharlierLeGonidecRigo2011} or to study the carry propagation of the successor function~\cite{BertheFrougnyRigoSakarovitch2020}. Finally, we note that from the proof of the fact that an infinite word is $\mathcal{S}$-automatic for some abstract numeration $\mathcal{S}$ if and only if it is morphic, it can be deduced that an infinite word is $\mathcal{S}$-automatic for some abstract numeration $\mathcal{S}$ if and only if it is $\mathcal{S}'$-automatic for some abstract numeration $\mathcal{S}'$ having a prefix-closed numeration language. See~\cite{RigoMaes2002}, and also \cite{CharlierKarkiRigo2010} for a multidimensional version of this result.

In the present work, the fact that $\bL$ is prefix-closed is used in order to obtain the three characterizations of $(\bS,\K)$-regular sequences in terms of the $\bS$-kernel given in Section~\ref{Sec : Kernel}. It also has repercussions on some results of Section~\ref{Sec : Automatic}. Every result that requires this working hypothesis is marked by (WH).

\section{$\bS$-kernel of a sequence}
\label{Sec : Kernel}

The following considerations generalize those of \cite[Chapter 5]{BerstelReutenauer2011}.

\begin{definition}
For $f\colon\N^d\to\K$ and $\bw\in\bA^*$, we define a sequence $f\circ\bw\colon\N^d\to\K$ by 
\[
	(f\circ\bw) (\bn) = 
	\begin{cases}
	f(\val_{\bS}	(\rep_{\bS}(\bn)\bw)) 
	&\text{if } \rep_{\bS}(\bn)\bw \in \bL\\
	0
	& \text{otherwise}
	\end{cases}
\]
for all $\bn\in \N^d$. A $\K$-submodule of $\K^{\N^d}$ is called \emph{stable} if it is closed under all operations $f \mapsto f\circ\bw$ for all $\bw\in \bA^*$. The \emph{$\bS$-kernel} of a sequence $f\colon\N^d\to\K$, denoted $\ker_{\bS}(f)$, is the set of all sequences of the form $f\circ\bw$:
\[
	\ker_{\bS}(f) =\{f\circ\bw \colon \bw\in \bA^*\}.
\]
\end{definition}

\begin{runningexample}
For all $\bw\in\bA^+\setminus(\bA^*\couple{a}{a}\cup\bA^*\couple{b}{b})$, we have $f\circ\bw=0$.
For all $u\in a^*b^*$, we have $ub\in a^*b^*$ so we get 
\begin{equation}
\label{Eq : f-circ-bb}
	f\circ \couple{b}{b}=f+1.
\end{equation}
Moreover, for all $\bn\in\N^2$, 
\[
	(f\circ \couple{a}{a})(\bn)=
	\begin{cases}
	f(\bn)+1		&\text{if } \bn\in \val_{\mathcal{S}}(a^*)\times \val_{\mathcal{S}}(a^*)\\
	0			&\text{else}.
	\end{cases}
\]
Some values of the function $f\circ \couple{ab}{ab}$ are given in Figure~\ref{Fig : circ}.
\begin{figure}[htb]
\[
	\begin{array}{c|ccccc}
	\bn & \couple{0}{1} & \couple{1}{2}   & \couple{3}{2}  &\couple{6}{3}  \\[3pt]	
	\hline
	\rep_{\bS}(\bn) & \couple{\#}{a}  & \couple{a}{b}    & \couple{aa}{\# b}  & \couple{aaa}{\# aa} \\[3pt]
	\rep_{\bS}(\bn)\couple{ab}{ab}   & \couple{\#ab}{aab}  & \couple{aab}{bab}   &  \couple{aaab}{\# bab}  & \couple{aaaab}{\# aaab} \\[3pt]
	\val_{\bS}(\rep_{\bS}(\bn)\couple{ab}{ab}) & \couple{4}{7}   & \nexists  	& \nexists  &  \couple{16}{11}  \\[3pt]
	(f\circ \couple{ab}{ab} ) (\bn) & 2 & 0  & 0& 4
	\end{array}
\]
\caption{Some values of the function $f\circ \couple{ab}{ab}$.}
\label{Fig : circ}
\end{figure}
\end{runningexample}

Let us establish a link between sequences and series, i.e., between the $\K$-module $\K^{\N^d}$ equipped with the operations $f\mapsto f\circ\bw$ for all $\bw\in \bA^*$, and the $\K$-module $\KAd$ equipped with the operations $S\mapsto S\bw^{-1}$ for all $\bw\in \bA^*$. Let 
\[
	Z = \{S\in\KAd\colon \text{ for all } \bw\in\bA^*\setminus\bL,\ (S,\bw)=0 \}
\] and let 
\[
	h\colon \K^{\N^d} \to Z,\ f \mapsto S_f.
\] 
Clearly, $Z$ is a $\K$-submodule of $\KAd$.

\begin{lemma}
\label{Lem : h-iso-modules}
The map $h$ is an isomorphism of $\K$-modules. \end{lemma}

\begin{proof}
Let us first show that $h$ is injective. Let $f_1,f_2\colon\N^d\to\K$ such that $S_{f_1}=S_{f_2}$. For all $\bn\in\N^d$, $f_1(\bn)=(S_{f_1},\rep_{\bS}(\bn))=(S_{f_2},\rep_{\bS}(\bn))=f_2(\bn)$. Therefore $f_1=f_2$. The map $h$ is surjective since for all $S\in Z$, the sequence $f\colon\N^d\to\K,\ \bn\mapsto (S,\rep_{\bS}(\bn))$ is such that $S=S_f$. Finally, for all $f_1,f_2\colon\N^d\to\K$ and $k_1,k_2 \in \K$, we have $S_{k_1f_1 + k_2 f_2}=k_1S_{f_1}+ k_2 S_{f_2}$.
\end{proof}

\begin{lemma}(WH)
The $\K$-submodule $Z$ of $\KAd$ is stable.
\end{lemma}

\begin{proof}
Let $\bw\in \bA^*$ and $S\in Z$. Since $\bL$ is prefix-closed, for all $\bu\in\bA^*\setminus\bL$, $\bu\bw\notin \bL$, hence $(S\bw^{-1},\bu)=(S,\bu\bw)=0$. Therefore, $S\bw^{-1}\in Z$.
\end{proof}

\begin{lemma}(WH)
\label{Lem : S_f-circ}
For all $\bw\in\bA^*$ and $f\colon\N^d\to\K$, $S_{f\circ\bw}=S_f\bw^{-1}$.
\end{lemma}

\begin{proof}
Let $\bw,\bu\in\bA^*$ and $f\colon\N^d\to\K$. By using that $\bL$ is prefix-closed, we obtain
\begin{align*}
	(S_{f \circ \bw},\bu)
	&=\begin{cases}
	(f\circ\bw)(\val_{\bS}(\bu)) 
	& \text{if } \bu \in \bL \\
	0 
	& \text{otherwise}
	\end{cases}\\
	&=\begin{cases}
	f(\val_{\bS}(\bu\bw)) 
	& \text{if } \bu\in\bL \text{ and } \bu\bw\in\bL\\
	0 
	& \text{otherwise}
	\end{cases}\\
	&=\begin{cases}
	f(\val_{\bS}(\bu\bw)) 
	& \text{if } \bu\bw\in\bL\\
	0
	& \text{otherwise}
	\end{cases}\\
	&=(S_f,\bu\bw)\\
	&=(S_f \bw^{-1},\bu).
\end{align*}
\end{proof}

\begin{remark}
The notion of $\bS$-kernel of a sequence is not left-right symmetric. The $\bS$-kernel defined above may be seen as the \emph{right} $\bS$-kernel. The \emph{left $\bS$-kernel} of a sequence $f\colon\N^d\to \K$ would then be the set of sequences $\{\bw\circ f\colon \bw\in\bA^*\}$ where for every $\bw\in\bA^*$, $\bw\circ f\colon\N^d\to\K$ is the sequence defined by
\[
	(\bw\circ f) (\bn) = 
	\begin{cases}
	f(\val_{\bS}	(\bw\,\rep_{\bS}(\bn))) 
	&\text{if } \bw\,\rep_{\bS}(\bn) \in \bL\\
	0
	& \text{otherwise}
	\end{cases}
\]
for all $\bn\in\N^d$. In this case, we need to adapt the conventions used so far: we pad representations of vectors of integers with $\#$'s on the right, we ask the numeration language $\bL$ to be \emph{suffix-closed} and we say that a $\K$-submodule of $\K^{\N^d}$ is \emph{left stable} if it is closed under all operations $f\mapsto \bw\circ f$ for all $\bw\in\bA^*$. Provided that these conventions are taken, all the results of this work can be adapted to the left version of the $\bS$-kernel. 

In this paper, we chose the right version of the operation $\circ$ to stick to the usual definition of the kernel in integer bases. This also justifies the choice of the operation $S\mapsto S\bw^{-1}$  in Section~\ref{Sec : Formal series}. Indeed, the left version of Lemma~\ref{Lem : S_f-circ} is to say that $S_{\bw\circ f}=\bw^{-1}S_f$. 
\end{remark}

\begin{theorem}(WH)
\label{The : Reg-sous-module}
A sequence $f\colon\N^d\to\K$ is $(\bS,\K)$-regular if and only if there exists a stable finitely generated $\K$-submodule of $\K^{\N^d}$ containing $f$. 
\end{theorem}

\begin{proof}
Suppose that such a $\K$-submodule $M$ exists. We have to show that $S_f$ is $\K$-recognizable. To this aim, we prove that $h(M)$ is a stable finitely generated $\K$-submodule of $\KAd$ containing $S_f$, and we conclude by Theorem~\ref{The : Series-Rec-Submodule}. First, $h(M)$ is a finitely generated $\K$-submodule of $Z$ by Lemma~\ref{Lem : h-iso-modules}, and hence also of $\KAd$. Next, $S_f\in h(M)$ since $f\in M$. Finally, let us show that $h(M)$ is stable. Let $g\in M$ and $\bw\in \bA^*$. By Lemma~\ref{Lem : S_f-circ}, $S_g\bw^{-1} = S_{g\circ\bw}$. Since $M$ is stable, $g\circ\bw\in M$, yielding $S_g\bw^{-1}\in h(M)$.

Conversely, let $f\colon\N^d\to\K$ be a $(\bS,\K)$-regular sequence. By definition, the series $S_f$ is $\K$-recognizable. By Theorem~\ref{The : Series-Rec-Submodule}, there exists a stable finitely generated $\K$-submodule $N$ of $\KAd$ containing $S_f$. Let $G_1,\ldots,G_n \in \KAd$ be such that $N=\langle G_1,\ldots,G_n\rangle_\K$. Since $\bL$ is a regular language, there exist words $w_1,\ldots,w_\ell\in \bA^*$ such that 
\[
	\{ \bL \bw^{-1} \colon \bw\in \bA^* \} 
	= \{\bL 	\bw_1^{-1},\ldots,\bL \bw_\ell^{-1}\}.
\]
Without loss of generality, we may assume that $\bL \bw_1^{-1}=\bL$. Define
\[
	N' = \big\langle 
	\{G_i\odot(\underline{\bL} \bw_j^{-1}) 
	\colon i\in[\![1,n]\!],\ j\in[\![1,\ell]\!]\} 		\big\rangle_\K.
\]
Clearly, $N'$ is a finitely generated $\K$-submodule of $\KAd$. Let us prove that $N'$ contains $S_f$, is stable and is a subset of $Z$. 

First, since $S_f\in N$, there exist $\alpha_1, \ldots, \alpha_n\in \K$ such that $S_f=\sum_{i=1}^n \alpha_i G_i$. Then
\[
	S_f 
	= S_f \odot \underline{\bL}
	= \sum_{i=1}^n \alpha_i(G_i\odot\underline{\bL}) 
	= \sum_{i=1}^n \alpha_i(G_i\odot(\underline{\bL}\bw_1^{-1})) \in N'.
\]
Second, we show that $N'$ is stable. The operations $\KAd\to \KAd,\ S\mapsto S\bw^{-1}$ being linear for all $\bw\in\bA^*$ and since $(S\bu^{-1})\bv^{-1}=S(\bv\bu)^{-1}$ for all $S\in\KAd$ and $\bu,\bv\in\bA^*$, it is enough to show that $N'$ contains the series $S\ba^{-1}$ for all series $S$ that generate $N'$ and all $\ba\in\bA$. Let thus $i\in[\![1,n]\!]$, $j\in[\![1,\ell]\!]$ and $\ba\in\bA$. Since $N$ is stable, $G_i \ba^{-1} = \sum_{t=1}^n \beta_{i,\ba,t} G_t$ for some $\beta_{i,\ba,1},\ldots,\beta_{i,\ba,n}\in \K$. Moreover, $\bL(\ba\bw_j)^{-1}=\bL \bw_{k_{j,\ba}}^{-1}$ for some $k_{j,\ba}\in[\![1,\ell]\!]$. We get
\begin{align*}
	\big(G_i\odot(\underline{\bL} \bw_j^{-1})\big)\ba^{-1}
	&= G_i\ba^{-1}\odot (\underline{\bL}\bw_j^{-1})\ba^{-1}\\
	&= G_i\ba^{-1}\odot \underline{\bL}(\ba\bw_j)^{-1}\\
	&= \sum_{t=1}^n \beta_{i,\ba,t} (G_t \odot \underline{\bL}\bw_{k_{j,\ba}}^{-1})
\end{align*}
which shows that indeed $\big(G_i\odot(\underline{\bL} \bw_j^{-1})\big)\ba^{-1}$ belongs to $N'$. Third, in order to get $N' \subseteq Z$, it suffices to see that for all $j\in[\![1,\ell]\!]$ and $\bw\in \bA^*\setminus \bL$, $(\underline{\bL} \bw_j^{-1},\bw)=0$, which follows from the fact that $\bL$ is prefix-closed.

By Lemmas~\ref{Lem : h-iso-modules} and~\ref{Lem : S_f-circ}, $M=h^{-1}(N^\prime)$ is a stable finitely generated $\K$-submodule of $\K^{\N^d}$ containing $f$, as desired.
\end{proof}

\begin{remark}
The proof of Theorem~\ref{The : Series-Rec-Submodule} is constructive in the sense that any linear representation of a series $S\in\KA$ gives rise to a set of generators of a finitely generated stable $\K$-submodule of $\KA$ containing $S$, and conversely, any set of generators of such a $\K$-submodule provides a linear representation for $S$.

Provided that DFAs accepting the numeration languages $L_1,\ldots,L_d$ are known, this implies that the proof of Theorem~\ref{The : Reg-sous-module} is constructive as well: any linear representation of $S_f$ gives rise to a set of generators of a stable finitely generated $\K$-submodule of $\K^{\N^d}$ containing $f$, and conversely, any set of generators of such a $\K$-submodule provides a linear representation for $S_f$.
\end{remark} 
 
Now, we prove two properties that are crucial for the proofs of Theorems~\ref{The : PracticalCriterion} and~\ref{The : reg-noyau-fin-eng}. First, for all $\bw\in\bA$, the operation $f\mapsto f\circ\bw$ is linear. 

\begin{lemma}
\label{Lem : Linear-Circ}
For all $\bw\in \bA^*$, $f_1,f_2\colon\N^d\to\K$ and $k_1,k_2\in \K$, $(k_1f_1+k_2f_2)\circ \bw=k_1(f_1\circ \bw)+k_2(f_2\circ \bw)$.
\end{lemma}

\begin{proof}
This is a straightforward verification.
\end{proof}

The second property of $\circ$ that we establish cannot be obtained from the notion of $\bS$-kernel used in~\cite{RigoMaes2002}.

\begin{lemma}(WH)
\label{Lem : composition-circ}
For all $\bu,\bv\in \bA^*$ and $f\colon\N^d\to\K$, $(f\circ \bv)\circ \bu = f \circ \bu\bv$.
\end{lemma}

\begin{proof}
This is a consequence of Lemmas~\ref{Lem : h-iso-modules} and~\ref{Lem : S_f-circ} and the property that for all $\bu,\bv\in \bA^*$ and $S\in\KA$, $(S\bv^{-1})\bu^{-1}=S(\bu\bv)^{-1}$.
\end{proof}

\begin{runningexample}
We illustrate the previous lemma by computing $\big((f\circ \couple{b}{b})\circ \couple{a}{a}\big)(\bn)$ for $\bn\in\left\{\couple{0}{1}, \couple{1}{2}, \couple{3}{2},\couple{6}{3}\right\}$. 
Since $\rep_{\bS}\couple{1}{3}=\couple{\#a}{aa}$ and $\rep_{\bS}\couple{10}{6}=\couple{aaaa}{\#aaa}$, we get $f\couple{1}{3}=1$ and $f\couple{10}{6}=3$.
Then by using~\eqref{Eq : f-circ-bb}, we obtain the array of Figure~\ref{Fig : Composition}. 
\begin{figure}[htb]
\[
	\begin{array}{c|ccccc}
	\bn 				
	& \couple{0}{1} 		& \couple{1}{2}		& \couple{3}{2}  		&\couple{6}{3} \\[3pt]	
	\hline
	\rep_{\bS}(\bn) 
	& \couple{\#}{a}		& \couple{a}{b}		& \couple{aa}{\# b} 		& \couple{aaa}{\# aa} \\[3pt]
	\rep_{\bS}(\bn)\couple{a}{a}   
	& \couple{\#a}{aa}  & \couple{aa}{ba}	&  \couple{aaa}{\# ba}	& \couple{aaaa}{\# aaa} \\[3pt]
	\val_{\bS}(\rep_{\bS}(\bn)\couple{a}{a}) 
	& \couple{1}{3}   	& \nexists  			& \nexists  				&  \couple{10}{6} \\[3pt]
	((f\circ \couple{b}{b} ) \circ \couple{a}{a})(\bn) 
	& 2 					& 0  				& 0							& 4
	\end{array}
\]
\caption{Some values of the function $(f\circ \couple{b}{b})\circ \couple{a}{a}$.}
\label{Fig : Composition}
\end{figure}
Observe that the last row coincides with that of Figure~\ref{Fig : circ}.
\end{runningexample}

We obtain the following practical criterion for $(\bS,\K)$-regularity.

\begin{theorem}(WH)
\label{The : PracticalCriterion}
A sequence $f\colon\N^d\to \K$ is $(\bS,\K)$-regular if and only if there exist $r\in\N_{\ge 1}$ and $f_1,f_2,\ldots,f_r \colon\N^d\to\K$ such that $f=f_1$ and for all $\ba\in \bA$ and all $i\in[\![1,r]\!]$, there exist $k_{\ba,i,1},\ldots,k_{\ba,i,r}\in\K$ such that
\[
	f_i\circ \ba = \sum_{j=1}^r k_{\ba,i,j} f_j. 
\]
\end{theorem}

\begin{proof}
The necessary condition follows from Theorem~\ref{The : Reg-sous-module}. Conversely, suppose that such sequences $f_1,f_2,\ldots,f_r$ exist. By induction on the length of $\bw$ and by using Lemmas~\ref{Lem : Linear-Circ} and~\ref{Lem : composition-circ}, we obtain that for all $\bw\in\bA^*$, $f_i\circ \bw$ is  a $\K$-linear combination of $f_1,f_2,\ldots,f_r$. Therefore, $\langle f_1,\ldots, f_r\rangle_\K$ is a stable finitely generated $\K$-submodule that contains $f$. By Theorem~\ref{The : Reg-sous-module}, $f$ is $(\bS,\K)$-regular.
\end{proof}

\begin{runningexample}
We prove the $(\bS,\N)$-regularity of the sequence $f$ by using Theorem~\ref{The : PracticalCriterion}. Define the sequence 
\[
	g\colon\N^2\to\N,\ 
	\bn\mapsto
	\begin{cases}
	f(\bn) 	&\text{if } \bn\in \val_{\mathcal{S}}(a^*)\times\val_{\mathcal{S}}(a^*)\\
	0		&\text{otherwise}.
	\end{cases}
\]
We show that the ten functions
$f$, 
$g$, 
$\chi_{\{0\}\times\val_{\mathcal{S}}(a^*)}$, 
$\chi_{\{0\}\times\N}$, 
$\chi_{\val_{\mathcal{S}}(a^*)\times\{0\}}$, 
$\chi_{\val_{\mathcal{S}}(a^*)\times\val_{\mathcal{S}}(a^*)}$, 
$\chi_{\val_{\mathcal{S}}(a^*)\times\N}$,
$\chi_{\N\times\{0\}}$,
$\chi_{\N\times\val_{\mathcal{S}}(a^*)}$ 
and $1$ satisfy Theorem~\ref{The : PracticalCriterion}. Let $\ba\in \bA$.
First, we have 
\[
	f\circ \ba 
	=\begin{cases}
	g+\chi_{\val_{\mathcal{S}}(a^*)\times\val_{\mathcal{S}}(a^*)}		& \text{if } \ba = \couple{a}{a} \\
f+1 	& \text{if } \ba = \couple{b}{b} \\
0 & \text{otherwise }
	\end{cases}
\]
and 
\[
	g\circ  \ba 
	=\begin{cases}
	g+\chi_{\val_{\mathcal{S}}(a^*)\times\val_{\mathcal{S}}(a^*)} 	& \text{if } \ba = \couple{a}{a} \\
	0 	& \text{otherwise}.
\end{cases}
\]
Next, take $X_1,X_2\in \{\{0\},\val_{\mathcal{S}}(a^*),\N\}$ such that not both $X_1,X_2$ are equal to $\{0\}$. Note that $\chi_{\N\times\N}=1$. We have 
$\chi_{X_1\times X_2}\circ \ba
=\chi_{Y_1\times Y_2}$ where for each $i \in \{1,2\}$,
\[
	Y_i
	=\begin{cases}
	\{0\} 					& \text{if } a_i=\# \\
	\val_{\mathcal{S}}(a^*) 	& \text{if } a_i=a \text{ and } X_i\in \{\val_{\mathcal{S}}(a^*),\N\}\\
	\N 			& \text{if } a_i=b \text{ and } X_i=\N\\
	\emptyset 		& \text{otherwise}.
	\end{cases}
\]
Following the proof of Theorem~\ref{The : Reg-sous-module}, the ten series $S_f$, 
$S_g$, 
$S_{\chi_{\{0\}\times\val_{\mathcal{S}}(a^*)}}$, 
$S_{\chi_{\{0\}\times\N}}$, 
$S_{\chi_{\val_{\mathcal{S}}(a^*)\times\{0\}}}$, 
$S_{\chi_{\val_{\mathcal{S}}(a^*)\times\val_{\mathcal{S}}(a^*)}}$, 
$S_{\chi_{\val_{\mathcal{S}}(a^*)\times\N}}$,
$S_{\chi_{\N\times\{0\}}}$,
$S_{\chi_{\N\times\val_{\mathcal{S}}(a^*)}}$ 
and $S_1$ generate a stable $\N$-submodule of $\N\langle \langle \bA \rangle \rangle$ containing $S_f$. 
\end{runningexample}

It follows from Lemmas~\ref{Lem : Linear-Circ} and~\ref{Lem : composition-circ} that $\langle\ker_{\bS}(f)\rangle_\K$ is the smallest stable $\K$-submodule of $\K^{\N^d}$ containing $f$. For an arbitrary semiring $\K$, the fact that $f$ is an $(\bS,\K)$-regular sequence does not imply that $\langle\ker_{\bS}(f)\rangle_\K$ is finitely generated. The following theorem provides us with some cases where $\langle\ker_{\bS}(f)\rangle_\K$ is indeed finitely generated.

\begin{theorem}(WH)
\label{The : reg-noyau-fin-eng}
Let $f\colon\N^d\to\K$. If $\langle\ker_{\bS}(f)\rangle_\K$ is finitely generated then $f$ is $(\bS,\K)$-regular. If $f$ is $(\bS,\K)$-regular and if moreover $\K$ is finite or is a commutative ring, then $\langle\ker_{\bS}(f)\rangle_\K$ is finitely generated.
\end{theorem}

\begin{proof}
Since $\langle \ker_{\bS}(f) \rangle_\K$ is a stable $\K$-submodule of $\K^{\N^d}$ containing $f$, the first part of the statement follows from Theorem~\ref{The : Reg-sous-module}. Second, suppose that $f$ is $(\bS,\K)$-regular and that $\K$ is finite or is a commutative ring. Then the series $S_f$ is $\K$-recognizable and by Theorem~\ref{The : Series-Rec-Smallest-Submodule}, the $\K$-submodule $M:=\langle \{S_f\bw^{-1}\colon\bw\in\bA^*\}\rangle_\K$ of $\KAd$ is finitely generated. By Lemma~\ref{Lem : S_f-circ}, for all $\bw\in\bA^*$, we have $h(f\circ\bw)=S_{f\circ\bw}=S_f\bw^{-1}$. Therefore, $\langle \ker_{\bS}(f) \rangle_\K=h^{-1}(M)$, and the conclusion follows from Lemma~\ref{Lem : h-iso-modules}.
\end{proof}

\begin{runningexample}
For $\bw\in \bA^*\setminus \couple{a}{a}^*\couple{b}{b}^*$, we have $f \circ \bw=0$, for $k\in \N$, we have $f\circ \couple{b}{b}^k=f+k$ and for $k,k'\in \N$ with $k\ge 1$, we have
\[
	\left(f\circ \couple{a}{a}^k\couple{b}{b}^{k'}\right)(\bn)= 	\begin{cases}
	f(\bn)+k+k' 	& \text{if } \bn \in\val_{\mathcal{S}}(a^*)\times \val_{\mathcal{S}}(a^*)\\
	0 			& \text{otherwise}.
	\end{cases}
\]
Therefore, $\langle \ker_{\bS}(f) \rangle_\N$ is not finitely generated. However, $\langle \ker_{\bS}(f) \rangle_\Z$ is finitely generated by Theorem~\ref{The : reg-noyau-fin-eng}. Indeed, it is easily seen that $\langle \ker_{\bS}(f) \rangle_\Z= \langle f,f\circ\couple{a}{a},f\circ\couple{b}{b},f\circ\couple{aa}{aa}\rangle_\Z$.
\end{runningexample}

\begin{remark}
As mentioned earlier, Theorem~\ref{The : reg-noyau-fin-eng} can be reformulated in terms of the left $\bS$-kernel. The notion of left (resp.\ right) kernel then corresponds to that of left (resp.\ right) $(\mathcal{S},\K)$-regular sequence, which is obtained by using a right (resp.\ left) padding. In the unidimensional case, the families of left $(\mathcal{S},\K)$-regular sequences and of right $(\mathcal{S},\K)$-regular sequences obviously coincide since no padding is necessary. Therefore, for any sequence $f\colon\N\to\K$ where $\K$ is finite or is a commutative ring, even though the left and right $\mathcal{S}$-kernels of $f$ may be different sets, they generate $\K$-submodules of $\K^\N$ that are simultaneously finitely generated. However, there is no such nice analogy in higher dimensions since it might be that a left $(\bS,\K)$-regular sequence is not a right $(\bS,\K)$-regular sequence, or vice-versa. 
\end{remark}

\section{$\bS$-Automatic sequences}
\label{Sec : Automatic}

Automatic sequences for abstract numeration systems were originally defined in~\cite{Rigo2000} and were further studied in~\cite{CharlierKarkiRigo2010,RigoMaes2002}.

\begin{definition}
A sequence $f\colon\N^d\to\K$ is called \emph{$\bS$-automatic} if there exists a deterministic finite automaton with output (DFAO) $\mathcal{A}=(Q,q_0,\delta,\bA,\tau,\Delta)$, where the output alphabet $\Delta$ is a subset of $\K$, such that for all $\bn\in \N^d$, $f(\bn)=\tau(\delta(q_0,\rep_{\bS}(\bn)))$. We say that such a DFAO \emph{generates} the sequence $f$. 
\end{definition}

\begin{lemma}
\label{Lem : DFAO-with-0}
If $f\colon\N^d\to\K$ is an $\bS$-automatic sequence then there exists a complete DFAO $\mathcal{A}=(Q,i,\delta,\bA,\tau,\Delta)$ such that for all $\bw\in \bA^*$, 
\begin{align*}
	\tau(\delta(i,\bw))
	=\begin{cases}
	f(\val_{\mathcal{S}}(\bw)) 
	& \text{if } \bw\in\bL\\
	0 
	&\text{else}.
	\end{cases}
\end{align*}
\end{lemma}

\begin{proof}
Let $\mathcal{A}_1=(Q_1,i_1,\delta_1,\bA,\tau_1,\Delta)$ be a complete DFAO generating $f$ and let $\mathcal{A}_2=(Q_2,i_2,\delta_2,\bA,F)$ be the minimal automaton of $\bL$. The following DFAO satisfies the desired property: $\mathcal{B}=(Q_1 \times Q_2,(i_1,i_2),\delta,\bA,\tau,\Delta\cup\{0\})$ where 
\[
	\delta\colon (Q_1\times Q_2)\times \bA\to Q_1\times Q_2,\
	((q_1,q_2),\ba)\mapsto
	(\delta_1(q_1,\ba),\delta_2(q_2,\ba))
\] 
and 
\[
	\tau \colon Q_1 \times Q_2 \to \Delta\cup\{0\},\ 
	(q_1,q_2) \mapsto 
	\begin{cases}
	\tau_1(q_1) &\text{if } q_2\in F\\
	0	&\text{else}.
\end{cases}
\]
\end{proof}

In the proof of Theorem~\ref{The : automatic-finite-kernel} below, we work with reverse representations $\rep_\mathcal{S}(\bn)^R$. A sequence $f\colon\N^d\to\K$ is \emph{reversal-$\bS$-automatic} if there exists a DFAO $\mathcal{A}=(Q,q_0,\delta,\bA,\tau,\Delta)$, where $\Delta$ is a finite subset of $\K$, such that for all $\bn\in \N^d$, $f(\bn)=\tau(\delta(q_0,\rep_{\bS}(\bn)^R))$.

\begin{lemma}
\label{Lem : automatic-reversal-automatic}
A sequence $f\colon\N^d\to\K$ is $\bS$-automatic if and only if it is reversal-$\bS$-automatic.
\end{lemma}

\begin{proof}
This is a straightforward generalization of the proof of~\cite[Proposition~9]{RigoMaes2002} to multidimensional sequences.
\end{proof}

We now characterize $\bS$-automatic sequences by means of the $\bS$-kernel. 

\begin{theorem}
\label{The : automatic-finite-kernel}
Let $f\colon\N^d\to\K$. If $f$ is $\bS$-automatic then $\ker_{\bS}(f)$ is finite. Under WH, if $\ker_{\bS}(f)$ is finite then $f$ is $\bS$-automatic.
\end{theorem}

\begin{proof}
First, suppose that $f$ is an $\bS$-automatic sequence. Let $\mathcal{A}=(Q,i,\delta,\bA,\tau,\Delta)$ be a DFAO as in Lemma~\ref{Lem : DFAO-with-0}. For all $\bw\in\bA^*$, define a DFAO $\mathcal{A}_{\bw}=(Q,i,\delta,\bA,\tau_{\bw},\Delta)$ where for all $q\in Q$, $\tau_{\bw}(q)=\tau(\delta(q,\bw))$. Then for all $\bw\in\bA^*$, $\mathcal{A}_{\bw}$ generates the sequence $f\circ\bw$. Since $\Card(\ker_{\bS}(f))\le\Card\{\tau_{\bw}\colon\bw\in\bA^*\}\le \Card(\Delta)^{\Card(Q)}$, we obtain that $\ker_{\bS}(f)$ is finite.

Conversely, assume that $\ker_{\bS}(f)$ is finite.
Define a DFAO $\mathcal{A}=(Q,i,\delta,\bA,\tau,\Delta)$ where $Q=\{(\bL\bw^{-1},f\circ\bw)\colon \bw\in\bA^*\}$, $i=(\bL,f)$ and for all $\bw\in\bA^*$, $\delta((\bL\bw^{-1}, f\circ\bw),\ba)=(\bL(\ba\bw)^{-1},f\circ \ba\bw)$ and $\tau(\bL\bw^{-1}, f\circ\bw)=f(\val_{\bS}(\bw))$ if $\bw\in\bL$ (the output function $\tau$ can take whatever values on words $\bw\notin\bL$). Let us prove that for all $\bw\in \bA^*$, $\delta(i,\bw^R)=(\bL\bw^{-1}, f\circ\bw)$. We proceed by induction on the length of $\bw$. The base case is given by $\delta(i,\boldsymbol{\varepsilon})=i=(\bL,f)$. Now, suppose that the claim holds for some $\bw\in \bA^*$. Then, by Lemma~\ref{Lem : composition-circ}, for all $\ba\in \bA$, 
$\delta(i,(\ba\bw)^R)
	= \delta(\delta(i,\bw^R),a)
	= \delta((\bL\bw^{-1},f\circ\bw),\ba)
	= (\bL(\ba\bw)^{-1},f\circ\ba\bw)$.
It follows that for all $\bn \in \N^d$, 
$\tau(\delta(i,\rep_{\bS}(\bn)^R))
= \tau(\bL\rep_{\bS}(\bn)^{-1},f\circ\rep_{\bS}(\bn))= f(\bn)$. Thus, $f$ is reversal-$\bS$-automatic, hence also $\bS$-automatic by Lemma~\ref{Lem : automatic-reversal-automatic}.
\end{proof}

\begin{remark}
Even though the statement of Theorem~\ref{The : automatic-finite-kernel} (when restricted to unidimensional sequences) coincide with that of~\cite[Proposition 7]{RigoMaes2002}, Theorem~\ref{The : automatic-finite-kernel} is indeed a new result since we are working with a different notion of $\bS$-kernel. As a consequence, we obtain that, for any given sequence $f$, both kernels are simultaneously finite. In fact, the cardinality of the $\bS$-kernel defined here is always greater than or equal to the cardinality of the kernel defined in \cite{RigoMaes2002}.
\end{remark}

We next establish the link between $\bS$-automatic sequences and $\bS$-regular sequences. 

\begin{lemma}
\label{Lem : Automatic-Fibers}
A sequence $f\colon\N^d\to\K$ is $\bS$-automatic if and only if it takes only finitely many values and for all $k\in \K$, the subsets $f^{-1}(k)$ of $\N^d$ are $\bS$-recognizable.
\end{lemma}

\begin{proof}
This is a straightforward adaptation of the proof of \cite[Theorem 8]{Rigo2000} to the multidimensional setting.
\end{proof}

\begin{theorem}
\label{The : Regular-Automatic}
Let $f\colon \N^d\to \K$. 
\begin{itemize}
\item If $f$ is $\bS$-automatic then it is $(\bS,\K)$-regular. 
\item If $f$ is $(\bS,\K)$-regular and takes only finitely many values, and if moreover $\K$ is included in a commutative ring, then $f$ is $\bS$-automatic. 
\item If $f$ is $(\bS,\K)$-regular and if $\K$ is finite, then $f$ is $\bS$-automatic.
\end{itemize}
\end{theorem}

\begin{proof}
Let $f\colon \N^d\to \K$ be an $\bS$-automatic sequence and let $\mathcal{A}=(Q,i,\delta,\bA,\tau,\Delta)$ be a DFAO as in Lemma~\ref{Lem : DFAO-with-0}. For $q\in Q$, define
\[
	\lambda_q=\begin{cases}
				1	&\text{if }q=i\\
				0	&\text{else}
				\end{cases}
\]
and $\gamma_q=\tau(q)$. Moreover, for $\ba \in\bA$ and $q,q'\in Q$, define
\[
	(\mu(\ba))_{q,q'}=\begin{cases}
				1	&\text{if }\delta(q,\ba)=q'\\
				0	&\text{else}.
				\end{cases}
\]
Then $(\lambda,\mu,\gamma)$ is a linear representation of dimension $\Card(Q)$ of the series $S_f$. 

Now, suppose that $f$ is $(\bS,\K)$-regular and takes only finitely many values, and that $\K$ is finite or is included in a commutative ring. Then $S_f$ is a $\K$-recognizable series with a finite image. By Theorem~\ref{The : FibresRegulieres}, for all $k\in \K$, the language $S_f^{-1}(k)$ is regular. Since for all $k\in\K$, $\rep_{\bS}(f^{-1}(k))=S_f^{-1}(k) \cap \bL$,
it follows from Lemma~\ref{Lem : Automatic-Fibers} that $f$ is $\bS$-automatic.
\end{proof}

\begin{remark}
Under the working hypothesis, another way to obtain the first item of Theorem~\ref{The : Regular-Automatic} is to use Theorems~\ref{The : reg-noyau-fin-eng} and~\ref{The : automatic-finite-kernel}. Indeed, if $\ker_{\bS}(f)$ is finite then clearly $\langle\ker_{\bS}(f)\rangle_\K$ is finitely generated.
\end{remark}

\begin{corollary}
\label{Cor : Automatic-Modulo}
Let $f\colon\N^d\to\Z$ be a $(\bS,\Z)$-regular sequence. For all $m\in \N_{\ge 2}$, the sequences $f \bmod m\colon\N^d\to\Z/m\Z,\ \bn\mapsto f(\bn) \bmod m$ are $\bS$-automatic.
\end{corollary}

\begin{proof}
Let $m\in\N_{\ge 2}$. Since $f$ is $(\bS,\Z)$-regular, $f\bmod m$ is clearly $(\bS,\Z/m\Z)$-regular. The result then follows from the third item of Theorem~\ref{The : Regular-Automatic}.
\end{proof}

\begin{runningexample}
\label{Ex : Automatic}
Since the sequence $f$ is $(\bS,\Z)$-regular, Corollary~\ref{Cor : Automatic-Modulo} implies that the sequences $f \bmod m\colon\N^2\to\Z/m\Z,\ \bn\mapsto f(\bn) \bmod m$ are $\bS$-automatic for all $m\in \N_{\ge 2}$. Since $\langle\ker_{\bS}(f) \rangle_\Z=\langle f,f\circ\couple{a}{a},f\circ\couple{b}{b},f\circ\couple{aa}{aa} \rangle_\Z$, we get that
$\langle\ker_{\bS}(f\bmod m) \rangle_{\Z/m\Z}=
	\big\{\big(
	i\cdot f
	+j\cdot \left(f\circ\couple{a}{a}\right)
	+k\cdot \left(f\circ\couple{b}{b}\right)
	+\ell\cdot \left(f\circ\couple{aa}{aa}\right)
	\big) \bmod m
	\colon i,j,k,\ell\in\Z/m\Z
	\big\}$. Since the numeration language $\bL$ is prefix-closed, the proof of Theorem~\ref{The : automatic-finite-kernel} provides us with an effective construction of a DFAO of size $9\cdot m^4$ computing $f\bmod m$. By directly using the definition of $f$, it can be seen that the DFAO of Figure~\ref{Fig : AutomatonMod} also computes $f\bmod m$. The size of this DFAO is $2m+7$. 
\begin{figure}[htb]
\begin{center}
\hspace*{-2cm}
\scalebox{0.6}{
\begin{tikzpicture}
\tikzstyle{every node}=[shape=circle, fill=none, draw=black,
minimum size=20pt, inner sep=2pt]
\node(initial) at (0,0){$0$};

\node (diezA) at (1.76,4.24){$0$};
\node (Adiez) at (1.76,-4.24){$0$};

\node (diezB) at (6,6){$0$};
\node (AB) at (10.24,4.24){$0$};
\node (Bdiez) at (6,-6){$0$};
\node (BA) at (10.24,-4.24){$0$};

\node (AA0) at (4.58,1.41){$0$};
\node (AA1) at (4,0)  {$1$};
\node (AA2) at (4.58,-1.41) {$2$};
\tikzstyle{every node}=[shape=circle, fill=none, draw=black,
minimum size=20pt, inner sep=0.8pt]
\node (AAM1) at (6,2) {$\scriptscriptstyle{m-1}$};

\tikzstyle{every node}=[shape=circle, fill=none, draw=black,
minimum size=20pt, inner sep=2pt]
\node (BB0) at (12.58,1.41){$0$};
\node (BB1) at (12,0)  {$1$};
\node (BB2) at (12.58,-1.41) {$2$};
\tikzstyle{every node}=[shape=circle, fill=none, draw=black,
minimum size=20pt, inner sep=0.8pt]
\node (BBM1) at (14,2) {$\scriptscriptstyle{m-1}$};

\tikzstyle{every node}=[shape=circle, fill=none, draw=black,
minimum size=20pt, inner sep=2pt]

\tikzstyle{every node}=[shape=circle, fill=none, draw=black, minimum size=15pt, inner sep=2pt]
\tikzstyle{every path}=[color=black, line width=0.5 pt]
\tikzstyle{every node}=[shape=circle, minimum size=5pt, inner sep=2pt]
\draw [-Latex] (-1,0) to node [above=-0.2] {$ $} (initial);

\draw[line width=0.4mm, dotted,-Latex, vert] (4.95,-1.52)arc (-120:75:1.87);
\draw[line width=0.4mm, dotted,-Latex, red] (12.95,-1.52)arc (-120:75:1.87);

\draw [-Latex,brown] (diezA) [loop above] to node [right] {$ $} (diezA);
\draw [-Latex,violet] (diezB) [loop above] to node [right] {$ $} (diezB);
\draw [-Latex,blue] (AB) [loop above] to node [right] {$ $} (AB);
\draw [-Latex,orange] (Adiez) [loop below] to node [right] {$ $} (Adiez);
\draw [-Latex,magenta] (Bdiez) [loop below] to node [right] {$ $} (Bdiez);
\draw [-Latex,cyan] (BA) [loop below] to node [right] {$ $} (BA);

\draw [-Latex,brown] (initial) to node {$ $} (diezA);

\draw [-Latex,orange] (initial) to node {$ $} (Adiez);

\draw [-Latex,vert] (AAM1) to node {$ $} (AA0);

\draw [-Latex,vert] (initial) to node {$ $} (AA1);
\draw [-Latex,vert] (diezA) to node {$ $} (AA1);
\draw [-Latex,vert] (Adiez) to node {$ $} (AA1);
\draw [-Latex,vert] (AA0) to node {$ $} (AA1);

\draw [-Latex,vert] (AA1) to node {$ $} (AA2);

\draw [-Latex,violet] (diezA) to node {$ $} (diezB);
\draw [-Latex,violet] (initial) to  node {$ $} (diezB);

\draw [-Latex,blue] (diezB) to node {$ $} (AB);
\draw [-Latex,blue] (initial) [bend left=15] to node {$ $} (AB);
\draw [-Latex,blue] (AA0) to  [bend left=25, in=170] node {$ $} (AB);
\draw [-Latex,blue] (AA1) to node {$ $} (AB);
\draw [-Latex,blue] (AA2) to node {$ $} (AB);
\draw [-Latex,blue] (AAM1) to node {$ $} (AB);
\draw [-Latex,blue] (diezA) to node {$ $} (AB);
\draw [-Latex,blue] (Adiez) to  [bend right=25,in=-140]  node {$ $} (AB);

\draw [-Latex,magenta] (Adiez) to node {$ $} (Bdiez);
\draw [-Latex,magenta] (initial) to node {$ $} (Bdiez);

\draw [-Latex,cyan] (Bdiez) to node {$ $} (BA);
\draw [-Latex,cyan] (initial) to [bend right=15] node {$ $} (BA);
\draw [-Latex,cyan] (AA0) to node {$ $} (BA);
\draw [-Latex,cyan] (AA1) to node {$ $} (BA);
\draw [-Latex,cyan] (AA2) to node {$ $} (BA);
\draw [-Latex,cyan] (AAM1) to node {$ $} (BA);
\draw [-Latex,cyan] (Adiez) to node {$ $} (BA);
\draw [-Latex,cyan] (diezA) to  [out=-10,in=90]  node {$ $} (BA);

\draw [-Latex,red] (BBM1) to  node {$ $} (BB0);
\draw [-Latex,red] (AAM1) to node {$ $} (BB0);

\draw [-Latex,red] (BB0) to node {$ $} (BB1);
\draw [-Latex,red] (AA0) to node {$ $} (BB1);
\draw [-Latex,red] (Bdiez) to node {$ $} (BB1);
\draw [-Latex,red] (BA) to node {$ $} (BB1);
\draw [-Latex,red] (diezB) to node {$ $} (BB1);
\draw [-Latex,red] (AB) to node {$ $} (BB1);
\draw [-Latex,red] (Adiez) to [out=10, in=210] node {$ $} (BB1);
\draw [-Latex,red] (diezA) to [out=-5, in=150] node {$ $} (BB1);
\draw [-Latex,red] (initial) to [bend right=65, looseness=1.3, out=-30,in=200]  node {$ $} (BB1);

\draw [-Latex,red] (BB1) to node {$ $} (BB2);
\draw [-Latex,red] (AA1) to node {$ $} (BB2);
\end{tikzpicture}
}
\end{center}
\vspace*{-2.5cm}
\hfill\scalebox{0.8}{
\fbox{\begin{tabular}{cc|cc}
\textcolor{brown}{$\blacksquare$} 	& $\couple{\#}{a}$ 
&\textcolor{blue}{$\blacksquare$} 	& $\couple{a}{b}$\\
\textcolor{violet}{$\blacksquare$} 	& $\couple{\#}{b}$ 
&\textcolor{magenta}{$\blacksquare$} & $\couple{b}{\#}$\\
\textcolor{orange}{$\blacksquare$} 	& $\couple{a}{\#}$
&\textcolor{cyan}{$\blacksquare$} 	& $\couple{b}{a}$\\
\textcolor{vert}{$\blacksquare$} 	& $\couple{a}{a}$  
& \textcolor{red}{$\blacksquare$} 	& $\couple{b}{b}$\\
\end{tabular}}
}
\vspace{0.5cm}
\caption{A DFAO generating $f \bmod m$ for $m\in\N_{\ge 2}$.}
\label{Fig : AutomatonMod}
\end{figure}
\end{runningexample}

\begin{remark}
Under the working hypothesis and assuming that $\K$ is finite or is included in a commutative ring, Theorem~\ref{The : Regular-Automatic} combined with Theorem~\ref{The : reg-noyau-fin-eng} provides us with another proof of the fact that the finiteness of $\ker_{\bS}(f)$ implies that $f$ is $\bS$-automatic. Indeed, let $f\colon\N^d\to\K$ and suppose that $\ker_{\bS}(f)$ is finite. Then $f$ is $(\bS,\K)$-regular by Theorem~\ref{The : reg-noyau-fin-eng}. Now, suppose towards a contradiction that $f$ takes infinitely many values. Let $(\bn_i)_{i\in\N} \in (\N^d)^\N$ be such that for all distinct $i,j\in \N$, $f(\bn_i)\ne f(\bn_j)$. Let $i,j\in \N$ with $i\neq j$. Since $\bL$ is prefix-closed, $\rep_{\bS}(\bz)=\varepsilon$ and we have $(f\circ\rep_{\bS}(\bn_i)) (\bz) 
=f(\bn_i)\neq f(\bn_j)
=(f\circ\rep_{\bS}(\bn_j)) (\bz)$. Since $f\circ\rep_{\bS}(\bn_i)$ and $f\circ\rep_{\bS}(\bn_j)$ both belong to $\ker_{\bS}(f)$, we conclude that $\ker_{\bS}(f)$ is infinite, a contradiction. Therefore, $f$ has a finite image and we conclude by Theorem~\ref{The : Regular-Automatic}. 
\end{remark}

The family of $\bS$-automatic sequences is closed under projection. 

\begin{corollary}
Let $f\colon\N^d\to \K$ be an $\bS$-automatic sequence, let $i\in[\![1,d]\!]$ and let $k\in\N$. Then the $(d{-}1)$-dimensional sequence $\delta_{i,k}(f)\colon\N^{d-1}\to \K$ (defined as in Proposition~\ref{Pro : Projection-Regular}) is $\delta_i(\bS)$-automatic.
\end{corollary}

\begin{proof}
Since $f$ has a finite image and since there exist finite semirings of all sizes, we may assume that $\K$ is finite. Then the result follows from Proposition~\ref{Pro : Projection-Regular} and Theorem~\ref{The : Regular-Automatic}.
\end{proof}

\section{Enumerating $\bS$-recognizable properties of $\bS$-automatic sequences give rise to $(\bS,\N)$-regular sequences}
\label{Sec : Enumeration}

In this section, we show how some enumeration properties of $\bS$-automatic sequences give rise to $(\bS,\N)$-regular sequences. The technique that we use contains three ingredients that we describe in Subsections~\ref{Subsec : Ingredient2},~\ref{Subsec : Ingredient1} and~\ref{Subsec : Ingredient3}. In doing so, we positively answer~\cite[Problem 3.5.7]{Charlier2018}.

We focus on the semirings $\N$ and $\N_\infty=\N\cup\{\infty\}$. The sum and product over $\N$ are extended to $\N_\infty$ as follows: for all $n\in\N_\infty$, $\infty+n=n+\infty=\infty$; for all $n\in\N_\infty\setminus\{0\}$, $\infty\cdot n=n\cdot\infty=\infty$; and $\infty\cdot 0=0\cdot\infty=0$.

\subsection{First ingredient: $\bS$-recognizable enumerations of $\N^d$}
\label{Subsec : Ingredient2}

We define an enumeration $E_{\bS}\colon\N^d\to \N$ recursively as follows. We fix a total order on $\bA$ and we consider the induced radix order on $\bA^*$. Then we define a total order $<_{\bS}$ on $\N^d$ by declaring that
\[
	\forall \bm,\bn\in\N^d,\
	\bm<_{\bS} \bn \iff \rep_{\bS}(\bm)<_{\rad} \rep_{\bS}(\bn). 
\]
For all $\bn\in\N^d$, we define $E_{\bS}(\bn)=i$ if $\bn$ is the $i$-th element of $\N^d$ with respect to this total order on $\N^d$. Note that we start indexing at $i=0$, so $E_{\bS}(\bz)=0$.

\begin{runningexample}
We fix the following total order on the alphabet $\bA$:  
\[
\def\arraystretch{1.1}
	\begin{array}{c|c|c|c|c|c|c|c}
1&2&3&4&5&6&7&8\\
\hline
\couple{\#}{a} & \couple{\#}{b} &
\couple{a}{\#} & \couple{a}{a} & \couple{a}{b}		& \couple{b}{\#} & \couple{b}{a} &\couple{b}{b}.
	\end{array}
\] 
This choice induces the following order on the pairs of integers that are represented by a single letter:
\[
\def\arraystretch{1.1}
	\begin{array}{c|c|c|c|c|c|c|c}
1&2&3&4&5&6&7&8\\
\hline
\couple{0}{1} & \couple{0}{2} &
\couple{1}{0} & \couple{1}{1} & \couple{1}{2}		& \couple{2}{0} & \couple{2}{1} &\couple{2}{2}.
	\end{array}
\] 
Note that in this example, all letters in $\bA$ actually belong to the numeration language $\bL$.
In general, the list of $d$-tuples of integers represented by a single letter might be shorter than the size of the alphabet $\bA$. The corresponding radix order on the words over $\bA^*$ of length $2$ whose components belong to $\#^*a^*b^*$ is given by
\[
\def\arraystretch{1.1}
	\begin{array}{c|c|c|c|c|c|c|c|c}
	9&10&11&12&13&14&15&16&17\\
	\hline
	\vspace{.3cm}
	\couple{\#\#}{aa} & \couple{\#\#}{ab}
	& \couple{\#a}{aa} & \couple{\#a}{ab}
	& \couple{\#b}{aa} & \couple{\#b}{ab}
	& \couple{\#\#}{bb} & \couple{\#a}{bb} 			
	& \couple{\#b}{bb}     \\ 
	18&19&20&21&22&23&24&25&26\\
	\hline
	\vspace{.3cm}
	\couple{aa}{\#\#} & \couple{aa}{\#a} 
	& \couple{aa}{\#b} & \couple{ab}{\#\#}
	& \couple{ab}{\#a} & \couple{ab}{\#b}
	& \couple{aa}{aa} & \couple{aa}{ab}
	& \couple{ab}{aa} \\
	27&28&29&30&31&32&33&34&35\\
	\hline
	\couple{ab}{ab} & \couple{aa}{bb} 
	& \couple{ab}{bb} & \couple{bb}{\#\#}
	& \couple{bb}{\#a} & \couple{bb}{\#b} 
	& \couple{bb}{aa} & \couple{bb}{ab} 
	& \couple{bb}{bb}
	\end{array}
\]
This provides us with the following order on the corresponding pairs of integers:
\[
\def\arraystretch{1.1}
	\begin{array}{c|c|c|c|c|c|c|c|c}
	9&10&11&12&13&14&15&16&17\\
	\hline
	\vspace{.3cm}
	\couple{0}{3} & \couple{0}{4}
	& \couple{1}{3} & \couple{1}{4}
	& \couple{2}{3} & \couple{2}{4}
	& \couple{0}{5} & \couple{1}{5} 			
	& \couple{2}{5}     \\ 
	18&19&20&21&22&23&24&25&26\\
	\hline
	\vspace{.3cm}
	\couple{3}{0} & \couple{3}{1} 
	& \couple{3}{2} & \couple{4}{0}
	& \couple{4}{1} & \couple{4}{2}
	& \couple{3}{3} & \couple{3}{4}
	& \couple{4}{3} \\
	27&28&29&30&31&32&33&34&35\\
	\hline
	\couple{4}{4} & \couple{3}{5} 
	& \couple{4}{5} & \couple{5}{0}
	& \couple{5}{1} & \couple{5}{2} 
	& \couple{5}{3} & \couple{5}{4} 
	& \couple{5}{5}
	\end{array}
\]
The first values of the enumeration $E_{\bS}$ are represented in Figure~\ref{Fig : Enumeration-a*b*}.
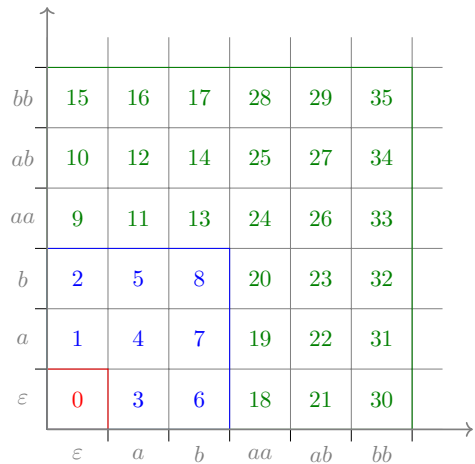
\begin{figure}[htb]
\begin{center}
\scalebox{0.8}{
\begin{tikzpicture}
\draw[step=1.0,gray,thin] (0,0) grid (6.5,6.5);

\draw[red] (0,0) rectangle (1,1);
\node[red] at (0.5,0.5) {$0$};

\draw[blue] (0,0) rectangle (3,3);
\node[blue] at (0.5,1.5) {$1$};
\node[blue] at (0.5,2.5) {$2$};
\node[blue] at (1.5,0.5) {$3$};
\node[blue] at (1.5,1.5) {$4$};
\node[blue] at (1.5,2.5) {$5$};
\node[blue] at (2.5,0.5) {$6$};
\node[blue] at (2.5,1.5) {$7$};
\node[blue] at (2.5,2.5) {$8$};

\draw[vert] (0,0) rectangle (6,6);
\node[vert] at (0.5,3.5) {$9$};
\node[vert] at (0.5,4.5) {$10$};
\node[vert] at (1.5,3.5) {$11$};
\node[vert] at (1.5,4.5) {$12$};
\node[vert] at (2.5,3.5) {$13$};
\node[vert] at (2.5,4.5) {$14$};
\node[vert] at (0.5,5.5) {$15$};
\node[vert] at (1.5,5.5) {$16$};
\node[vert] at (2.5,5.5) {$17$};
\node[vert] at (3.5,0.5) {$18$};
\node[vert] at (3.5,1.5) {$19$};
\node[vert] at (3.5,2.5) {$20$};
\node[vert] at (4.5,0.5) {$21$};
\node[vert] at (4.5,1.5) {$22$};
\node[vert] at (4.5,2.5) {$23$};
\node[vert] at (3.5,3.5) {$24$};
\node[vert] at (3.5,4.5) {$25$};
\node[vert] at (4.5,3.5) {$26$};
\node[vert] at (4.5,4.5) {$27$};
\node[vert] at (3.5,5.5) {$28$};
\node[vert] at (4.5,5.5) {$29$};
\node[vert] at (5.5,0.5) {$30$};
\node[vert] at (5.5,1.5) {$31$};
\node[vert] at (5.5,2.5) {$32$};
\node[vert] at (5.5,3.5) {$33$};
\node[vert] at (5.5,4.5) {$34$};
\node[vert] at (5.5,5.5) {$35$};

\foreach \x in {1,...,6} 
\draw[black] ($(\x,0)$) -- ($(\x,-0.2)$);

\node[gray] at ($(0.5,-0.4)$) {$\varepsilon$};
\node[gray] at ($(1.5,-0.4)$) {$a$};
\node[gray] at ($(2.5,-0.4)$) {$b$};
\node[gray] at ($(3.5,-0.4)$) {$aa$};
\node[gray] at ($(4.5,-0.4)$) {$ab$};
\node[gray] at ($(5.5,-0.4)$) {$bb$};

\foreach \y in {1,...,6} 
\draw[black] ($(0,\y)$) -- ($(-0.2,\y)$);

\node[gray] at ($(-0.4,0.5)$) {$\varepsilon$};
\node[gray] at ($(-0.4,1.5)$) {$a$};
\node[gray] at ($(-0.4,2.5)$) {$b$};
\node[gray] at ($(-0.4,3.5)$) {$aa$};
\node[gray] at ($(-0.4,4.5)$) {$ab$};
\node[gray] at ($(-0.4,5.5)$) {$bb$};

\draw[->,gray,thick] (0,0) -- (7,0);
\draw[->,gray,thick] (0,0) -- (0,7);
\end{tikzpicture}
}
\end{center}
\caption{The enumeration $E_{(\mathcal{S},\mathcal{S})}$ for $\mathcal{S}=(a^*b^*,a<b)$.}
\label{Fig : Enumeration-a*b*}
\end{figure}
\end{runningexample}

\begin{example} 
Consider the following order on the alphabet $\bA=\{\#,0,1\}^2\setminus\{\tiny\couple{\#}{\#}\}$:
\[
	\couple{\#}{0}<\couple{\#}{1}
	<\couple{0}{\#}<\couple{0}{0}<\couple{0}{1}
	<\couple{1}{\#}<\couple{1}{0}<\couple{1}{1}.
\]
Then the corresponding enumerations $E_{(\mathcal{S}_2,\mathcal{S}_2)}$ and $E_{(\mathcal{S}_F,\mathcal{S}_F)}$, where $\mathcal{S}_2=(1\{0,1\}^*\cup\{\varepsilon\},0<1)$ is the binary numeration system and $\mathcal{S}_F=(1\{0,01\}^*\cup\{\varepsilon\},0<1)$ is the Zeckendorf numeration system, are illustrated in Figure~\ref{Fig : EnumerationFibo}. 
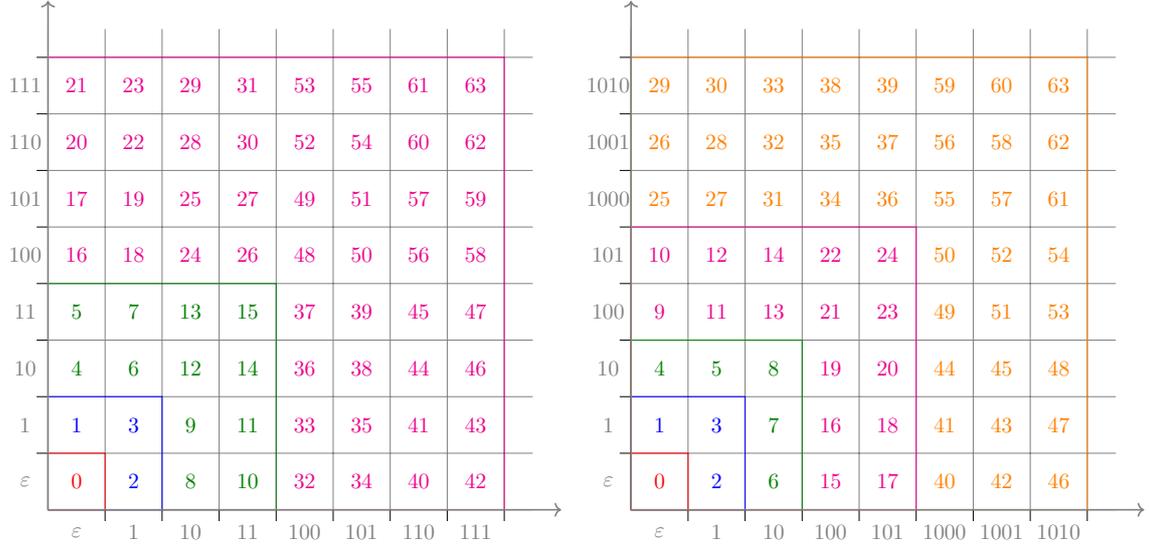
\begin{figure}[htb]
\begin{minipage}{0.5\textwidth}
\scalebox{0.75}{
\begin{tikzpicture}
\draw[step=1.0,gray,thin] (0,0) grid (8.5,8.5);

\draw[red] (0,0) rectangle (1,1);
\node[red] at (0.5,0.5) {$0$};

\draw[blue] (0,0) rectangle (2,2);
\node[blue] at (0.5,1.5) {$1$};
\node[blue] at (1.5,0.5) {$2$};
\node[blue] at (1.5,1.5) {$3$};

\draw[vert] (0,0) rectangle (4,4);
\node[vert] at (0.5,2.5) {$4$};
\node[vert] at (0.5,3.5) {$5$};
\node[vert] at (1.5,2.5) {$6$};
\node[vert] at (1.5,3.5) {$7$};
\node[vert] at (2.5,0.5) {$8$};
\node[vert] at (2.5,1.5) {$9$};
\node[vert] at (3.5,0.5) {$10$};
\node[vert] at (3.5,1.5) {$11$};
\node[vert] at (2.5,2.5) {$12$};
\node[vert] at (2.5,3.5) {$13$};
\node[vert] at (3.5,2.5) {$14$};
\node[vert] at (3.5,3.5) {$15$};

\draw[magenta] (0,0) rectangle (8,8);
\node[magenta] at (0.5,4.5) {$16$};
\node[magenta] at (0.5,5.5) {$17$};
\node[magenta] at (1.5,4.5) {$18$};
\node[magenta] at (1.5,5.5) {$19$};
\node[magenta] at (0.5,6.5) {$20$};
\node[magenta] at (0.5,7.5) {$21$};
\node[magenta] at (1.5,6.5) {$22$};
\node[magenta] at (1.5,7.5) {$23$};
\node[magenta] at (2.5,4.5) {$24$};
\node[magenta] at (2.5,5.5) {$25$};
\node[magenta] at (3.5,4.5) {$26$};
\node[magenta] at (3.5,5.5) {$27$};
\node[magenta] at (2.5,6.5) {$28$};
\node[magenta] at (2.5,7.5) {$29$};
\node[magenta] at (3.5,6.5) {$30$};
\node[magenta] at (3.5,7.5) {$31$};
\node[magenta] at (4.5,0.5) {$32$};
\node[magenta] at (4.5,1.5) {$33$};
\node[magenta] at (5.5,0.5) {$34$};
\node[magenta] at (5.5,1.5) {$35$};
\node[magenta] at (4.5,2.5) {$36$};
\node[magenta] at (4.5,3.5) {$37$};
\node[magenta] at (5.5,2.5) {$38$};
\node[magenta] at (5.5,3.5) {$39$};
\node[magenta] at (6.5,0.5) {$40$};
\node[magenta] at (6.5,1.5) {$41$};
\node[magenta] at (7.5,0.5) {$42$};
\node[magenta] at (7.5,1.5) {$43$};
\node[magenta] at (6.5,2.5) {$44$};
\node[magenta] at (6.5,3.5) {$45$};
\node[magenta] at (7.5,2.5) {$46$};
\node[magenta] at (7.5,3.5) {$47$};
\node[magenta] at (4.5,4.5) {$48$};
\node[magenta] at (4.5,5.5) {$49$};
\node[magenta] at (5.5,4.5) {$50$};
\node[magenta] at (5.5,5.5) {$51$};
\node[magenta] at (4.5,6.5) {$52$};
\node[magenta] at (4.5,7.5) {$53$};
\node[magenta] at (5.5,6.5) {$54$};
\node[magenta] at (5.5,7.5) {$55$};
\node[magenta] at (6.5,4.5) {$56$};
\node[magenta] at (6.5,5.5) {$57$};
\node[magenta] at (7.5,4.5) {$58$};
\node[magenta] at (7.5,5.5) {$59$};
\node[magenta] at (6.5,6.5) {$60$};
\node[magenta] at (6.5,7.5) {$61$};
\node[magenta] at (7.5,6.5) {$62$};
\node[magenta] at (7.5,7.5) {$63$};

\foreach \x in {1,...,8} 
\draw[black] ($(\x,0)$) -- ($(\x,-0.2)$);

\node[gray] at ($(0.5,-0.4)$) {$\varepsilon$};
\node[gray] at ($(1.5,-0.4)$) {$1$};
\node[gray] at ($(2.5,-0.4)$) {$10$};
\node[gray] at ($(3.5,-0.4)$) {$11$};
\node[gray] at ($(4.5,-0.4)$) {$100$};
\node[gray] at ($(5.5,-0.4)$) {$101$};
\node[gray] at ($(6.5,-0.4)$) {$110$};
\node[gray] at ($(7.5,-0.4)$) {$111$};

\foreach \y in {1,...,8} 
\draw[black] ($(0,\y)$) -- ($(-0.2,\y)$);

\node[gray] at ($(-0.4,0.5)$) {$\varepsilon$};
\node[gray] at ($(-0.4,1.5)$) {$1$};
\node[gray] at ($(-0.4,2.5)$) {$10$};
\node[gray] at ($(-0.4,3.5)$) {$11$};
\node[gray] at ($(-0.4,4.5)$) {$100$};
\node[gray] at ($(-0.4,5.5)$) {$101$};
\node[gray] at ($(-0.4,6.5)$) {$110$};
\node[gray] at ($(-0.4,7.5)$) {$111$};

\draw[->,gray,thick] (0,0) -- (9,0);
\draw[->,gray,thick] (0,0) -- (0,9);
\end{tikzpicture}
}
\end{minipage}
\begin{minipage}{0.5\textwidth}
\scalebox{0.75}{
\begin{tikzpicture}
\draw[step=1.0,gray,thin] (0,0) grid (8.5,8.5);

\draw[red] (0,0) rectangle (1,1);
\node[red] at (0.5,0.5) {$0$};

\draw[blue] (0,0) rectangle (2,2);
\node[blue] at (0.5,1.5) {$1$};
\node[blue] at (1.5,0.5) {$2$};
\node[blue] at (1.5,1.5) {$3$};

\draw[vert] (0,0) rectangle (3,3);
\node[vert] at (0.5,2.5) {$4$};
\node[vert] at (1.5,2.5) {$5$};
\node[vert] at (2.5,0.5) {$6$};
\node[vert] at (2.5,1.5) {$7$};
\node[vert] at (2.5,2.5) {$8$};

\draw[magenta] (0,0) rectangle (5,5);
\node[magenta] at (0.5,3.5) {$9$};
\node[magenta] at (0.5,4.5) {$10$};
\node[magenta] at (1.5,3.5) {$11$};
\node[magenta] at (1.5,4.5) {$12$};
\node[magenta] at (2.5,3.5) {$13$};
\node[magenta] at (2.5,4.5) {$14$};
\node[magenta] at (3.5,0.5) {$15$};
\node[magenta] at (3.5,1.5) {$16$};
\node[magenta] at (4.5,0.5) {$17$};
\node[magenta] at (4.5,1.5) {$18$};
\node[magenta] at (3.5,2.5) {$19$};
\node[magenta] at (4.5,2.5) {$20$};
\node[magenta] at (3.5,3.5) {$21$};
\node[magenta] at (3.5,4.5) {$22$};
\node[magenta] at (4.5,3.5) {$23$};
\node[magenta] at (4.5,4.5) {$24$};

\draw[orange] (0,0) rectangle (8,8);
\node[orange] at (0.5,5.5) {$25$};
\node[orange] at (0.5,6.5) {$26$};
\node[orange] at (1.5,5.5) {$27$};
\node[orange] at (1.5,6.5) {$28$};
\node[orange] at (0.5,7.5) {$29$};
\node[orange] at (1.5,7.5) {$30$};
\node[orange] at (2.5,5.5) {$31$};
\node[orange] at (2.5,6.5) {$32$};
\node[orange] at (2.5,7.5) {$33$};
\node[orange] at (3.5,5.5) {$34$};
\node[orange] at (3.5,6.5) {$35$};
\node[orange] at (4.5,5.5) {$36$};
\node[orange] at (4.5,6.5) {$37$};
\node[orange] at (3.5,7.5) {$38$};
\node[orange] at (4.5,7.5) {$39$};
\node[orange] at (5.5,0.5) {$40$};
\node[orange] at (5.5,1.5) {$41$};
\node[orange] at (6.5,0.5) {$42$};
\node[orange] at (6.5,1.5) {$43$};
\node[orange] at (5.5,2.5) {$44$};
\node[orange] at (6.5,2.5) {$45$};
\node[orange] at (7.5,0.5) {$46$};
\node[orange] at (7.5,1.5) {$47$};
\node[orange] at (7.5,2.5) {$48$};
\node[orange] at (5.5,3.5) {$49$};
\node[orange] at (5.5,4.5) {$50$};
\node[orange] at (6.5,3.5) {$51$};
\node[orange] at (6.5,4.5) {$52$};
\node[orange] at (7.5,3.5) {$53$};
\node[orange] at (7.5,4.5) {$54$};
\node[orange] at (5.5,5.5) {$55$};
\node[orange] at (5.5,6.5) {$56$};
\node[orange] at (6.5,5.5) {$57$};
\node[orange] at (6.5,6.5) {$58$};
\node[orange] at (5.5,7.5) {$59$};
\node[orange] at (6.5,7.5) {$60$};
\node[orange] at (7.5,5.5) {$61$};
\node[orange] at (7.5,6.5) {$62$};
\node[orange] at (7.5,7.5) {$63$};

\foreach \x in {1,...,8} 
\draw[black] ($(\x,0)$) -- ($(\x,-0.2)$);

\node[gray] at ($(0.5,-0.4)$) {$\varepsilon$};
\node[gray] at ($(1.5,-0.4)$) {$1$};
\node[gray] at ($(2.5,-0.4)$) {$10$};
\node[gray] at ($(3.5,-0.4)$) {$100$};
\node[gray] at ($(4.5,-0.4)$) {$101$};
\node[gray] at ($(5.5,-0.4)$) {$1000$};
\node[gray] at ($(6.5,-0.4)$) {$1001$};
\node[gray] at ($(7.5,-0.4)$) {$1010$};

\foreach \y in {1,...,8} 
\draw[black] ($(0,\y)$) -- ($(-0.2,\y)$);

\node[gray] at ($(-0.4,0.5)$) {$\varepsilon$};
\node[gray] at ($(-0.4,1.5)$) {$1$};
\node[gray] at ($(-0.4,2.5)$) {$10$};
\node[gray] at ($(-0.4,3.5)$) {$100$};
\node[gray] at ($(-0.4,4.5)$) {$101$};
\node[gray] at ($(-0.4,5.5)$) {$1000$};
\node[gray] at ($(-0.4,6.5)$) {$1001$};
\node[gray] at ($(-0.4,7.5)$) {$1010$};

\draw[->,gray,thick] (0,0) -- (9,0);
\draw[->,gray,thick] (0,0) -- (0,9);
\end{tikzpicture}
}
\end{minipage}
\caption{The enumeration $E_{(\mathcal{S}_2,\mathcal{S}_2)}$ on the left and the enumeration $E_{(\mathcal{S}_F,\mathcal{S}_F)}$ on the right.}
\label{Fig : EnumerationFibo}
\end{figure}
We can also mix both systems and work with the numeration system $(\mathcal{S}_2,\mathcal{S}_F)$. The corresponding mixed enumeration $E_{(\mathcal{S}_2,\mathcal{S}_F)}$ is depicted in the left part of Figure~\ref{Fig : EnumerationMixed}. Similarly, the right part of Figure~\ref{Fig : EnumerationMixed} corresponds to the mixed enumeration $E_{(\mathcal{S},\mathcal{S}_F)}$, where $\mathcal{S}$ is the abstract numeration system of the running example.
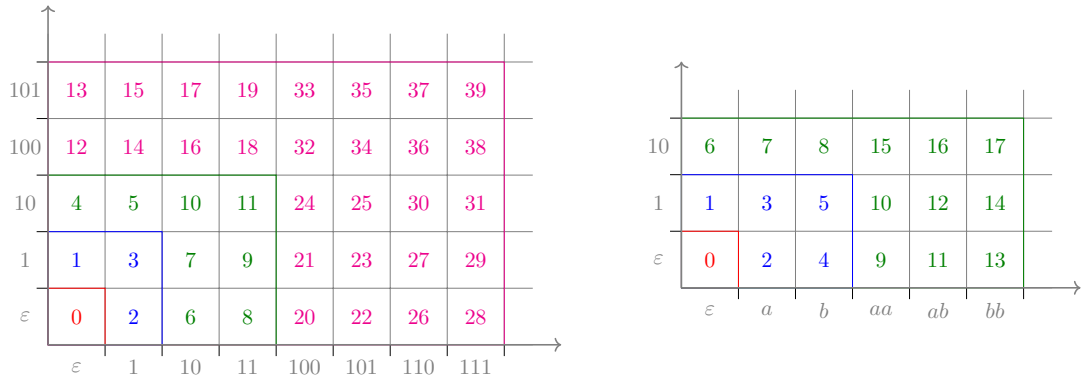
\begin{figure}[htb]
\begin{minipage}{0.5\textwidth}
\scalebox{0.75}{
\begin{tikzpicture}
\draw[step=1.0,gray,thin] (0,0) grid (8.5,5.5);

\draw[red] (0,0) rectangle (1,1);
\node[red] at (0.5,0.5) {$0$};

\draw[blue] (0,0) rectangle (2,2);
\node[blue] at (0.5,1.5) {$1$};
\node[blue] at (1.5,0.5) {$2$};
\node[blue] at (1.5,1.5) {$3$};

\draw[vert] (0,0) rectangle (4,3);
\node[vert] at (0.5,2.5) {$4$};
\node[vert] at (1.5,2.5) {$5$};
\node[vert] at (2.5,0.5) {$6$};
\node[vert] at (2.5,1.5) {$7$};
\node[vert] at (3.5,0.5) {$8$};
\node[vert] at (3.5,1.5) {$9$};
\node[vert] at (2.5,2.5) {$10$};
\node[vert] at (3.5,2.5) {$11$};

\draw[magenta] (0,0) rectangle (8,5);
\node[magenta] at (0.5,3.5) {$12$};
\node[magenta] at (0.5,4.5) {$13$};
\node[magenta] at (1.5,3.5) {$14$};
\node[magenta] at (1.5,4.5) {$15$};
\node[magenta] at (2.5,3.5) {$16$};
\node[magenta] at (2.5,4.5) {$17$};
\node[magenta] at (3.5,3.5) {$18$};
\node[magenta] at (3.5,4.5) {$19$};
\node[magenta] at (4.5,0.5) {$20$};
\node[magenta] at (4.5,1.5) {$21$};
\node[magenta] at (5.5,0.5) {$22$};
\node[magenta] at (5.5,1.5) {$23$};
\node[magenta] at (4.5,2.5) {$24$};
\node[magenta] at (5.5,2.5) {$25$};
\node[magenta] at (6.5,0.5) {$26$};
\node[magenta] at (6.5,1.5) {$27$};
\node[magenta] at (7.5,0.5) {$28$};
\node[magenta] at (7.5,1.5) {$29$};
\node[magenta] at (6.5,2.5) {$30$};
\node[magenta] at (7.5,2.5) {$31$};
\node[magenta] at (4.5,3.5) {$32$};
\node[magenta] at (4.5,4.5) {$33$};
\node[magenta] at (5.5,3.5) {$34$};
\node[magenta] at (5.5,4.5) {$35$};
\node[magenta] at (6.5,3.5) {$36$};
\node[magenta] at (6.5,4.5) {$37$};
\node[magenta] at (7.5,3.5) {$38$};
\node[magenta] at (7.5,4.5) {$39$};

\foreach \x in {1,...,8} 
\draw[black] ($(\x,0)$) -- ($(\x,-0.2)$);

\node[gray] at ($(0.5,-0.4)$) {$\varepsilon$};
\node[gray] at ($(1.5,-0.4)$) {$1$};
\node[gray] at ($(2.5,-0.4)$) {$10$};
\node[gray] at ($(3.5,-0.4)$) {$11$};
\node[gray] at ($(4.5,-0.4)$) {$100$};
\node[gray] at ($(5.5,-0.4)$) {$101$};
\node[gray] at ($(6.5,-0.4)$) {$110$};
\node[gray] at ($(7.5,-0.4)$) {$111$};

\foreach \y in {1,...,5} 
\draw[black] ($(0,\y)$) -- ($(-0.2,\y)$);

\node[gray] at ($(-0.4,0.5)$) {$\varepsilon$};
\node[gray] at ($(-0.4,1.5)$) {$1$};
\node[gray] at ($(-0.4,2.5)$) {$10$};
\node[gray] at ($(-0.4,3.5)$) {$100$};
\node[gray] at ($(-0.4,4.5)$) {$101$};

\draw[->,gray,thick] (0,0) -- (9,0);
\draw[->,gray,thick] (0,0) -- (0,6);
\end{tikzpicture}
}
\end{minipage}
\qquad
\begin{minipage}{0.3\textwidth}
\scalebox{0.75}{
\begin{tikzpicture}
\draw[step=1.0,gray,thin] (0,0) grid (6.5,3.5);

\draw[red] (0,0) rectangle (1,1);
\node[red] at (0.5,0.5) {$0$};

\draw[blue] (0,0) rectangle (3,2);
\node[blue] at (0.5,1.5) {$1$};
\node[blue] at (1.5,0.5) {$2$};
\node[blue] at (1.5,1.5) {$3$};
\node[blue] at (2.5,0.5) {$4$};
\node[blue] at (2.5,1.5) {$5$};

\draw[vert] (0,0) rectangle (6,3);
\node[vert] at (0.5,2.5) {$6$};
\node[vert] at (1.5,2.5) {$7$};
\node[vert] at (2.5,2.5) {$8$};
\node[vert] at (3.5,0.5) {$9$};
\node[vert] at (3.5,1.5) {$10$};
\node[vert] at (4.5,0.5) {$11$};
\node[vert] at (4.5,1.5) {$12$};
\node[vert] at (5.5,0.5) {$13$};
\node[vert] at (5.5,1.5) {$14$};
\node[vert] at (3.5,2.5) {$15$};
\node[vert] at (4.5,2.5) {$16$};
\node[vert] at (5.5,2.5) {$17$};

\foreach \x in {1,...,6} 
\draw[black] ($(\x,0)$) -- ($(\x,-0.2)$);

\node[gray] at ($(0.5,-0.4)$) {$\varepsilon$};
\node[gray] at ($(1.5,-0.4)$) {$a$};
\node[gray] at ($(2.5,-0.4)$) {$b$};
\node[gray] at ($(3.5,-0.4)$) {$aa$};
\node[gray] at ($(4.5,-0.4)$) {$ab$};
\node[gray] at ($(5.5,-0.4)$) {$bb$};

\foreach \y in {1,...,3} 
\draw[black] ($(0,\y)$) -- ($(-0.2,\y)$);

\node[gray] at ($(-0.4,0.5)$) {$\varepsilon$};
\node[gray] at ($(-0.4,1.5)$) {$1$};
\node[gray] at ($(-0.4,2.5)$) {$10$};

\draw[->,gray,thick] (0,0) -- (7,0);
\draw[->,gray,thick] (0,0) -- (0,4);
\end{tikzpicture}
}
\end{minipage}
\caption{The enumeration $E_{(\mathcal{S}_2,\mathcal{S}_F)}$ on the left and the enumeration $E_{(\mathcal{S},\mathcal{S}_F)}$.}
\label{Fig : EnumerationMixed}
\end{figure}
\end{example}

\begin{proposition}
\label{Pro : AutEnumeration}
For each $\diamond\in\{=,>,<\}$, the set
$\{\couple{\bm}{\bn}\in\N^{2d}\colon 					E_{\bS}(\bm) \diamond E_{\bS}(\bn)\}$ is $(\bS,\bS)$-recognizable. 
\end{proposition}

\begin{proof}
We extend the total order on $\bA$ to a total order on $\bA_{\bdiese}$ by declaring the letter $\bdiese$ to be less than all the letters in $\bA$. We can then consider the induced radix order on $(\bA_{\bdiese})^*$.
For each $\diamond\in\{=,>,<\}$, the language 
$R:=\{\couple{\bu}{\bv}\in (\bA_{\bdiese}\times\bA_{\bdiese})^*\colon  
	\bu\diamond_{\rad}\bv\}$
is regular as a DFA $\mathcal{A}_{\diamond}$ accepting this language is obtained as follows. The set of states is $\{q_=,q_>,q_<\}$. The initial state is $q_=$ and the only final state is $q_\diamond$. For all $\ba,\bb\in \bA_{\bdiese}$, $\delta(q_>, \couple{\ba}{\bb}) = q_>$,
$\delta(q_<, \couple{\ba}{\bb}) = q_<$ and
\[
	\delta(q_=, \couple{\ba}{\bb})
	=
	\begin{cases}
	q_= & \text{if }\ba=\bb\\ 
	q_> & \text{if }\ba>\bb\\
	q_< & \text{if } \ba<\bb.
	\end{cases}
\]
Since $\rep_{(\bS,\bS)}\big(\{\couple{\bm}{\bn}\in\N^{2d}\colon E_{\bS}(\bm) \diamond E_{\bS}(\bn)\}\big)=\rep_{(\bS,\bS)}(\N^{2d})\cap R$, the conclusion follows.
\end{proof}

\begin{remark}
If $d\ge 2$, then it is easily seen by using the pumping lemma that the subset $\{\couple{\bn}{E_{(\mathcal{S},\ldots,\mathcal{S})}(\bn)}\colon\bn\in\N^d\})$ of $\N^{d+1}$ is not $((\mathcal{S},\ldots,\mathcal{S}),\mathcal{S})$-recognizable for any abstract numeration system $\mathcal{S}$. 
\end{remark}

The enumeration $E_{\bS}$ was defined recursively and depends on the chosen order on $\bA$. In the case of integer base numeration systems, we are able to give a closed formula. Let $b\in\N_{\ge 2}$, let
\[
	\mathcal{S}_b=(\{1,\ldots,b-1\}\{0,\ldots,b-1\}^*\cup\{\varepsilon\},0<\cdots<b-1)
\]
and let $\bS_{b,d}=(\mathcal{S}_b,\ldots,\mathcal{S}_b)$ be the $d$-dimensional abstract numeration system made of $d$ copies of $\mathcal{S}_b$. Suppose that $\#<0$ and consider the total order on the alphabet $\{\#,0,\ldots,b-1\}^d\setminus\{\bdiese\}$ induced by the lexicographic order on the components, i.e.,
\begin{equation}
\label{Eq : LexOrder}
	\duple{a_1}{a_d}< \duple{b_1}{b_d}\iff
	\exists i\in[\![1,d]\!],\ a_i<b_i \text{ and } 
	\forall j\in[\![1,i-1]\!],\ a_j=b_j.
\end{equation}

Let us introduce some extra notation. First, if for each $i\in[\![1,d]\!]$, $w_i$ is a word of length $\ell$, then the \emph{perfect shuffle} of $w_1,\ldots,w_d$ is the word of length $\ell d$ given by
\[
	\Sh\duple{w_1}{w_d}=\prod_{j=1}^{\ell}(w_1[j]\cdots w_d[j]).
\]

\begin{example}
We have $\Sh\triple{ab}{cd}{ef}=acebdf$.
\end{example}

We define the function $\val_b$ as the usual $b$-value function:
\[
	\val_b\colon \{0,\ldots,b-1\}^*\to\N,\
	w\mapsto \sum_{j=1}^{|w|} w[j]b^{|w|-j}.
\]
Note that the functions $\val_b$ and $\val_{\mathcal{S}_b}$ coincide on words not starting with the letter $0$. 
 
\begin{proposition}
For all $\bn\in\N^d$, $E_{\bS_{b,d}}(\bn)=
	\val_b \big(\sigma_{\#,0}(\Sh(\rep_{\bS_{b,d}}(\bn)))\big)$ where $\sigma_{\#,0}$ is the morphism that replaces $\#$ by $0$ and leaves the other letters unchanged.  
\end{proposition}

In the previous formula, the inserted morphism $\sigma_{\#,0}$ could be removed if we had taken the convention to pad the shortest representations with the letter $0$ instead of $\#$ as is usually done in integer base numeration systems.

\begin{proof}
First, we prove that for all $\bm,\bn\in\N^d$,
$\rep_{\bS_{b,d}}(\bm)<_{\rad}\rep_{\bS_{b,d}}(\bn)$
implies that 
\begin{equation}
\label{Eq : Val_b}
	\val_b \big(\sigma_{\#,0}(\Sh(\rep_{\bS_{b,d}}(\bm)))\big)
	<\val_b \big(\sigma_{\#,0}(\Sh(\rep_{\bS_{b,d}}(\bn)))\big).
\end{equation}
Let $\bm,\bn\in\N^d$. Observe that the first letters of both $\rep_{\bS_{b,d}}(\bm)$ and $\rep_{\bS_{b,d}}(\bm)$ have at least a component in $\{1,\ldots,b-1\}$. Therefore, the first length-$d$ blocks of both $\sigma_{\#,0}(\Sh(\rep_{\bS_{b,d}}(\bm)))$ and $\sigma_{\#,0}(\Sh(\rep_{\bS_{b,d}}(\bm)))$ differ from $0^d$. 

Suppose that $\rep_{\bS_{b,d}}(\bm)<_{\rad}\rep_{\bS_{b,d}}(\bn)$, and let $k=|\rep_{\bS_{b,d}}(\bm)|$ and $\ell=|\rep_{\bS_{b,d}}(\bn)|$. We get from the previous observation that $\val_b\big(\sigma_{\#,0}(\Sh(\rep_{\bS_{b,d}}(\bm)))\big)$ belongs to the interval $[\![b^{(k-1)d},b^{kd}-1]\!]$ whereas $\val_b\big(\sigma_{\#,0}(\Sh(\rep_{\bS_{b,d}}(\bn)))\big)$ is in $[\![b^{(\ell-1)d},b^{\ell d}-1]\!]$. This proves~\eqref{Eq : Val_b} if $k<\ell$. Now assume that $k=\ell$. Then we have
$|\sigma_{\#,0}(\Sh(\rep_{\bS_{b,d}}(\bm)))|=
|\sigma_{\#,0}(\Sh(\rep_{\bS_{b,d}}(\bn)))|$. By choice of the order~\eqref{Eq : LexOrder}, we get that
\[
	\sigma_{\#,0}(\Sh(\rep_{\bS_{b,d}}(\bm)))			
	<_{\lex}	\sigma_{\#,0}(\Sh(\rep_{\bS_{b,d}}(\bn))).
\]
It is classical that the function $\val_b$ respects the lexicographic order on words of the same length, i.e., for any $u,v\in\{0,\ldots,b-1\}^*$ such that $|u|=|v|$, we have $\val_b(u)<\val_b(v)\iff u<_{\lex}v$. Therefore, $\eqref{Eq : Val_b}$ holds in this case as well.

In order to conclude, it suffices to show that the map 
\[
	\N^d\to \N,\ 
	\bn\mapsto \val_b\big(\sigma_{\#,0}(\Sh(\rep_{\bS_{b,d}}(\bn)))\big)
\]
is surjective. Let $e\in \N$. Define $\ell\in \N$ as the least integer such that $\ell d\ge |\rep_{\mathcal{S}_b}(e)|$ and factorize $0^{\ell d- |\rep_{\mathcal{S}_b}(e)|} \rep_{\mathcal{S}_b}(e)= w_1 \cdots w_{\ell}$ where each factor $w_j$ has length $d$. Then, for all $i\in[\![1,d]\!]$, we define $n_i= \val_b(w_1[i] \cdots w_{\ell}[i])$. By choice of $\ell$, there exists $i\in[\![1,d]\!]$ such that $w_1[i]\ne 0$, and hence such that $\rep_{\mathcal{S}_b}(n_i)=w_1[i] \cdots w_{\ell}[i]$. We obtain
\[
	\val_b\Big(\sigma_{\#,0}\Big(\Sh\Big(\rep_{\bS_{b,d}}\duple{n_1}{n_d}\Big)\Big)\Big)
	=\val_b\Big(\Sh\duple{w_1[1] \cdots w_{\ell}[1]}{w_1[d] \cdots w_{\ell}[d]}\Big)
	=\val_b(w_1 \cdots w_{\ell})
	=e. 
\]
\end{proof}

\begin{example} 
We have 
$E_{(\mathcal{S}_2,\mathcal{S}_2)}\couple{1}{6}
=\val_2 \big(\sigma_{\#,0}\big(\Sh\couple{\#\#1}{110}\big)\big)
=\val_2(\sigma_{\#,0}(\#1\#110))
=\val_2(010110)
=22$, which indeed corresponds to the value found in the left part of Figure~\ref{Fig : EnumerationFibo}.
\end{example}

\subsection{Second ingredient: generating $(\bS,\N)$-regular sequences from $(\bS,\bS')$-recognizable sets}
\label{Subsec : Ingredient1}

We first recall the following two results on formal series from \cite{CharlierRampersadShallit2012}; also see \cite{Charlier2018} for a survey.

\begin{proposition}
\label{Pro :  bounded-N_infty-recognizable}
If a series $S\colon A^*\to\N$ is $\N_{\infty}$-recognizable, then it is $\N$-recognizable.
\end{proposition}

\begin{theorem}
\label{The : S-N-infty-recognizable}
Let $S\colon A^*\to\N_\infty$. The following assertions are equivalent.
\begin{enumerate}
\item The series $S$ is $\N_\infty$-recognizable.
\item There exists a regular language $L\subseteq (A_\$\times \Delta)^*$ (where $\$\notin A$ and $\Delta$ is a finite alphabet) such that for all $w\in A^+$, $(S,w)=\Card\{z\in L\colon \tau_\$(\pi_1(z))=w\}$, where $\pi_1$ is the projection onto the first component.
\end{enumerate}
\end{theorem}

In what follows, we sometimes consider an extra abstract numeration system $\bS'$ of dimension $d'$. The notation of Section~\ref{Sec : Preliminaries} are extended to this context in a natural manner. This theorem can be seen as a generalization of \cite[Theorem 3.4.15]{Charlier2018} to abstract numeration systems. Note that here, we only use the notion of recognizability of sets of vectors of integers (see Section~\ref{Sec : NumSys}), whereas the notion of definability of such sets was used in \cite{Charlier2018}. 

\begin{theorem}
\label{The : enumeration}
If $X$ is a $(\bS,\bS')$-recognizable subset of $\N^{d+d'}$, then the sequence 
\begin{equation}
\label{Eq : f-Card}
f\colon\N^d\to\N_{\infty},\ \bn\mapsto \Card\{\bn'\in\N^{d'}\colon \couple{\bn}{\bn'}\in X\}
\end{equation}
is $(\bS,\N_{\infty})$-regular. If moreover $f(\N)\subseteq\N$ then $f$ is $(\bS,\N)$-regular.
\end{theorem}

\begin{proof}
Let $X$ be an $(\bS,\bS')$-recognizable subset of $\N^{d+d'}$ and let $f\colon\N^d\to\N_{\infty}$ be defined as in~\eqref{Eq : f-Card}. Then for all $\bn\in\N^d$, 
\[
	f(\bn)
	=\Card\{\boldsymbol{z}\in \rep_{(\bS,\bS')}(X)
	\colon\pi_1(\boldsymbol{z})\in \bdiese^*\rep_{\bS}(\bn)\}
\]
where $\pi_1$ denotes the projection onto the first $d$ components. Consider the series
\[
	S\colon \bA^* \to\N_{\infty},\ 
	\bw\mapsto 
	\Card\{\boldsymbol{z}\in \rep_{(\bS,\bS')}(X)	
	\colon\pi_1(\boldsymbol{z})\in \bdiese^*\bw\}.
\] 
For all $\bn\in\N^d$, we have $(S,\rep_{\bS}(\bn))=f(\bn)$. Therefore, $S_f=S\odot \underline{\bL}$. Since $\bL$ is a regular language, by Proposition~\ref{Pro : HadamardReg}, in order to obtain that the sequence $f$ is $(\bS,\N_{\infty})$-regular, it suffices to show that the series $S$ is $\N_{\infty}$-recognizable. For all $\bw\in \bA^*$ and $\boldsymbol{z}\in \rep_{(\bS,\bS')}(X)$, we have $\pi_1(\boldsymbol{z})\in \bdiese^*\bw$ if and only if $\tau_{\bdiese}(\pi_1(\boldsymbol{z}))=\bw$. Since $\rep_{(\bS,\bS')}(X)$ is a regular language, we obtain that $S$ is $\N_{\infty}$-recognizable by using Theorem~\ref{The : S-N-infty-recognizable}.

The fact that $f$ is $(\bS,\N)$-regular if $f(\N)\subseteq\N$ follows from Proposition~\ref{Pro :  bounded-N_infty-recognizable}.
\end{proof}

\subsection{Third ingredient: building $\bS$-recognizable predicates}
\label{Subsec : Ingredient3}

For any $m\in\N$, we let $\bS^m$ denote the $md$-dimensional numeration system $(\bS,\ldots,\bS)$ (where $\bS$ is repeated $m$ times) and we say that a predicate $P$ on $\N^{md}$ is \emph{$\bS$-recognizable} if its characteristic set 
\[
	\Big\{\duple{\bn_1}{\bn_m}\in\N^{md}
	\colon P(\bn_1,\ldots,\bn_m) \text{ is true}\Big\}
\] 
is $\bS^m$-recognizable. 
We let $\bx=\by$ and $\bx< \by$ denote the $2d$-ary predicates $(x_1=y_1)\land \cdots\land (x_d=y_d)$ and $(x_1< y_1)\land \cdots\land (x_d< y_d)$ respectively. These predicates are always $\bS$-recognizable since the languages
\[
	\rep_{(\bS,\bS)}
	\{\couple{\bn}{\bn}\colon \bn\in\N^d\}
	=\{\couple{\bw}{\bw}\colon \bw\in \bL\}
\]
and
\[
	\rep_{(\bS,\bS)}
	\{\couple{\bm}{\bn}\in\N^{2d}\colon \bm<\bn\}
	=\{\couple{\bu}{\bv}^\#\colon \bu,\bv\in \bL,\ u_1<_{\rad}v_1,\ldots,u_d<_{\rad}v_d\}
\] 
are both regular. On the other side, addition, which corresponds to the $3d$-ary predicate $\bx+\by=\boldsymbol{z}$ is not $\bS$-recognizable in general~\cite{CharlierRigoSteiner2008,LecomteRigo2001}. Note that addition is $\bS$-recognizable if and only if it is $\mathcal{S}_i$-recognizable for every $i\in[\![1,d]\!]$.

\begin{proposition}
\label{Prop : EqualityofAutSeqisRec}
If $f$ is an $\bS$-automatic $d$-dimensional sequence, then the $2d$-ary predicate $f(\bx)=f(\by)$ is $\bS$-recognizable.
\end{proposition}

\begin{proof}
In order to get a DFA accepting $\rep_{\bS}(\{\couple{\bm}{\bn}\in\N^{2d}\colon f(\bm)=f(\bn)\})$ from a DFAO $\mathcal{A}$ computing $f$, we compute the product $\mathcal{A}\times \mathcal{A}$ reading $2d$-tuples and we declare a state $(q,q')$ to be final if the outputs of $q$ and $q'$ in $\mathcal{A}$ coincide, and then we intersect the obtained DFA with a DFA accepting $(\bL\times\bL)^\#$.
\end{proof}

For a $d$-ary predicate $P(x_1,\ldots,x_d)$, we use the notation $\forall\bx P(\bx)$ and $\exists\bx P(\bx)$ as shortcuts for $\forall x_1 \cdots \forall x_d\, P(x_1,\ldots,x_d)$ and $\exists x_1 \cdots \exists x_d\, P(x_1,\ldots,x_d)$ respectively. 

\begin{theorem}
\label{The : P-S-rec}
Any predicate on $\N^{md}$ (with $m\in\N$) that is defined recursively from $\bS$-recognizable predicates by only using the logical connectives $\land,\lor,\neg,\implies,\iff$ and the quantifiers $\forall$ and $\exists$ on variables describing elements of $\N^d$, is $\bS$-recognizable.
\end{theorem}

\begin{proof}
This is a straightforward generalization of the forward direction of the proof of~\cite[Theorem 6.1]{BruyereHanselMichauxVillemaire1994}. The only difference is when we take the negation of a $md$-ary predicate $P$, we do not simply complement the corresponding language of $\bS^m$-representations but we need to intersect with the numeration language on each component afterwards:
\begin{multline*}
	\rep_{\bS^m}
	\Big\{\duple{\bn_1}{\bn_m}\in\N^{md}
	\colon \neg P(\bn_1,\ldots,\bn_m)\text{ is true}\Big\}\\
	=(\bL^m)^\#
	\setminus \rep_{\bS^m}
	\Big\{\duple{\bn_1}{\bn_m}\in\N^{md}\colon P(\bn_1,\ldots,\bn_m)\text{ is true}\Big\}.
\end{multline*}
Since the numeration language $\bL$ is regular, when starting with an $\bS$-recognizable predicate $P$, the resulting language is regular, i.e., the predicate $\neg P$ is $\bS$-recognizable.
\end{proof}

\begin{corollary}
For any predicate $P$ on $\N^d$ that is defined recursively from $\bS$-recognizable predicates by only using the logical connectives $\land,\lor,\neg,\implies,\iff$ and the quantifiers $\forall$ and $\exists$ on variables describing elements of $\N^d$, the closed predicates $\forall \bx P(\bx)$, $\exists \bx P(\bx)$ and $\exists^{\infty} \bx P(\bx)$ are decidable.
\end{corollary}

\subsection{Application to factor complexity, and other enumeration properties,  of $\bS$-automatic sequences}
\label{Subsec : Applications}

We show how to apply the results of the previous three sections in order to obtain various families of $\bS$-regular sequences. 

For a sequence $f\colon\N^d\to \K$ and $\bp,\bs\in\N^d$, we let $f[\bp,\bs]$ denote the factor of \emph{size} $\bs$ occurring at \emph{position} $\bp$ in $f$. Formally, $f[\bp,\bs]\colon\dcart{[\![0,s_1-1]\!]}{[\![0,s_d-1]\!]}\to \K,\ \bi\mapsto f(\bp+\bi)$. The \emph{factor complexity} of a sequence $f\colon\N^d\to \K$ is the function $\rho_f\colon\N^d\mapsto\N_{\infty}$ that maps each $\bs\in\N^d$ to the number of factors of size $\bs$ occurring in $f$. Note that if $f$ has a finite image (as is the case for automatic sequences) then for all $\bs\in\N^d$, $\rho_f(\bs)\in\N$.

\begin{theorem}
\label{The : FactorComplexity}
Suppose that addition is $\bS$-recognizable. Then the factor complexity of an $\bS$-automatic sequence is a $(\bS,\N)$-regular sequence.
\end{theorem}

\begin{proof}
Let $f$ be an $\bS$-automatic $d$-dimensional sequence. First, we note that for all $\bs\in\N^d$,
\[
	\rho_f(\bs)
	=\Card\{
	\bp\in\N^d\colon 
	\forall\bp'\in\N^d\, 
	\big(
	E_{\bS}(\bp')<E_{\bS}(\bp)
	\implies f[\bp',\bs]\ne f[\bp,\bs]
	\big)
	\}.
\]
By Theorem~\ref{The : enumeration}, it suffices to prove that the set
\[
	X:=\{(\bs,\bp)\in\N^{2d}\colon
	\forall\bp'\in\N^d \, \big(
	E_{\bS}(\bp')<E_{\bS}(\bp)
	\implies f[\bp',\bs]\ne f[\bp,\bs]
	\big)\}
\] 
is $(\bS,\bS)$-recognizable. The notation $f[\bp',\bs]\ne f[\bp,\bs]$ is equivalent to $\exists \bi<\bs,\, f(\bp'+\bi)\ne f(\bp+\bi)$. By Proposition~\ref{Prop : EqualityofAutSeqisRec} and Theorem~\ref{The : P-S-rec}, since $f$ is $\bS$-automatic and addition is $\bS$-recognizable, the predicate $f[\bp',\bs]\ne f[\bp,\bs]$ is $\bS$-recognizable. By Proposition~\ref{Pro : AutEnumeration}, the predicate $E_{\bS}(\bp')<E_{\bS}(\bp)$ is $\bS$-recognizable as well. It then follows from Theorem~\ref{The : P-S-rec} that $X$ is $(\bS,\bS)$-recognizable as desired.
\end{proof}

In particular, Theorem~\ref{The : FactorComplexity} is valid for all $d$-dimensional Pisot numeration systems since addition is recognizable in such systems~\cite{BruyereHansel1997,FrougnySolomyak1996}. 
This observation solves~\cite[Probem 3.4.1]{Charlier2018}. 

In a similar manner, we obtain that a great amount of measures of $\bS$-automatic sequences gives rise to $(\bS,\N_{\infty})$-regular sequences. In particular, this holds true for all the measures mentioned in~\cite{CharlierRampersadShallit2012}. Let us illustrate our words in the case of the ($d$-dimensional) \emph{recurrence function} $R_f$. For a sequence $f\colon\N^d\to \K$, we let $R_f\colon\N^d\to\N_{\infty}$ be the function that maps every $\bs\in\N^d$ to the infimum of the nonnegative integers $\ell$ such that every factor of $f$ of size $(\ell,\ldots,\ell)$ contains all factors of $f$ of size $\bs$. 

\begin{proposition}
Suppose that addition is $\bS$-recognizable.
If $f\colon\N^d\to \K$ is $\bS$-automatic, then the recurrence function $R_f\colon\N^d\to\N_{\infty}$ is $(\bS,\N_{\infty})$-regular.
\end{proposition}

\begin{proof}
For all $\bs\in\N^d$, 
\begin{align*}
	R_f(\bs)
	&=\Card\{\ell\in\N\colon R_f(\bs)>\ell\}\\
	&=\Card\{\ell\in\N\colon \exists \bp,\bp'\in\N^d,\ f[\bp,(\ell,\ldots,\ell)] \text{ contains no occurrence of } f[\bp',\bs]\} \\
	&=\Card\{\ell\in\N\colon \exists \bp,\bp'\in\N^d,\ \forall \bk\le(\ell,\ldots,\ell)-\bs,\ f[\bp+\bk,\bs]\ne f[\bp',\bs]\}
\end{align*}
Applying the same arguments than in the proof of Theorem~\ref{The : FactorComplexity}, the result follows. 
\end{proof}

\section{$(\bS,\bS')$-Synchronized sequences}
\label{Sec : Synchronized}

A well-known family of $(b,\K)$-regular sequences is that of $b$-synchronized sequences~\cite{CarpiMaggi2001}. In this section, we generalize this notion to multidimensional abstract numeration systems.

\begin{definition}
A sequence $f\colon\N^d\to\N^{d'}$ is \emph{$(\bS,\bS')$-synchronized} if its graph 
\[
	\mathcal{G}_f=\{\couple{\bn}{f(\bn)}\colon \bn\in\N^d\}
\]
is an $(\bS,\bS')$-recognizable subset of $\N^{d+d'}$.
\end{definition}

\begin{runningexample}
The sequence $f$ is not $(\bS,\mathcal{S})$-synchronized. Proceed by contradiction and assume that there exists a DFA $\mathcal{A}$, say with $k$ states, recognizing $\rep_{(\bS,\mathcal{S})}(\mathcal{G}_f)$. For all $\ell\in \N$, $\rep_{\mathcal{S}}(\mathsf{v}(\ell))=a^{\ell+1}$ and hence, ${\tiny\triple{b^{\mathsf{v}(\ell)}}{b^{\mathsf{v}(\ell)}}{a^{\ell+1}}^\#}\in\rep_{(\bS,\mathcal{S})}(\mathcal{G}_f)$. Now suppose that $\ell+1\ge k$. By the pumping lemma, there exists $m\in \N_{\ge 1}$ such that the word ${\tiny\triple{b^{\mathsf{v}(\ell)+m}}{b^{\mathsf{v}(\ell)+m}}{a^{\ell+1+m}}}$ is accepted by $\mathcal{A}$ as well. This is impossible since $\mathsf{v}(\ell)+m\ne \mathsf{v}(\ell+m)$ and hence $\rep_{\mathcal{S}}(\mathsf{v}(\ell)+m)\ne a^{\ell+1+m}$.  

However, the sequence $f$ is $(\bS,\mathcal{S}_c)$-synchronized where $\mathcal{S}_c$ is the unary abstract numeration system $(c^*,c)$. A DFA recognizing the language $\rep_{(\bS,\mathcal{S}_c)}(\mathcal{G}_f)$ is depicted in Figure~\ref{Fig : AutomatonSynchronized}.
\begin{figure}[htb]
\begin{center}
\hspace*{-3.5cm}
\scalebox{0.6}{
\begin{tikzpicture}
\tikzstyle{every node}=[shape=circle, fill=none, draw=black,
minimum size=20pt, inner sep=2pt]

\node(initial) at (0,0){$ $};

\node (diezA) at (1.76,4.24){$ $};
\node (Adiez) at (1.76,-4.24){$ $};
\node (diezB) at (6,6){$ $};
\node (AB) at (10.24,4.24){$ $};
\node (Bdiez) at (6,-6){$ $};
\node (BA) at (10.24,-4.24){$ $};
\node (BB) at (12,0)  {$ $};

\node (AAwaiting) at (5,0){$ $};
\node (AA) at (7,0){$ $};

\tikzstyle{every node}=[shape=circle, fill=none, draw=black, minimum size=15pt, inner sep=2pt]
\node (initialbis) at (0,0){$ $};
\node (AAbis) at (7,0){$ $};
\node (BBbis) at (12,0){$ $};
\tikzstyle{every path}=[color=black, line width=0.5 pt]
\tikzstyle{every node}=[shape=circle, minimum size=5pt, inner sep=2pt]
\draw [-Latex] (-1,0) to node [above=-0.2] {$ $} (initial);

\draw [-Latex,brown] (diezA) [loop above] to node {$ $} ();
\draw [-Latex,violet] (diezB) [loop above] to node {$ $} ();
\draw [-Latex,blue] (AB) [loop above] to node {$ $} ();
\draw [-Latex,orange] (Adiez) [loop below] to node {$ $} ();
\draw [-Latex,magenta] (Bdiez) [loop below] to node {$ $} ();
\draw [-Latex,cyan] (BA) [loop below] to node {$ $} ();
\draw [-Latex,amber] (AAwaiting) [loop right] to node {$ $} ();
\draw [-Latex,vert] (AA) [loop above] to node {$ $} ();
\draw [-Latex,red] (BB) [loop right] to node {$ $} ();

\draw [-Latex,brown] (initial) to node {$ $} (diezA);

\draw [-Latex,orange] (initial) to node {$ $} (Adiez);

\draw [-Latex,vert] (initial) to [out=20,in=160] node {$ $} (AA);
\draw [-Latex,vert] (diezA) to node {$ $} (AA);
\draw [-Latex,vert] (Adiez) to node {$ $} (AA);

\draw [-Latex,amber] (initial) to node {$ $} (AAwaiting);
\draw [-Latex,amber] (diezA) to node {$ $} (AAwaiting);
\draw [-Latex,amber] (Adiez) to node {$ $} (AAwaiting);
 
\draw [-Latex,violet] (diezA) to node {$ $} (diezB);
\draw [-Latex,violet] (initial) to  node {$ $} (diezB);

\draw [-Latex,blue] (diezB) to node {$ $} (AB);
\draw [-Latex,blue] (initial) [bend left=15] to node {$ $} (AB);
\draw [-Latex,blue] (AAwaiting) to node {$ $} (AB);
\draw [-Latex,blue] (diezA) to node {$ $} (AB);
\draw [-Latex,blue] (Adiez) to  [bend right=20,in=-145]  node {$ $} (AB);

\draw [-Latex,magenta] (Adiez) to node {$ $} (Bdiez);
\draw [-Latex,magenta] (initial) to node {$ $} (Bdiez);

\draw [-Latex,cyan] (Bdiez) to node {$ $} (BA);
\draw [-Latex,cyan] (initial) to [bend right=15] node {$ $} (BA);
\draw [-Latex,cyan] (AAwaiting) to node {$ $} (BA);
\draw [-Latex,cyan] (Adiez) to node {$ $} (BA);
\draw [-Latex,cyan] (diezA) to  [out=-30,in=100]  node {$ $} (BA);

\draw [-Latex,red] (AA) to node {$ $} (BB);
\draw [-Latex,red] (Bdiez) to node {$ $} (BB);
\draw [-Latex,red] (BA) to node {$ $} (BB);
\draw [-Latex,red] (diezB) to node {$ $} (BB);
\draw [-Latex,red] (AB) to node {$ $} (BB);
\draw [-Latex,red] (Adiez) to [out=20, in=210] node {$ $} (BB);
\draw [-Latex,red] (diezA) to node {$ $} (BB);
\draw [-Latex,red] (initial) to [bend right=65, looseness=1.3, out=-30,in=200]  node {$ $} (BB);
\end{tikzpicture}
}
\end{center}
\vspace*{-2.5cm}
\hfill\scalebox{0.8}{
\fbox{\begin{tabular}{cc|cc|cc}
\textcolor{brown}{$\blacksquare$} 	& $\triple{\#}{a}{\#}$ 
& \textcolor{amber}{$\blacksquare$} 	& $\triple{a}{a}{\#}$  
&\textcolor{cyan}{$\blacksquare$} 	& $\triple{b}{a}{\#}$\\
\textcolor{violet}{$\blacksquare$} 	& $\triple{\#}{b}{\#}$ 
&\textcolor{vert}{$\blacksquare$} 	& $\triple{a}{a}{c}$
& \textcolor{red}{$\blacksquare$} 	& $\triple{b}{b}{c}$\\
\textcolor{orange}{$\blacksquare$} 	& $\triple{a}{\#}{\#}$
& \textcolor{magenta}{$\blacksquare$} & $\triple{b}{\#}{\#}$
&\textcolor{blue}{$\blacksquare$} 	& $\triple{a}{b}{\#}$  
\end{tabular}}
}
\vspace{0.5cm}
\caption{A DFA recognizing $\rep_{(\bS,\mathcal{S}_c)}(\mathcal{G}_f)$ where $\mathcal{S}_c=(c^*,c)$.}
\label{Fig : AutomatonSynchronized}
\end{figure}
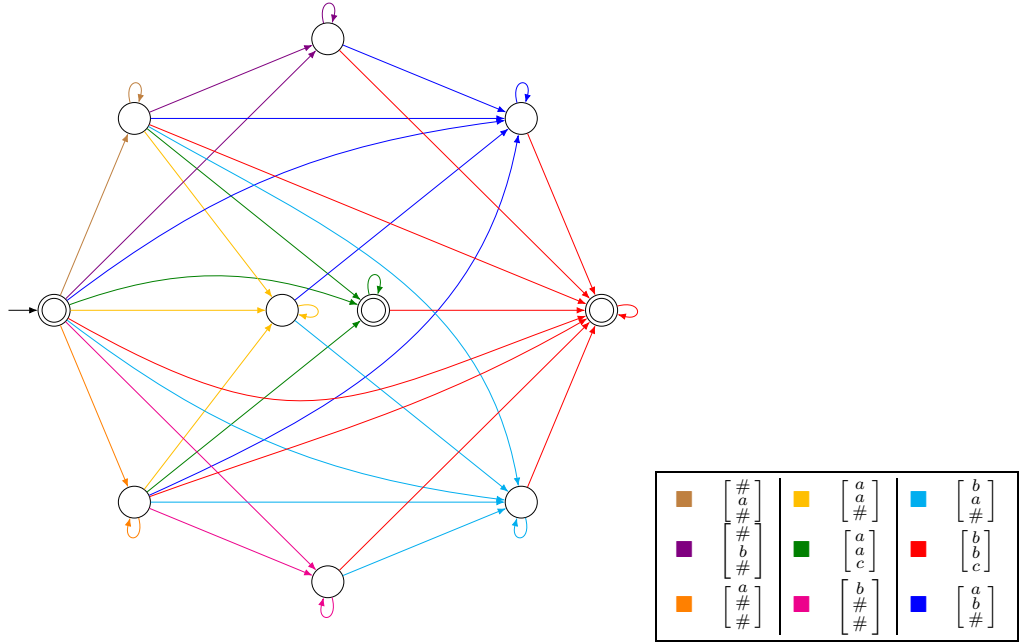
\end{runningexample}

First, let us provide some examples of $(\bS,\bS')$-synchronized sequences. One obtains the following lemma by adapting the proof of~\cite[Lemma 2.10]{CharlierLacroixRampersad2012} to left padding symbols. This involves adapting~\cite[Section 2.6.3]{FrougnySakatovitch2010} to this framework as well.

\begin{lemma}
\label{Lem : +k}
For all $k\in \N$ and all abstract numeration systems $\mathcal{S}$, the sequence $\N\to\N,\ n\mapsto n+k$ is $(\mathcal{S},\mathcal{S})$-synchronized.
\end{lemma}

\begin{proposition}
\label{Pro : +k}
For all $\bk\in \N^d$, the sequence $\N^d\to\N^d,\ \bn\mapsto \bn+\bk$ is $(\bS,\bS)$-synchronized.
\end{proposition} 

\begin{proof}
By Lemma~\ref{Lem : +k}, the $d$ graphs $\{ \couple{n}{n+k_i} \colon n\in\N \}$ are $(\mathcal{S}_i,\mathcal{S}_i)$-recognizable, for $i\in[\![1,d]\!]$. We deduce that the graph $\{\couple{\bn}{\bn+\bk} \colon \bn\in\N^d \}$ is $(\bS,\bS)$-recognizable since it is equal to
$\bigcap_{i=1}^d\ \{\couple{\bn}{\bn'}
\colon n'_j= n_j \text{ for } j\ne i,\ n'_i= n_i+k_i\}$.
\end{proof}

We prove that if $d'=1$ then 
for any choice of abstract numeration system 
$\mathcal{S}':=\mathcal{S}'_1$, the family of $(\bS,\mathcal{S}')$-synchronized sequences lies in between those of $\bS$-automatic sequences and $(\bS,\N)$-regular sequences. Note that the hypothesis $d'=1$ is used in Proposition~\ref{Pro : SynchRegular} only.

\begin{proposition}
\label{Pro : AutSynch}
A sequence $f\colon\N^d\to \N^{d'}$ is $\bS$-automatic if and only if it is $(\bS,\bS')$-synchronized and takes only finitely many values.
\end{proposition}

\begin{proof}
Let $f\colon\N^d\to\N^{d'}$ be an $\bS$-automatic sequence. 
By Lemma~\ref{Lem : Automatic-Fibers}, for all $\bn'\in \N^{d'}$, the sets $f^{-1}(\bn')$ are $\bS$-recognizable. Then $\mathcal{G}_f$ is $(\bS,\bS')$-recognizable since it is equal to the finite union $\cup_{\bn'\in f(\N^d)}\,\big(f^{-1}(\bn')\times \{\bn'\}\big)$. 

Conversely, let $f\colon \N^d\to\N^{d'}$ be an $(\bS,\bS')$-synchronized sequence with finite image.  Then for all $\bn'\in\N^{d'}$, $\rep_{\bS}\big(f^{-1}(\bn')\big)=\{\bw\in\bL\colon \couple{\bw}{\rep_{\bS'}(\bn')}^\#\in\rep_{(\bS,\bS')}(\mathcal{G}_f\}$ is regular. We conclude by using Lemma~\ref{Lem : Automatic-Fibers}.
\end{proof}

\begin{proposition}
\label{Pro : SynchRegular}
Any $(\bS,\mathcal{S}')$-synchronized sequence $f\colon\N^d\to\N$ is $(\bS,\N)$-regular.
\end{proposition}

\begin{proof}
Let $f\colon \N^d\to\N$ be an $(\bS,\mathcal{S}')$-synchronized sequence. Then the set 
\[
	X=\{\couple{\bn}{\ell}\in\N^{d+1} 
	\colon \ell<f(\bn)\}
\]
is $(\bS,\mathcal{S}')$-recognizable since it is equal to
\[
	\{\couple{\bn}{\ell}\in\N^{d+1}
	\colon \exists m,\ \couple{\bn}{m}\in \mathcal{G}_f\ \land\ \ell<m \}.
\]
Since for all $\bn\in\N^d$, $f(\bn)=\Card\{\ell\in\N\colon \couple{\bn}{\ell}\in X\}$, the result follows from Theorem~\ref{The : enumeration}.
\end{proof}

\begin{remark}
Even though both families of $\bS$-automatic sequences and $(\bS,\N)$-regular sequences are closed under sum, product and product by a constant, it is no longer the case of the family of $(\bS,\mathcal{S}')$-synchronized sequences. For instance, the sequence $\N \to \N,\ n \mapsto n$ is $(\mathcal{S},\mathcal{S})$-synchronized for any abstract numeration system $\mathcal{S}$. However, the sequence $\N \to \N,\ n \mapsto 2n$ is not $(\mathcal{S},\mathcal{S})$-synchronized in general. For example, it is not for the unary system $\mathcal{S}=(c^*,c)$.
\end{remark}

Let us show that, similarly to the families of multidimensional automatic and regular sequences, the family of multidimensional synchronized sequences is closed under projection. 

\begin{proposition}
Let $f\colon\N^d\to \N^{d'}$ be an $(\bS,\bS')$-synchronized sequence, let $i\in[\![1,d]\!]$ and let $k\in\N$. Then the sequence $\delta_{i,k}(f)\colon\N^{d-1}\to \N^{d'}$ (defined as in Proposition~\ref{Pro : Projection-Regular} where $\K$ is replaced by $\N^{d'}$) is $(\delta_i(\bS),\bS')$-synchronized.
\end{proposition}

\begin{proof}
By definition, the graph $\mathcal{G}_f$ is an $(\bS,\bS')$-recognizable subset of $\N^{d+d'}$. Thus, so is 
\[
	\mathcal{G}_f\cap
	\big\{\couple{\bn}{\bn'}\in\N^{d+d'}
	\colon \pi_i(\bn)=k\big\}.
\]
We have
\[
	\delta_i\big(\mathcal{G}_f\cap
	\big\{\couple{\bn}{\bn'}\in\N^{d+d'}
	\colon \pi_i(\bn)=k\big\}\big)
	=\mathcal{G}_{\delta_{i,k}(f)}
\]
where for $X\subseteq \N^{d+d'}$, $\delta_i(X)$ deletes the $i$-th component of all elements in $X$. Since no $(\bS,\bS')$-representation contains the letter {\tiny $\couple{\bdiese}{\bdiese'}$}, we can build a DFA accepting the language
\[
	\couple{\bdiese}{\bdiese'}^*
	\rep_{(\bS,\bS')}\big(\mathcal{G}_f\cap
	\big\{\couple{\bn}{\bn'}\in\N^{d+d'}
	\colon \pi_i(\bn)=k\big\}\big)
\]
by adding a loop labeled by {\tiny $\couple{\bdiese}{\bdiese'}$} on the initial state. By deleting the $i$-th component of every label in this DFA, we obtain an NFA accepting a language $K$ such that
\[
	\rep_{(\delta_i(\bS),\bS')}(\mathcal{G}_{\delta_{i,k}(f)})
	=\tau_{\tiny \couple{\delta_i(\bdiese)}{\bdiese'}}(K).
\]
The image of a regular language under a morphism remaining regular, we obtain that $\delta_{i,k}(f)$ is $(\delta_i(\bS),\bS')$-synchronized.
\end{proof}

In order to obtain more properties on $(\bS,\bS')$-synchronized sequences, we give a characterization of these sequences in terms of synchronized relations. We consider relations $R\colon A^*\to B^*$ where $A$ and $B$ are arbitrary alphabets. The \emph{graph} of such a relation is the subset
\[
	\mathcal{G}_R
	=\{\couple{u}{v}\in A^*\times B^* \colon u Rv\}
\]
of $A^*\times B^*$. Note that a relation is completely determined by its graph. We let $\$\notin A\cup B$ and we let $(\mathcal{G}_R)^\$=\{\couple{u}{v}^\$\colon \couple{u}{v}\in \mathcal{G}_R\}$ designate the corresponding language over the alphabet $\big(A_\$\times B_\$\big)\setminus{\{\small\couple{\$}{\$}\}}$. A relation $R\colon A^*\to B^*$ is \emph{synchronized} if the language $(\mathcal{G}_R)^\$ $ is regular.

\begin{runningexample}
Consider the relation $R: \bA^* \to \bA^*$ whose graph is $\mathcal{G}_R=\{\couple{\bu}{\bv}\in\bA^*\times\bA^*\colon \big| |\bu|-|\bv| \big| \le 1\}$. This relation is synchronized since the language $(\mathcal{G}_R)^\$$ is recognized by the DFA depicted in Figure~\ref{Fig : RelationAutomaton}.
\begin{figure}[htb]
\begin{center}
\begin{tikzpicture}
\tikzstyle{every node}=[shape=circle, fill=none, draw=black,minimum size=20pt, inner sep=2pt]
\node(1) at (0,0) {$1$};
\node(2) at (6,0) {$2$};
\tikzstyle{every node}=[shape=circle, fill=none, draw=black,minimum size=15pt, inner sep=2pt]
\node(1f) at (0,0) {};
\node(2f) at (6,0) {};
\tikzstyle{every path}=[color=black, line width=0.5 pt]
\tikzstyle{every node}=[]
\draw [-Latex] (-1,0) to node { } (1) ; 
\draw [-Latex] (2) to [loop above] node {$\couple{\ba}{\bb},\ \ba,\bb\in\bA$} (2);
\draw [-Latex] (1) to node [below] {$\couple{\$}{\ba}, \couple{\ba}{\$}, \couple{\ba}{\bb},\  
\ba,\bb\in\bA $} 
(2);
\end{tikzpicture}
\end{center}
\caption{A DFA accepting $(\mathcal{G}_R)^\$$ where $\mathcal{G}_R=\{\couple{\bu}{\bv}\in\bA^*\times\bA^*\colon \big| |\bu|-|\bv| \big| \le 1\}$.}
\label{Fig : RelationAutomaton}
\end{figure}
\end{runningexample}

For a sequence $f\colon\N^d\to\N^{d'}$, we define a relation $R_{f,\bS,\bS'}\colon \bA^*\to(\bA')^*$ by
\[
	\mathcal{G}_{R_{f,\bS,\bS'}}
	=\{\couple{\bw}{\bw'}\in \bL\times \bL'
	\colon f(\val_{\bS}(\bw))=\val_{\bS'}(\bw')\}.
\]	 
Note that the letter $\#$ already appears in the alphabet $\bA$ if $d>1$ (resp.\ in the alphabet $\bA'$ if $d'>1$), so it is convenient to use a symbol $\$$ different from $\#$ in order to pad elements of the graph $\mathcal{G}_{R_{f,\bS,\bS'}}$. In this way, there is no ambiguity between the $\#$-padding of representations in a multidimensional abstract numeration system and the $\$$-padding of elements in the graph of a relation. 
 
\begin{proposition}
\label{Pro : SynchSeqGraph}
A sequence $f\colon\N^d\to\N^{d'}$ is $(\bS,\bS')$-synchronized if and only if the relation $R_{f,\bS,\bS'}$ is synchronized.
\end{proposition}

\begin{proof}
Let $\sigma\colon \big((\bA_\$\times \bA'_\$)\setminus\{\couple{\$}{\$}\}\big)^*\to \big((\bA_{\bdiese}\times \bA'_{\bdiese'})\setminus\{{\tiny\couple{\bdiese}{\bdiese'}}\}\big)^*$ be the morphism defined by
\[
	\sigma\couple{\ba}{\$}=\couple{\ba}{\bdiese'},\quad
	\sigma\couple{\$}{\ba'}=\couple{\bdiese}{\ba'},\quad
	\sigma\couple{\ba}{\ba'}=\couple{\ba}{\ba'}
\]
for $\ba\in\bA$ and $\ba'\in\bA'$. Here $\bdiese'$ is the $d'$-dimensional letter all whose components are equal to $\#$. Then 
\[
	\rep_{(\bS,\bS')}(\{\couple{\bn}{f(\bn)}\colon \bn\in\N^d\})
	=\sigma((\mathcal{G}_{R_{f,\bS,\bS'}})^\$).
\]
The conclusion follows from the injectivity of $\sigma$ and the fact that the family of regular languages is closed under both image and inverse image under a morphism.
\end{proof}

The family of regular sequences is not closed under composition~\cite[p.\ 169, (iii)]{AlloucheShallit1992}. Thanks to Proposition~\ref{Pro : SynchSeqGraph}, we show that the composition of synchronized sequences is synchronized. To do so, in the following proposition, we consider an extra multidimensional abstract numeration system $\bS''$, of dimension $d''$.

\begin{proposition}
\label{Pro : Composition-Synchronized}
Let $f\colon \N^d\to\N^{d'}$ be an $(\bS,\bS')$-synchronized sequence and let $g\colon \N^{d'}\to\N^{d''}$ be an $(\bS',\bS'')$-synchronized sequence. Then their composition $g\circ f\colon \N^d\to\N^{d''}$ is $(\bS,\bS'')$-synchronized.
\end{proposition}

\begin{proof}
By Proposition~\ref{Pro : SynchSeqGraph}, the relations $R_{f,\bS,\bS'}$ and $R_{g,\bS',\bS''}$ are synchronized. From~\cite[Theorem~2.6.6]{FrougnySakatovitch2010}, the composition $R_{g,\bS',\bS''}\circ R_{f,\bS,\bS'}$ is synchronized. Since $R_{g,\bS',\bS''}\circ R_{f,\bS,\bS'}=R_{g\circ f,\bS,\bS''}$, we conclude by using Proposition~\ref{Pro : SynchSeqGraph} again.
\end{proof}

\begin{corollary}
If $f\colon\N^d\to\N^{d'}$ is an $(\bS,\bS')$-synchronized sequence, then for all $\bk\in \N^d$, the sequences $\N^d\to\N^{d'},\ \bn\mapsto f(\bn+\bk)$ are $(\bS,\bS')$-synchronized.
\end{corollary}

\begin{proof}
This is a consequence of Propositions~\ref{Pro : +k} and~\ref{Pro : Composition-Synchronized}.
\end{proof}

\begin{corollary}
Let $f\colon\N^d\to\N^{d'}$ and $g\colon\N^{d'}\to\N^{d''}$.
\begin{itemize}
\item If $f$ is $\bS$-automatic and $g$ is $(\bS',\bS'')$-synchronized, then $g\circ f$ is $\bS$-automatic.
\item If $f$ is $(\bS,\bS')$-synchronized and $g$ is $\bS'$-automatic, then $g\circ f$ is $\bS$-automatic.
\end{itemize}
\end{corollary}

\begin{proof}
This is a consequence of Propositions~\ref{Pro : AutSynch} and~\ref{Pro : Composition-Synchronized}.\end{proof}

\section{Mixing regular sequences and synchronized sequences}
\label{Sec : Reg-Synch}

Even though the family of regular sequences is not closed under composition in general, in this section we prove that the composition of a regular sequence and a synchronized one is regular. This result can be seen as a generalization of Proposition~\ref{Pro : Composition-Synchronized}.

In order to do so, we will first show a general result concerning the composition of recognizable series and synchronized relations. Our proof of this result is based on automata. Indeed, automata are intrinsincally present in the notion of synchronized relations. Moreover, it is classical that recognizable series may be defined through an automata point of view. To make this link precise, we first present the notion of weighted automaton in Section~\ref{Sec : Weighted automata}. Next, we shall turn to the composition of recognizable series and synchronized relations in Section~\ref{Sec : CompRecSynch}, and finally we will prove the announced result on the composition of regular and synchronized sequences in Section~\ref{Sec : CompRegSynch}.

\subsection{Weighted automata}
\label{Sec : Weighted automata}

A \emph{weighted (finite) automaton} $\mathcal{A} = (Q, I, T, A, E)$ with weights in a semiring $\K$, or simply a \emph{$\K$-automaton}, is composed of a (finite) set $Q$ of states, a (finite) alphabet $A$ and of three mappings $I \colon Q \to \K$, $T \colon Q \to \K$ and $E \colon Q \times A \times Q \to \K$. We call a state $q$ \emph{initial} if $I(q) \ne 0$ and \emph{final} if $T(q) \ne 0$. A triple $(p,a,q)\in Q \times A \times Q$ is called a \emph{transition}. The \emph{label} of a transition $(p,a,q)$ is the letter $a$ and its \emph{weight} is $E(p,a,q)$. A \emph{path} in $\mathcal{A}$ is a sequence 
\[
	c = (q_0,a_1,q_1)(q_1,a_2,q_2) \cdots (q_{n-1},a_n,q_n)
\] 
of transitions, which will sometimes be shortened as
\begin{equation}
\label{Eq:chemin}	
	c=c_\mathcal{A}(q_0q_1\cdots q_n,a_1a_2\cdots a_n).
\end{equation}
The \emph{weight} of the path $c$ is the product 
\[
	E(c) = E(q_0,a_1,q_1)E(q_1,a_2,q_2)\cdots E(q_{n-1},a_n,q_n)
\] 
of the weights of its transitions. Its \emph{label} is the word $a_1a_2\cdots a_n$. We let $i_c$ and $t_c$ denote the first and last states of $c$ respectively. A path $c$ is \emph{initial} if $i_c$ is initial and \emph{final} if $t_c$ is final. A state $q\in Q$ is \emph{co-accessible} if there exists a non-zero weight final path starting in $q$. For $w\in A^*$, we let $C_\mathcal{A}(w)$ denote the set of paths in $\mathcal{A}$ of label $w$ that are both initial and final. The \emph{weight} of $w$ in $\mathcal{A}$ is the quantity 
\[
	\sum_{c\in C_\mathcal{A}(w)} I(i_c)E(c)T(t_c).
\] 

A $\K$-automaton can be represented by a graph where the states are the vertices and each transition $(p,a,q)$ is an arrow from $p$ to $q$ of label $a | E(p,a,q)$. In practice, we omit to represent the zero weight transitions. Every initial state $q$ has an incoming arrow labeled by $I(q)$ and every final state  has an outgoing arrow labeled by $T(q)$.

\begin{runningexample}
Considering the $\N$-automaton $\mathcal{A}$ over the alphabet $\bA$ depicted in Figure~\ref{Fig : WeightedAutomaton}. Since all represented weights are equal to $1$, the weight of a word $\bw\in \bA^*$ is equal to the number of represented paths labeled by $\bw$ that are both initial and final. For instance, the weights of the words $\couple{\#ab}{aab}$, $ \couple{aaaab}{\#aaab}$, $\couple{aab}{bab}$, $\couple{aa}{ab}$ and $\couple{a\#a}{aba}$ are respectively equal to $2$, $4$, $2$, $0$ and $1$. Those values correspond to those given in Figure~\ref{Fig : CoeffS}.
\begin{figure}[htb]
\begin{center}
\begin{tikzpicture}
\tikzstyle{every node}=[shape=circle, fill=none, draw=black,minimum size=20pt, inner sep=2pt]
\node(1) at (0,0) {$T$};
\node(2) at (5,0) {$S$};
\tikzstyle{every path}=[color=black, line width=0.5pt]
\tikzstyle{every node}=[]
\draw [-Latex] (2) to node [above] {$1$} (6,0); 
\draw [-Latex] (-1,0) to node [above] {$1$} (1) ; 
\draw [-Latex] (2) to [loop above] node [above] {$\couple{a}{a}|1,\couple{b}{b}| 1$} (2);
\draw [-Latex] (1) to node [above] {$\couple{a}{a}|1,\couple{b}{b}| 1$} (2);
\draw [-Latex] (1) to [loop above] node [above] {$\ba|1, \ba\in\bA$} (1);
\end{tikzpicture}
\end{center}
\caption{An $\N$-automaton over $\bA=\left\{\couple{\#}{a},\couple{\#}{b},\couple{a}{\#},\couple{a}{a},\couple{a}{b},\couple{b}{\#},\couple{b}{a},\couple{b}{b}\right\}$.}
\label{Fig : WeightedAutomaton}
\end{figure}
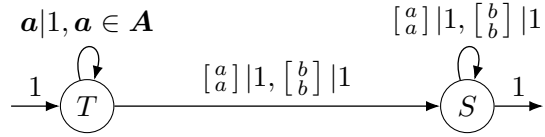
\end{runningexample}

In what follows, it will be convenient to also use some of the previous definitions for a DFA $\mathcal{A}$ (in an obvious adapted manner). For instance, for a DFA $\mathcal{A}$ with a (partial) transition function $\delta$, the notation \eqref{Eq:chemin} designates the path in $\mathcal{A}$ labeled by $a_1a_2\cdots a_n$ and visiting the states $q_0,\ldots,q_n$, i.e., for each $i\in[\![1,n]\!]$, we have $\delta(q_{i-1},a_i)=q_i$. Similarly, for a DFA $\mathcal{A}$, we let $C_\mathcal{A}$ denote the set of accepting paths in $\mathcal{A}$, and moreover,  we let $L_\mathcal{A}$ denote the set of labels of the paths in $C_\mathcal{A}$, i.e., $L_\mathcal{A}$ is the language accepted by $\mathcal{A}$.

Formal series and $\K$-automata can be linked. 
The series \emph{recognized} by a $\K$-automaton $\mathcal{A}=(Q,I,T,A,E)$ is the series $S$ whose coefficients are the weights of the words over $A$ in $\mathcal{A}$. 

\begin{proposition}
A series is recognized by a $\K$-automaton if and only if it is $\K$-recognizable.
\end{proposition}

The proof of the previous proposition is constructive. Roughly, the linear representation encodes the weights in the $\K$-automaton.

\begin{runningexample}
The $\N$-automaton of Figure~\ref{Fig : WeightedAutomaton} recognizes the series $S$. It corresponds to the linear representation given in~\eqref{Eq:LinearRep}. An $\N$-automaton recognizing the series $S_f$ is depicted in Figure~\ref{Fig : WeightedAutomaton-S_f}. The latter $\N$-automaton corresponds to the $\N$-submodule generated by the ten series $S_f$, 
$S_g$, 
$S_{\chi_{\{0\}\times\val_{\mathcal{S}}(a^*)}}$, 
$S_{\chi_{\{0\}\times\N}}$, 
$S_{\chi_{\val_{\mathcal{S}}(a^*)\times\{0\}}}$, 
$S_{\chi_{\val_{\mathcal{S}}(a^*)\times\val_{\mathcal{S}}(a^*)}}$, 
$S_{\chi_{\val_{\mathcal{S}}(a^*)\times\N}}$,
$S_{\chi_{\N\times\{0\}}}$,
$S_{\chi_{\N\times\val_{\mathcal{S}}(a^*)}}$ 
and $S_1$: the states are numbered according the previous list of generators and the arrows are colored depending on their labels. The only final state corresponds to the series $S_f$. A state $j$ is initial if the coefficient of the empty word in the $j$-th generator is non-zero. In fact, all states are initial except those corresponding to the series $S_f$ and $S_g$. Indeed, we have $(S_g,\boldsymbol{\varepsilon})=(S_f,\boldsymbol{\varepsilon})=0$ since $g(\bz)=f(\bz)=0$.
\begin{figure}[htb]
\begin{center}
\hspace{-2cm}
\scalebox{0.6}{\begin{tikzpicture}
\tikzstyle{every node}=[shape=circle, fill=none, draw=black,
minimum size=20pt, inner sep=2pt]
\node(1) at (0,6) {$1$};
\node(2) at (-4.24,4.24) {$2$};
\node(3) at (0,-2) {$3$};
\node(4) at (0,2) {$4$};
\node(5) at (0,-6) {$5$};
\node(6) at (-6,0) {$6$};
\node(7) at (-4.24,-4.24) {$7$};
\node(8) at (6,0) {$8$};
\node(9) at (4.24,-4.24) {$9$};
\node(10) at (4.24,4.24) {$10$};

\tikzstyle{every node}=[shape=circle, fill=none, draw=black, minimum size=15pt, inner sep=2pt]
\tikzstyle{every path}=[color=black, line width=0.5 pt]
\tikzstyle{every node}=[shape=circle, minimum size=5pt, inner sep=2pt]
\draw [-Latex] (1) to node [above] {$\text{\footnotesize{ 1}}$} (-1,6);
\draw [-Latex] (1,-2) to node [above] {$\text{\footnotesize{ 1}}$} (3); 
\draw [-Latex] (-1,2) to node [above] {$\text{\footnotesize{ 1}}$} (4); 
\draw [-Latex] (1,-6) to node [below] {$\text{\footnotesize{ 1}}$} (5);
\draw [-Latex] (-6,1) to node [left] {$\text{\footnotesize{ 1}}$} (6); 
\draw [-Latex] (-5.24,-4.24) to node [above] {$\text{\footnotesize{ 1}}$} (7); 
\draw [-Latex] (6,1) to node [right] {$\text{\footnotesize{ 1}}$} (8); 
\draw [-Latex] (5.24,-4.24) to node [above] {$\text{\footnotesize{ 1}}$} (9); 
\draw [-Latex] (5.24,4.24) to node [above] {$\text{\footnotesize{ 1}}$} (10);

\draw [-Latex,red] (1) to [loop above] node [above=-0.3] {$ $} (1);
\draw [-Latex,vert] (2) to node [above=-0.2] {$ $} (1);
\draw [-Latex,vert] (6) to node [above] {$ $} (1);
\draw [-Latex,red](10) to node [above=-0.2] {$ $} (1);

\draw [-Latex,vert] (2) to [loop above] node [above=-0.3] {$ $} (2);
\draw [-Latex,vert] (6) to node [left] {$ $} (2);

\draw [-Latex,brown] (3) to [loop below] node [below=-0.35] {$ $} (3);

\draw [-Latex,violet] (4) to [loop above] node [above=-0.4] {$ $} (4);
\draw [-Latex,brown] (3) to node [left=-0.1,pos=0.6] {$ $} (4);

\draw [-Latex,orange] (5) to [loop below] node [below=-0.35] {$ $} (5);

\draw [-Latex,vert] (6) to [loop left] node [left] {$ $} (6);
\draw [-Latex,brown] (3) to node [pos=0.7,above=-0.3] {$ $} (6);
\draw [-Latex,orange] (5) to node [above,pos=0.1] {$ $} (6);

\draw [-Latex,blue] (7) to [loop below] node [below=-0.35] {$ $} (7);
\draw [-Latex,brown] (3) to node [left=0.2,pos=0.001] {$ $} (7);
\draw [-Latex,violet] (4) to [bend right=10] node [left,pos=0.2] {$ $} (7);
\draw [-Latex,orange] (5) to node [below=-0.2] {$ $} (7);
\draw [-Latex,vert] (6) to node [left] {$ $} (7);

\draw [-Latex,magenta] (8) to [loop right] node [right=-0.1] {$ $} (8);
\draw [-Latex,orange] (5) to node [left,pos=0.7] {$ $} (8);

\draw [-Latex,cyan] (9) to [loop below] node [below=-0.35] {$ $} (9);
\draw [-Latex,brown] (3) to node [above=-0.2] {$ $} (9);
\draw [-Latex,orange] (5) to node [below=-0.2] {$ $} (9);
\draw [-Latex,vert] (6) to [bend right=23,looseness=1.2, out=-10, in=200] node [left=0.4,pos=0.65] {$ $} (9);
\draw [-Latex,magenta] (8) to node [right=-0.1] {$ $} (9);

\draw [-Latex,red] (10) to [loop above] node [above=-0.35] {$ $} (10);
\draw [-Latex,brown] (3) to node [right=-0.03,pos=0.1] {$ $} (10);
\draw [-Latex,violet] (4) to node [left=0.05, pos=0.6] {$ $} (10);
\draw [-Latex,orange] (5) to  node [right=-0.1,pos=0.6] {$ $} (10);
\draw [-Latex,vert] (6) to [bend left=10,in=150] node [above=-0.2] {$ $} (10);
\draw [-Latex,blue] (7) to node [left, pos=0.75] {$ $} (10);
\draw [-Latex,magenta] (8) to node [right=-0.1] {$ $} (10);
\draw [-Latex,cyan] (9) to [] node [right=-0.1,pos=0.6] {$ $} (10);
\end{tikzpicture}}
\end{center}
\vspace*{-2.5cm}
\hfill\scalebox{0.8}{
\fbox{\begin{tabular}{cc|cc}
\textcolor{brown}{$\blacksquare$} 	& $\couple{\#}{a}|\text{\footnotesize{ 1}}$ 
&\textcolor{blue}{$\blacksquare$} 	& $\couple{a}{b}|\text{\footnotesize{ 1}}$\\
\textcolor{violet}{$\blacksquare$} 	& $\couple{\#}{b}|\text{\footnotesize{ 1}}$ 
&\textcolor{magenta}{$\blacksquare$} & $\couple{b}{\#}|\text{\footnotesize{ 1}}$\\
\textcolor{orange}{$\blacksquare$} 	& $\couple{a}{\#}|\text{\footnotesize{ 1}}$
&\textcolor{cyan}{$\blacksquare$} 	& $\couple{b}{a}|\text{\footnotesize{ 1}}$\\
\textcolor{vert}{$\blacksquare$} 	& $\couple{a}{a}|\text{\footnotesize{ 1}}$  
& \textcolor{red}{$\blacksquare$} 	& $\couple{b}{b}|\text{\footnotesize{ 1}}$\\
\end{tabular}}
}
\vspace{0.5cm}
\caption{An $\N$-automaton recognizing the series $S_f$ corresponding to the $(\bS,\N)$-regular sequence $f \colon \N^2 \to \N,\ \couple{m}{n} \mapsto\max|\Suff(\rep_{\mathcal{S}}(m))\cap \Suff(\rep_{\mathcal{S}}(n))|$.}
\label{Fig : WeightedAutomaton-S_f}
\end{figure}
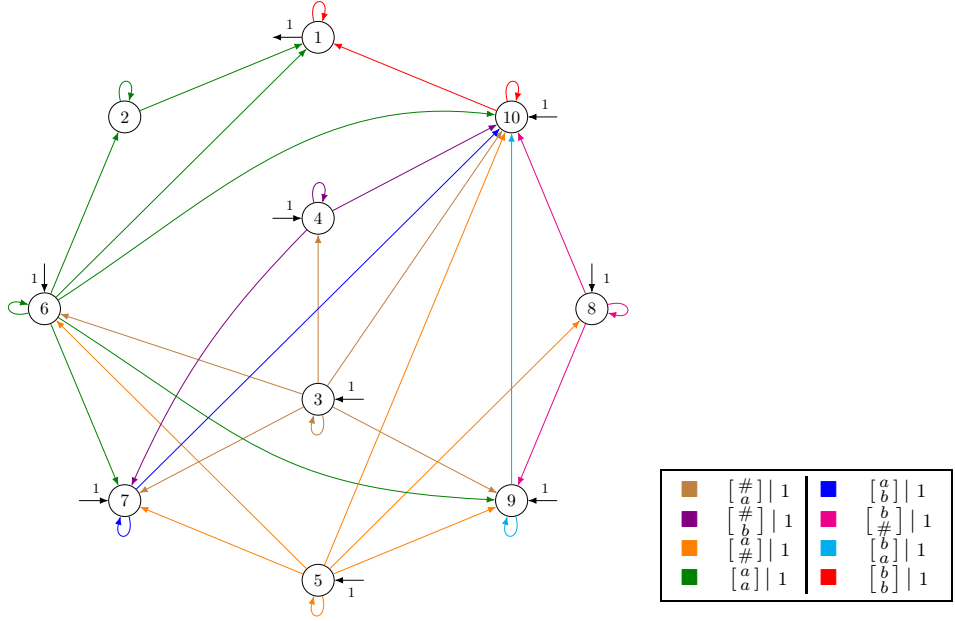
\end{runningexample}

\subsection{Composing synchronized relations and $\K$-recognizable series}
\label{Sec : CompRecSynch}

As a first step, we consider the composition of a synchronized relation and a $\K$-recognizable series. For a relation $R\colon A^*\to B^*$ and a series  $S \colon B^*\to\K$ such that for all $u\in A^*$, the language $\{v\in B^*\colon uRv\}$ is finite, we define the \emph{composition} of $R$ and $S$ as the series
\[
	S\circ R\colon  A^*\to\K,\ 
	u\mapsto \sum_{\substack{v\in B^*\\ uRv}} (S,v).
\]

\begin{runningexample}
For all $\bu \in \bA^*$, we have
\[
(S \circ R, \bu)= \sum_{ \substack{\bv \in \bA^* \\
| \bv | \in \{ |\bu|-1, |\bu|, |\bu|+1 \} }} (S, \bv).
\]
By counting, for each $n \in [\![1,|\bu|+1]\!]$, the number of words $\bv$ in $\bA^*$ of length in $\{|\bu|-1, |\bu|,|\bu|+1\}$ having a coefficient $(S,\bv)$ equal to $n$, we get the closed formula
\begin{equation}
\label{Eq : ClosedFormula}	
(S \circ R, \bu)=
	\begin{cases}
	2 	& \text{if }\bu=\boldsymbol{\varepsilon}\\	
	\frac{2^{|\bu|-1}}{3} (73\cdot 4^{|\bu|-1} -7)
	& \text{otherwise}.
\end{cases}
\end{equation}
\end{runningexample}

\begin{theorem}
\label{The : CCS}
Let $R\colon A^*\to B^*$ be a synchronized relation, let $S \colon B^*\to\K$ be a $\K$-recognizable series, and suppose that for all $u\in A^*$, the language $\{v\in B^*\colon uRv\}$ is finite. Then $S\circ R$ is a $\K$-recognizable series.
\end{theorem}

In order to prove this result, we define a composition-like operation between a $2$-tape automaton (i.e., a DFA reading pairs of letters) and a $\K$-automaton. 

\begin{definition}
\label{Def : CCS}
Consider two finite alphabets $A$ and $B$ and a symbol $\$\notin A\cup B$. Let $\mathcal{A}=(Q_\mathcal{A}, i_\mathcal{A}, T_\mathcal{A}, A_\$\times B_\$, \delta_\mathcal{A})$ be a DFA and let $\mathcal{B} = (Q_\mathcal{B}, I_\mathcal{B}, T_\mathcal{B}, B_\$, E_\mathcal{B})$ be a $\K$-automaton having only one initial state, denoted by $i_\mathcal{B}$. With such automata $\mathcal{A}$ and $\mathcal{B}$, we associate a new $\K$-automaton $\mathcal{B} \circ \mathcal{A}  = (Q, I, T, A_\$, E)$ as follows. 
\begin{enumerate}
\item $Q=  (Q_\mathcal{A} \times Q_\mathcal{B})\cup\{\alpha\}$.
\item $I\colon Q \to \K$ is defined by 
\begin{itemize}
\item $I(i_\mathcal{A},i_\mathcal{B})=	I_\mathcal{B}(i_\mathcal{B})$
\item For $(q,q')\in (Q_\mathcal{A} \times Q_\mathcal{B})\setminus\{(i_\mathcal{A},i_\mathcal{B})\}$, $I(q,q')=0$
\item $I(\alpha)=1$.
\end{itemize}
\item $T\colon Q \to \K$ is defined by
\begin{itemize}
\item For $(q,q')\in T_{\mathcal{A}}\times Q_\mathcal{B}$, $T(q,q')=T_\mathcal{B}(q')$ 
\item For $(q,q')\in (Q_\mathcal{A}\setminus T_\mathcal{A} )\times Q_\mathcal{B}$, $T(q,q')=0$
\item $T(\alpha)=0$.
\end{itemize}
\item 
$E\colon Q \times A_\$ \times Q \to \K$ is defined by \begin{itemize}
\item For $(q_1,q'_1),(q_2,q'_2)\in Q_\mathcal{A}\times Q_\mathcal{B}$ and $a \in A_\$$, 
\[
	E((q_1,q'_1),a,(q_2,q_2'))
	= \sum\limits_{\substack{b\in B_\$ \\ \delta_{\mathcal{A}}(q_1, \couple{a}{b})= q_2}} 			
	E_\mathcal{B}(q_1',b,q_2')
\]   
\item For $a \in A_\$$, $E(\alpha,a,\alpha)=0$ 
\item For $(q,q')\in Q_\mathcal{A} \times Q_\mathcal{B}$ and $a \in A_\$$, $E((q,q'),a,\alpha)=0$ 
\item For $(q,q')\in Q_\mathcal{A} \times Q_\mathcal{B}$ and $a \in A_\$$, 
\[
	E(\alpha,a,(q,q'))
	=\begin{cases}
	{\displaystyle{	
	I(i_\mathcal{A},i_\mathcal{B})
	\sum_{\ell\ge 1} \sum_{c \in C_{q,q',a,\ell}}} E(c)}
	&\text{ if } (q,q') \text{ is co-accessible}\\
	0 &\text{ else}
	\end{cases} 
\]
where $C_{q,q',a,\ell}$ denotes the set of non-zero  weight paths from $(i_\mathcal{A},i_\mathcal{B})$ to $(q,q')$ labeled by $\$^\ell a$. In the case where there exist a co-accessible state $(q,q')\in Q_\mathcal{A} \times Q_\mathcal{B}$, a letter $a\in A_\$$ and infinitely many $\ell\ge 1$ such that $C_{q,q',a,\ell}$ is nonempty, we take the convention that the $\K$-automaton $\mathcal{B} \circ \mathcal{A}$ is not defined. 
\end{itemize}
\end{enumerate}
\end{definition}

We first prove three technical lemmas to get a better understanding of paths in $\mathcal{B}\circ \mathcal{A}$. For these results, we consider a DFA $\mathcal{A}$ and a $\K$-automaton $\mathcal{B}$ such that the $\K$-automaton $\mathcal{B} \circ \mathcal{A}$ is well defined. For all $w\in (A_{\$})^*$ and all paths $c=c_{\mathcal{B}\circ \mathcal{A}}((q_0,q_0')\cdots (q_{|w|},q_{|w|}'),w) $ such that $(q_0,q_0')=(i_\mathcal{A},i_\mathcal{B})$, we define the set 
\[
	C_{\mathcal{B},c}=
	\{c_\mathcal{B}(q_0'\cdots q'_{|w|},w')\colon 
	w'\in (B_{\$})^*,\ 
	c_\mathcal{A}\big(q_0\cdots q_{|w|},\couple{w}{w'}\big)\in C_\mathcal{A}\}.
\]

\begin{lemma}
\label{Lem : Lem1CCS}
Let $w\in (A_{\$})^*$ and let $c\in C_{\mathcal{B}\circ \mathcal{A}}(w)$ such that $i_c=(i_\mathcal{A},i_\mathcal{B})$. Then
\[
	E(c)T(t_c)
	=\sum_{c'\in C_{\mathcal{B},c}} 
	E_\mathcal{B}(c')T_\mathcal{B}(t_{c'}).
\]
\end{lemma}

\begin{proof}
Write $c=c_{\mathcal{B}\circ\mathcal{A}}\big((q_0,q'_0)\cdots(q_{|w|},q'_{|w|}),w\big)$. Since $c$ belongs to $C_{\mathcal{B}\circ \mathcal{A}}(w)$, it is a final path of $\mathcal{B}\circ \mathcal{A}$, and hence $T(q_{|w|},q_{|w|}')=T_\mathcal{B}(q_{|w|}')$. Then by definition of the $\K$-automaton $\mathcal{B}\circ\mathcal{A}$, we obtain
\begin{align*}
\allowdisplaybreaks
	E(c) T(t_c)
	&=\Bigg(\prod_{j=1}^{|w|} 
	E\big((q_{j-1},q'_{j-1}),w[j],(q_j,q'_j)\big)
	\Bigg) T(q_{|w|},q_{|w|}') \\
	&=\left(\prod_{j=1}^{|w|} 
	\sum_{\substack{b\in B_\$ \\ \delta_{\mathcal{A}}\left(q_{j-1}, \couple{w[j]}{b}\right)= q_{j}}}  
	E_\mathcal{B}(q'_{j-1},b,q'_j) 
	\right) T_\mathcal{B}(q_{|w|}')\\
	&=\sum_{\substack{w'\in (B_\$)^*\\ 		
	c_\mathcal{A}\big(q_0 \cdots q_{|w|}, \couple{w}{w'}\big)\in C_\mathcal{A}}} 
	\Bigg(\prod_{j=1}^{|w|} 
	E_\mathcal{B}(q'_{j-1},w'[j],q'_j) 
	\Bigg) 
	T_\mathcal{B}(q'_{|w|})\\
	&= \sum_{c'\in C_{\mathcal{B},c}} 
	E_\mathcal{B}(c')T_\mathcal{B}(t_{c'}).
\end{align*}
\end{proof}

\begin{lemma}
\label{Lem : Lem2CCS}
Let $w\in (A_\$)^*$. Then $\{C_{\mathcal{B},c}\colon c\in C_{\mathcal{B}\circ \mathcal{A}}(w),\ i_c=(i_\mathcal{A},i_\mathcal{B})\}$ is a partition of the set 
\[
	\bigcup_{\substack{w'\in (B_{\$})^*\\ \couple{w}{w'}\in L_\mathcal{A}}} C_\mathcal{B}(w').
\]
\end{lemma}

\begin{proof}
First, we show that the sets $C_{\mathcal{B},c}$ are pairwise disjoint. Consider distinct paths $c_1=c_{\mathcal{B}\circ\mathcal{A}}\big((q_{1,0},q'_{1,0})\cdots(q_{1,|w|},q'_{1,|w|}),w\big)$ and $c_2=c_{\mathcal{B}\circ\mathcal{A}}\big((q_{2,0},q'_{2,0})\cdots(q_{2,|w|},q'_{2,|w|}),w\big)$ of $C_{\mathcal{B}\circ \mathcal{A}}(w)$ such that $(q_{1,0},q'_{1,0})=(q_{2,0},q'_{2,0})=(i_\mathcal{A},i_\mathcal{B})$. Proceed by contradiction and suppose that $C_{\mathcal{B},c_1}\cap C_{\mathcal{B},c_2}\ne\emptyset$. This means that there exists some word $w'\in (B_{\$})^*$ such that on the one hand, $c_\mathcal{B}(q'_{1,0}\cdots q'_{1,|w|},w')=c_\mathcal{B}(q'_{2,0}\cdots q'_{2,|w|},w')$ and on the other hand, both $c_\mathcal{A}\big(q_{1,0}\cdots q_{1,|w|},\couple{w}{w'}\big)$ and $c_\mathcal{A}\big(q_{2,0}\cdots q_{2,|w|},\couple{w}{w'}\big)$ are accepting paths in $\mathcal{A}$. The first condition implies that $q'_{1,0}=q'_{2,0},\ \ldots,\ q'_{1,|w|}=q'_{2,|w|}$. Since $\mathcal{A}$ is a DFA and $q_{1,0}=q_{2,0}=i_\mathcal{A}$, the second condition implies that $q_{1,1}=q_{2,1},\ \ldots,\ q_{1,|w|}=q_{2,|w|}$. But then $c_1=c_2$, a contradiction.

Second, we show that 
\[
	\bigcup_{\substack{c\in C_{\mathcal{B}\circ \mathcal{A}}(w)\\i_c=(i_\mathcal{A},i_\mathcal{B})}}C_{\mathcal{B},c}				
	=\bigcup_{\substack{w'\in (B_{\$})^*\\ \couple{w}{w'}\in L_\mathcal{A}}} C_\mathcal{B}(w').
\]
Let $c\in C_{\mathcal{B}\circ \mathcal{A}}(w)$ such that $i_c=(i_\mathcal{A},i_\mathcal{B})$ and let $c'\in C_{\mathcal{B},c}$. Let $w'$ be the label of $c'$. By definition of $C_{\mathcal{B},c}$, the word $\couple{w}{w'}$ is accepted by $\mathcal{A}$ and $c'\in C_\mathcal{B}(w')$. Conversely, let $w'\in (B_{\$})^*$ such that $\couple{w}{w'}$ is accepted by $\mathcal{A}$ and let $c'\in C_\mathcal{B}(w')$. Consider the path $c=c_{\mathcal{B}\circ\mathcal{A}}\big((q_0,q'_0)\cdots(q_{|w|},q'_{|w|}),w\big)$ where $q_0,\ldots,q_{|w|}$ are the states visited along the (unique) accepting path labeled by $\couple{w}{w'}$ in $\mathcal{A}$ and $q'_0,\ldots,q'_{|w|}$ are the states of $\mathcal{B}$ visited along the path $c'$. Then $c'\in C_{\mathcal{B},c}$. Moreover, $(q_0,q'_0)=(i_\mathcal{A},i_\mathcal{B})$ and $(q_{|w|},q'_{|w|})$ is a final state of $\mathcal{B}\circ\mathcal{A}$, hence $c\in C_{\mathcal{B}\circ \mathcal{A}}(w)$.
\end{proof}

\begin{lemma}
\label{Lem : Lem3CCS}
Let $w\in (A_{\$})^*$ and let $S$ be the series recognized by $\mathcal{B}$. We have 
\[
	I_\mathcal{B}(i_\mathcal{B})\sum_{\substack{c\in C_{\mathcal{B}\circ \mathcal{A}}(w)\\i_c=(i_\mathcal{A},i_\mathcal{B})}} E(c)T(t_c)
	=\sum_{\substack{w'\in (B_{\$})^*\\ \couple{w}{w'}\in L_\mathcal{A}}} (S,w').
\]
\end{lemma}

\begin{proof}
By first using Lemma~\ref{Lem : Lem1CCS} and then Lemma~\ref{Lem : Lem2CCS},  we have 
\begin{align*}
	I_\mathcal{B}(i_\mathcal{B})\sum_{\substack{c\in C_{\mathcal{B}\circ \mathcal{A}}(w)\\i_c=(i_\mathcal{A},i_\mathcal{B})}}E(c)T(t_c)
	&=I_\mathcal{B}(i_\mathcal{B})\sum_{\substack{c\in C_{\mathcal{B}\circ \mathcal{A}}(w)\\i_c=(i_\mathcal{A},i_\mathcal{B})}}\sum_{c'\in C_{\mathcal{B},c}} 
	E_\mathcal{B}(c')T_\mathcal{B}(t_{c'})\\
	&= \sum_{\substack{w'\in (B_{\$})^*\\ \couple{w}{w'}\in L_\mathcal{A}}} \sum_{c'\in C_\mathcal{B}(w')} I_\mathcal{B}(i_\mathcal{B})
	E_\mathcal{B}(c')T_\mathcal{B}(t_{c'})\\
	&=\sum_{\substack{w'\in (B_{\$})^*\\ \couple{w}{w'}\in L_\mathcal{A}}} (S,w'). 
\end{align*}
\end{proof}

We now prove Theorem~\ref{The : CCS}.

\begin{proof}[Proof of Theorem~\ref{The : CCS}]
Let $\mathcal{A}=(Q_\mathcal{A}, i_\mathcal{A}, T_\mathcal{A}, (A_\$\times B_\$)\setminus\{\couple{\$}{\$}\}, \delta_\mathcal{A})$ be a DFA recognizing $\mathcal{G}_R^\$$ such that the initial state $i_\mathcal{A}$ has no incoming transition. Next, consider a $\K$-automaton recognizing the series $S$ having only one initial state, with no incoming transition. We modify the latter automaton to read words over $B_\$$ by adding a loop on the unique initial state of label $\$$ and of weight $1$ and by setting the weight of all other transitions labeled by $\$$ to $0$. We obtain a new $\K$-automaton that we denote by $\mathcal{B}=(Q_\mathcal{B}, I_\mathcal{B}, T_\mathcal{B}, B_\$, E_\mathcal{B})$.  The unique initial state of $\mathcal{B}$ is denoted by $i_\mathcal{B}$.

We consider the $\K$-automaton $\mathcal{B} \circ \mathcal{A}$ from Definition~\ref{Def : CCS}. Let us argue that this $\K$-automaton is indeed well defined. Otherwise, there exist a co-accessible state $(q,q')\in Q_\mathcal{A} \times Q_\mathcal{B}$, a letter $a \in A_\$$ and infinitely many $\ell \ge 1$ such that the set $C_{q,q',a,\ell}$ of non-zero  weight paths from $(i_\mathcal{A},i_\mathcal{B})$ to $(q,q')$ labeled by $\$^\ell a$ is nonempty. By co-accessibility of $(q,q')$, there exists at least one  non-zero  weight final path starting in $(q,q')$. Let $u$ denote the label of such a path. For all $\ell\ge 1$ such that $C_{q,q',a,\ell}$ is nonempty, there exists $v_\ell\in B_\$^*$ such that $\couple{\$^\ell au}{v_\ell}$ is accepted by the DFA $\mathcal{A}$, and hence $v_\ell\in B^*$. But this means that the language $\{v\in B^*\colon \tau_{\$}(au)Rv\}$ is infinite, contradicting the assumption.

We let $T$ denote the series recognized by $\mathcal{B}\circ \mathcal{A}$. In order to get that $S\circ R$ is $\K$-recognizable, it suffices to show that for all $u\in A^+$, $(T,u)=(S\circ R,u)$. Therefore, we fix $u\in A^+$ and we prove that
\begin{equation}
\label{Eq : Egalité S-T}
	(T,u)=\sum_{\substack{v\in B^*\\ uRv}} (S,v).
\end{equation} 
Since the initial states of $\mathcal{B}\circ \mathcal{A}$ are $(i_\mathcal{A},i_\mathcal{B})$ and $\alpha$, we have
\[
	(T,u)
	=I_\mathcal{B}(i_\mathcal{B})\sum_{\substack{c\in \mathcal{C}_{\mathcal{B}\circ \mathcal{A}}(u)\\ i_c=(i_\mathcal{A},i_\mathcal{B})}} E(c) T(t_c)
	+ \sum_{\substack{c\in \mathcal{C}_{\mathcal{B}\circ \mathcal{A}}(u)\\ i_c=\alpha}} E(c) T(t_c).
\]
To get Equality~\eqref{Eq : Egalité S-T}, it suffices to prove
\begin{align}
\label{Eq : PathsCi}
	I_\mathcal{B}(i_\mathcal{B})\sum_{\substack{c\in \mathcal{C}_{\mathcal{B}\circ \mathcal{A}}(u)\\ i_c=(i_\mathcal{A},i_\mathcal{B})}} E(c) T(t_c)
	&=\sum_{\substack{v\in B^*\\ uRv,\ |v| \le |u|}} (S,v)\\
\label{Eq : PathsCalpha}
	\sum_{\substack{c\in \mathcal{C}_{\mathcal{B}\circ \mathcal{A}}(u)\\ i_c=\alpha}} E(c) T(t_c)
	&=\sum_{\substack{v\in B^*\\ uRv,\ |v| > |u|}} (S,v).
\end{align}

We first show Equality~\eqref{Eq : PathsCi}. We have $\{w'\in (B_\$)^*\colon \couple{u}{w'}\in L_\mathcal{A}\}=\{\$^{|u|-|v|}v\colon v\in B^*,\ u R v,\ |v|\le|u|\}$. Moreover, by definition of the $\K$-automaton $\mathcal{B}$, for all $v\in B^*$ and $d\in\N$, the weight of $\$^dv$ in $\mathcal{B}$ equals $(S,v)$. We conclude by Lemma~\ref{Lem : Lem3CCS}.

Now, let us prove Equality~\eqref{Eq : PathsCalpha}. Write $u=au_2$ with $a\in A$ and $u_2\in A^*$. Any path in $\mathcal{C}_{\mathcal{B}\circ \mathcal{A}}(u)$ starting in $\alpha$ has the form $(\alpha,a,(q,q'))c_2$ where $(q,q')\in Q_\mathcal{A}\times Q_\mathcal{B}$ and $c_2$ is a path labeled by $u_2$ from $(q,q')$ to a final state. We let $C_{q,q'}$ denote the set of those paths. Then 
\begin{align*}
	\sum_{\substack{c\in \mathcal{C}_{\mathcal{B}\circ \mathcal{A}}(u)\\ i_c=\alpha}}E(c)T(t_c)
	&=\sum_{(q,q')\in Q_\mathcal{A}\times Q_\mathcal{B}} 
	E(\alpha,a,(q,q'))
	\sum_{c_2\in C_{q,q'}}E(c_2)T(t_{c_2})\\
	&=\sum_{(q,q')\in Q_\mathcal{A}\times Q_\mathcal{B}} 
	\left(
	I(i_\mathcal{A}, i_\mathcal{B})
	\sum_{\ell\ge 1}
	\sum_{c_1\in C_{q,q',a,\ell}} E(c_1)
	\right)
	\sum_{c_2\in C_{q,q'}}E(c_2)T(t_{c_2})\\
	&=	\sum_{\ell\ge 1}	
	I(i_\mathcal{A}, i_\mathcal{B}) 		
	\sum_{\substack{(q,q')\in Q_\mathcal{A}\times Q_\mathcal{B}\\ 
	c_1\in C_{q,q',a,\ell}\\
	c_2\in C_{q,q'}}}
	E(c_1c_2)T(t_{c_2}) \\
	&= \sum_{\ell\ge 1}
	I_\mathcal{B}(i_\mathcal{B}) 		
	\sum_{\substack{c_3\in C_{\mathcal{B}\circ\mathcal{A}}(\$^\ell u)\\ i_{c_3}=(i_\mathcal{A}, i_\mathcal{B})}} 
	E(c_3)T(t_{c_3}).
\end{align*}
For all $\ell \ge 1$, we have $\{w'\in (B_\$)^*\colon \couple{\$^\ell u}{w'}\in L_\mathcal{A}\}=\{v\in B^*\colon uRv,\ |v|=\ell+|u|\}$. 
By definition of the $\K$-automaton $\mathcal{B}$, the weight of a word $v\in B^*$ in $\mathcal{B}$ is equal to $(S,v)$. Hence, by Lemma~\ref{Lem : Lem3CCS}, we obtain
\[
	\sum_{\substack{c\in \mathcal{C}_{\mathcal{B}\circ \mathcal{A}}(u)\\ i_c=\alpha}}E(c)T(t_c)
	=\sum_{\ell\ge 1}\sum_{\substack{v\in B^*\\ uRv,\ |v|=\ell+|u|}} (S,v)
	=\sum_{\substack{v\in B^*\\ uRv,\ |v|>|u|}} 		(S,v).
\]
This concludes the proof.
\end{proof}

\begin{remark}
\label{Rem : CCS}
In the previous proof, we may modify the automaton $\mathcal{B}\circ \mathcal{A}$ so that Equality~\eqref{Eq : Egalité S-T} also holds for the empty word. To do so, we change the final weight of $(i_\mathcal{A},i_\mathcal{B})$ in $\mathcal{B}\circ \mathcal{A}$ from $T_\mathcal{B}(i_\mathcal{B})$ to $\frac{1}{I_\mathcal{B}(i_\mathcal{B})}\Big(\sum_{\substack{v\in B^* \\ \varepsilon R v}} (S,v)\Big)$. Therefore, the equality $(T,\varepsilon)=(S\circ R,\varepsilon)$ is satisfied by construction. Note that this modification only affects the weight of the empty word since the state $(i_\mathcal{A},i_\mathcal{B})$ has no incoming transition: considering a word $u\in A^+$, any path of label $u$ never ends in $(i_\mathcal{A},i_\mathcal{B})$.
\end{remark}


\begin{runningexample}
The $\N$-automaton of Figure~\ref{Fig : WeightedAutomaton} can be modified in order to have a unique initial state $i$ with no incoming transition, while keeping the same recognized series $S$. Moreover, following the construction of the proof of Theorem~\ref{The : CCS}, we add a loop on $i$ of label $\$$ and of weight $1$. This new $\N$-automaton is depicted in Figure~\ref{Fig : ModifiedWeightedAutomaton}.
\begin{figure}[htb]
\begin{center}
\begin{tikzpicture}
\tikzstyle{every node}=[shape=circle, fill=none, draw=black,minimum size=20pt, inner sep=2pt]
\node(i) at (-3.5,0) {$i$};
\node(1) at (0,0) {$T$};
\node(2) at (3.5,0) {$S$};
\tikzstyle{every node}=[shape=circle, fill=none, draw=black,
minimum size=15pt, inner sep=2pt]
\tikzstyle{every path}=[color=black, line width=0.5 pt]
\tikzstyle{every node}=[]
\draw [-Latex] (2) to node [above] {$1$} (4.5,0); 
\draw [-Latex] (-4.5,0) to node [above] {$1$} (i) ; 
\draw [-Latex] (i) to [loop above] node {$\$|1$} (i);
\draw [-Latex] (2) to [loop above] node {$\couple{a}{a}|1,\couple{b}{b}| 1$} (2);
\draw [-Latex] (1) to node [above] {$\couple{a}{a}|1,\couple{b}{b}| 1$} (2);
\draw [-Latex] (1) to [loop above] node {$\ba|1, \ba\in\bA$} (1);
\draw [-Latex] (i) to [bend right=25] node [below] {$\couple{a}{a}|1,\couple{b}{b}| 1$} (2);
\draw [-Latex] (i) to node [above] {$\ba|1, \ba\in\bA$} (1);
\end{tikzpicture}
\end{center}
\caption{The modification of the $\N$-automaton of Figure~\ref{Fig : WeightedAutomaton} described in the proof of Theorem~\ref{The : CCS}.}
\label{Fig : ModifiedWeightedAutomaton}
\end{figure}
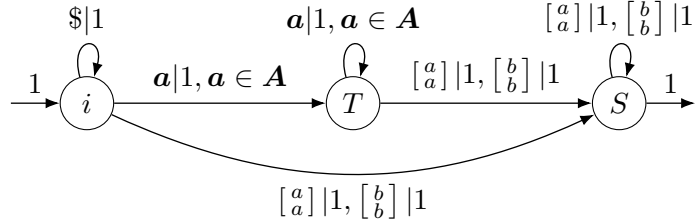
Let $\mathcal{A}$ and $\mathcal{B}$ denote the DFA of Figure~\ref{Fig : RelationAutomaton} and the $\N$-automation of Figure~\ref{Fig : ModifiedWeightedAutomaton} respectively. The accessible part of the $\N$-automaton $\mathcal{B}\circ \mathcal{A}$, where the final weight of the state $(1,i)$ has been modified as in Remark~\ref{Rem : CCS}, is depicted in Figure~\ref{Fig : CCSRelation}. For the sake of conciseness, every transition labeled by $\ba$ corresponds to $8$ labels, one for each letter $\ba\in \bA$. 
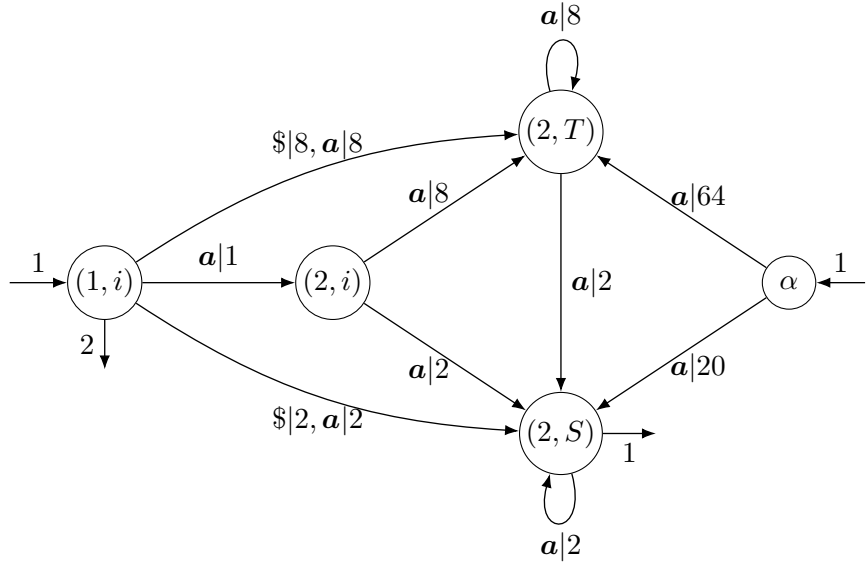
\begin{figure}[htb]
\begin{center}
\begin{tikzpicture}
\tikzstyle{every node}=[shape=circle, fill=none, draw=black,minimum size=20pt,inner sep=1pt]
\node(1i) at (0,0) {$(1,i)$};
\node(2i) at (3,0) {$(2,i)$};
\node(2T) at (6,2) {$(2,T)$};
\node(2S) at (6,-2) {$(2,S)$};
\node(alpha) at (9,0) {$\alpha$};
\tikzstyle{every node}=[shape=circle, fill=none, draw=black,
minimum size=15pt, inner sep=2pt]
\tikzstyle{every path}=[color=black, line width=0.5 pt]
\tikzstyle{every node}=[] 
\draw [-Latex] (-1.25,0) to node [above] {$1$} (1i); 
\draw [-Latex] (10,0) to node [above] {$1$} (alpha); 
\draw [-Latex] (2S) to node [below] {$1$} (7.25,-2);
\draw [-Latex] (1i) to node [left] {$2$} (0,-1.15);

\draw [-Latex] (1i) to node [above] {$\ba|1$} (2i);
\draw [-Latex] (2i) to [pos=0.4] node [above] {$\ba|8$} (2T);
\draw [-Latex] (2i) to [pos=0.4] node [below] {$\ba|2$} (2S);
\draw [-Latex] (1i) to [bend left=15] node [above] {$\$|8, \ba|8 $} (2T);
\draw [-Latex] (1i) to [bend right=15] node [below] {$\$|2, \ba|2 $} (2S);
\draw [-Latex] (2T) to node [right] {$\ba|2$} (2S);
\draw [-Latex] (alpha) to [pos=0.4] node [above] {$\ba|64$} (2T);
\draw [-Latex] (alpha) to [pos=0.4] node [below] {$ \ba|20$} (2S);
\draw [-Latex] (2T) to [loop above] node {$\ba|8$} (2T);
\draw [-Latex] (2S) to [loop below] node {$\ba|2$} (2S);
\end{tikzpicture}
\end{center}
\caption{The $\N$-automaton $\mathcal{B}\circ\mathcal{A}$. Removing the transitions of label $\$$, we obtain an $\N$-automaton recognizing the series $S \circ R$.}
\label{Fig : CCSRelation}
\end{figure}
The projection of this $\N$-automaton onto the alphabet $\bA$ (i.e., where all transitions of label $\$$ have been removed) recognizes the series $S\circ R$. It can be checked that the weights of the words over $\bA$ of length $1$, $2$ and $3$ are equal to $22$, $190$ and $1548$ respectively. Those values indeed correspond to those obtained thanks to the closed formula~\eqref{Eq : ClosedFormula}.
\end{runningexample}

\begin{remark}
In a previous work~\cite{CharlierCisterninoStipulanti}, we used a composition of a DFA and a $\K$-automaton similar to that of Definition~\ref{Def : CCS}. The two compositions differ for two reasons. First, since we were concerned with Pisot numeration systems $\boldsymbol{U}$, we used the letter $0$ as our padding symbol as is usual in this context. This involves some technicalities since the padding symbol may belong to the numeration alphabets. Second, and more importantly, in the present work we have dealt with an arbitrary relation $R\colon A^*\to B^*$. From that point of view, the situation of \cite{CharlierCisterninoStipulanti} is simpler because the DFA involved in the composition accepts the (padded) graph of a \emph{function}. More precisely, this DFA is the normalizer $\mathcal{N}_{\boldsymbol{U},\bA}$, which recognizes the graph of the normalization function $\nu_{\boldsymbol{U},\bA}\colon \bA^*\to \bA_{\boldsymbol{U}}^*$ mapping a word $\bw$ over $\bA$ to the canonical $\boldsymbol{U}$-representation of $\val_{\boldsymbol{U}}(\bw)$, where $\bA$ is an arbitrary alphabet included in $\Z^d$ and $\bA_{\boldsymbol{U}}$ is the canonical alphabet of the numeration system $\boldsymbol{U}$. 
\end{remark}

\subsection{Composing $(\bS,\bS')$-synchronized and $( \bS,\K)$-regular sequences}
\label{Sec : CompRegSynch}

We are now ready to prove the announced result of this section. 

\begin{theorem}
\label{The : CompSynchReg}
If $f\colon\N^d\to\N^{d'}$ is an $(\bS,\bS')$-synchronized sequence and $g\colon\N^{d'}\to\K$ is an $(\bS',\K)$-regular sequence, then the sequence $g\circ f \colon\N^d\to\K$ is $(\bS,\K)$-regular.
\end{theorem}

\begin{proof}
By Proposition~\ref{Pro : SynchSeqGraph}, the relation $R_{f,\bS,\bS'}$ is synchronized. We have $S_g\circ R_{f,\bS,\bS'}=S_{g\circ f}$. Then by Theorem~\ref{The : CCS},  $S_{g\circ f}$ is $\K$-recognizable, i.e., $g\circ f$ is $(\bS,\K)$-regular. 
\end{proof}

\begin{corollary}
If $f\colon\N^d\to\K$ is a $(\bS,\K)$-regular sequence, then for all $\bk\in \N^d$, the sequence $\N^d\to\K,\ \bn\mapsto f(\bn+\bk)$ is $(\bS,\K)$-regular.
\end{corollary}

\begin{proof}
This is a consequence of Proposition~\ref{Pro : +k} and Theorem~\ref{The : CompSynchReg}.
\end{proof}

\section{Acknowledgment}
Célia Cisternino is supported by the FNRS Research Fellow grant 1.A.564.19F. Manon Stipulanti is supported by the FNRS Research grant 1.B.397.20.

\bibliographystyle{plain}
\bibliography{RegularSequences.bib}
\end{document}